\documentclass[a4paper, 11pt, twoside, leqno]{amsart}

\usepackage[T1]{fontenc}
\usepackage[utf8x]{inputenc}
\usepackage[english]{babel}
\usepackage{amsmath, amsthm, amssymb, stmaryrd}
\usepackage{mathrsfs, esint, bbm}   
\usepackage{fancyhdr}               
\usepackage{fixltx2e}               
\usepackage{enumitem}               
\usepackage[all,cmtip]{xy}          
\usepackage{graphicx,xcolor}
\usepackage[textwidth=16cm, centering]{geometry}

\newtheorem{theorem}{Theorem}            
\newtheorem{corollary}[theorem]{Corollary}
\newtheorem{lemma}[theorem]{Lemma}
\newtheorem{prop}[theorem]{Proposition}

\newtheorem{mainthm}{Theorem}          
\newenvironment{maintheorem}[1]
  {\renewcommand\themainthm{#1}\mainthm}
  {\endmainthm}
  
\theoremstyle{definition}              

\theoremstyle{remark}                  
\newtheorem{step}{Step}
\newtheorem{case}{Case}
\newtheorem{remark}{Remark}

\newcommand{\N}{\mathbb{N}}                 
\newcommand{\Z}{\mathbb{Z}}                 
\newcommand{\R}{\mathbb{R}}                 
\newcommand{\Rt}{\R^3}                          
\newcommand{\C}{\mathbb{C}}                 
\renewcommand{\H}{\mathscr{H}}              
\renewcommand{\S}{\mathbb{S}}               

\newcommand{\Wude}{W^{1,2}_{\varepsilon}}    

\newcommand{\n}{\mathbf{n}}                 

\renewcommand{\u}{\mathbf{u}}
\newcommand{\w}{\mathbf{w}}
\newcommand{\g}{\mathbf{g}}
\newcommand{\nnu}{{\boldsymbol{\nu}}}

\newcommand{\ttau}{{\boldsymbol{\tau}}}
\newcommand{\ggamma}{{\boldsymbol{\gamma}}}

\newcommand{\A}{\mathbf{A}}
\renewcommand{\SS}{\mathcal{S}}

\newcommand{\e}{\mathbf{e}}
\newcommand{\W}{\mathbb{W}}
\newcommand{\T}{\mathrm{T}}
\renewcommand{\v}{\mathbf{v}}               
\newcommand{\h}{\mathbf{h}}
\renewcommand{\d}{\mathrm{d}}               
\newcommand{\D}{\mathrm{D}}

\newcommand{\KK}{\mathscr{K}}
\newcommand{\WW}{\mathbb{W}}
\newcommand{\TT}{\mathcal{T}}
\newcommand{\ve}{\v_\varepsilon} 
\newcommand{\we}{\w_\varepsilon} 
\newcommand{\epsi}{\varepsilon}
\newcommand{\eps}{\epsi}
\newcommand{\nablas}{\nabla_{\mathrm{s}}}
\newcommand{\nablae}{\nabla_{\varepsilon}}
\newcommand{\argmin}{\hbox{argmin}}

\renewcommand{\j}{\jmath}
\DeclareMathOperator{\fflat}{flat}                                  
\DeclareMathOperator{\rad}{rad}
\DeclareMathOperator{\infess}{ess\,inf}                             
\DeclareMathOperator{\supess}{ess\,sup}                             
\DeclareMathOperator{\Id}{Id}                                       
\DeclareMathOperator{\dist}{dist}                                   
\DeclareMathOperator{\diam}{diam}                                   
\DeclareMathOperator{\sign}{sign}                                   
\DeclareMathOperator{\spt}{spt}                                    
\DeclareMathOperator{\Lip}{Lip}                                     
\DeclareMathOperator{\ind}{ind}                                     
\DeclareMathOperator{\curl}{\curl}                                  

\DeclareMathOperator{\Iso}{Iso}
\newcommand{\abs}[1]{\left| #1 \right|}                             
\newcommand{\norm}[1]{\left\| #1 \right\|}                          
\newcommand{\mres}{\mathbin{\vrule height 1.6ex 
depth 0pt width 0.13ex\vrule height 0.13ex depth 0pt width 1.3ex}}  

\usepackage{color}
\usepackage[normalem]{ulem} 

\newcommand{\BBB}{\color{blue}} 
\newcommand{\RRR}{\color{red}}

\numberwithin{equation}{section}
\numberwithin{definition}{section}
\numberwithin{theorem}{section}
\numberwithin{remark}{section}

\title[Defects in Nematic Shells]
{Defects in Nematic Shells:\\ a $\Gamma$-convergence discrete-to-continuum approach}
  
\author{Giacomo Canevari}
\address{Mathematical Institute,
        University of Oxford,
        Andrew Wiles Building,
        Radcliffe Observatory Quarter,
        Woodstock Road,
        \mbox{OX2 6GG} Oxford, United Kingdom.}
\email[G. Canevari]{canevari@maths.ox.ac.uk}
  
\author{Antonio Segatti}
\address{Dipartimento di Matematica ``F. Casorati'',
  Universit\`a di Pavia,
  Via Ferrata 1, \mbox{27100} Pavia, Italy.}
\email[A. Segatti]{antonio.segatti@unipv.it}
\date{}


\begin{document}

\begin{abstract}
  In this paper we rigorously investigate the emergence of defects on Nematic Shells with genus
  different from one. 
  This phenomenon is related to a non trivial interplay between the topology of the shell and
  the alignment of the director field. 
  To this end, we consider a discrete $XY$ system on the shell $M$, described by a tangent vector field 
  with unit norm sitting at the vertices of a triangulation of the shell. Defects emerge when we let
  the mesh size of the triangulation go to zero, namely in the discrete-to-continuum limit. 
  In this paper we investigate the discrete-to-continuum limit in terms of $\Gamma$-convergence
  in two different asymptotic regimes. The first scaling promotes the appearance
  of a finite number of defects whose charges are in accordance with the topology of shell $M$, via 
  the Poincar\'e-Hopf Theorem. The second scaling produces the so called Renormalized Energy that 
  governs the equilibrium of the configurations with defects. 
\end{abstract}

\maketitle

\tableofcontents


\section{Introduction}


Nematic Liquid Crystals offer many intriguing and fascinating
examples of a non trivial interplay between 
topology, geometry, partial differential equations and physics (see the recent survey \cite{ball_survey}). 
Interestingly, Liquid Crystals manifest several visual representations
of the underlying geometric constraints. For instance, the word Nematic itself
originates from the Greek word $\nu\eta\mu\alpha$ (thread) and refers to a particular
type of thread-like topological defects that these type of Liquid Crystals
exhibit.

\smallskip
In this paper, we are interested in exploring these interplays for 
{\itshape Nematic Shells}. 
A {\itshape Nematic Shell} is a rigid colloidal particle with a typical dimension 
in the micrometer scale coated with a thin film of nematic liquid crystal
whose molecular orientation is subjected to a tangential anchoring. 
The study of these structures has recently received a good deal of interest. As suggested by Nelson \cite{Nelson02},
the interest in {\itshape Nematic Shells} is related to the possibility
of using them as building blocks of mesoatoms with a controllable valence.

From a mathematical point of view, a {\itshape Nematic
Shell} is usually represented as a two dimensional 
compact surface $M$ (without boundary, for simplicity) embedded in $\mathbb{R}^3$. 
As it happens for nematic liquid crystals occupying a domain in $\mathbb{R}^2$ or in
$\mathbb{R}^3$, the basic mathematical description is given in terms 
of a unit-norm vector field named director, describing the local orientation of the rod-shaped
molecules of the crystal (\cite{virga94}). When dealing 
with nematic shells, the local orientation of the molecules described
via a unit-norm tangent vector field $\n:M \to \mathbb{R}^3$ with $\n(x)\in \T_x M$ for any $x\in M$, $\T_x M$ being the 
tangent plane at the point~$x$
(\cite{Straley71}, \cite{LubPro92}, \cite{Nelson02}, \cite{NapVer12L}, \cite{NapVer12E}, \cite{ssvPRE}, \cite{ssvM3AS}). 

\smallskip
The study of these structures offers a non trivial interplay between the geometry and
the topology of the fixed substrate and the tangential anchoring constraint. Indeed, as 
observed in~\cite{VitNel06} and \cite{BowGio2009}, the liquid crystal 
equilibrium (and all its stable configurations, in general) is the result of 
the competition between two driving principles: on the one hand the 
minimization of the ``curvature of the texture'' penalized by the elastic 
energy, and on the other the frustration due to constraints of geometrical 
and topological nature, imposed by anchoring the nematic to the surface 
of the underlying particle. 

Moreover, the interaction between the local orientation of the molecules
and the topology of the surface $M$ (and possibly of the boundary
conditions, if any) can induce the formation of 
topological defects, that is regions of rapid changes in the director
field $\n$. It is important to note that point of defects play the role
of hot spots for the formation of the mesoatoms suggested by Nelson
in \cite{Nelson02}. Thus understanding the formation and, possibly, the energetics
of defect configurations is extremely significant for applications. 
This type of problems have been already discussed
in the physics community, see e.g. \cite{VitelliNelson} and \cite{BowGio2009}
and references therein.

When dealing with smooth vector fields, 
the classical Poin\-ca\-r\'e-Hopf Theorem establishes 
a link between the existence of a continuous tangent vector field
with unit norm on a surface $M$ and the topology of the surface itself. 
To have a clue on what happens for Nematic Shells, let us consider 
the simplest form of the energy (see \cite{Straley71}, \cite{LubPro92}, \cite{Nelson02}):
\begin{equation}
\label{eq:energy_intro}
E(\v) := \frac 12\int_{M}\vert \D\v\vert^2 \d S,
\end{equation}
where $\D$ is the covariant derivative on $M$. 
It turns out that 
the rigorous analysis of nematic shells has to face with 
possible weak forms of the Poincar\'e-Hopf Theorem. 
In particular, 
introducing the ``Sobolev set''
\[
W^{1,2}_{\tan}(M;\mathbb{S}^2):= \left\{\v: M\to \mathbb{R}^3, \,\,\,\vert \v\vert =1, \,\,
\v(x)\in \T_x M\,\,\hbox{ for a.a.}\,\, x\in M, \,\,\vert \D\v\vert\in L^2(M)\right\} \! ,
\]
we have that (see \cite{ssvM3AS} and \cite{AGM-VMO})
\begin{equation}
\label{eq:functional_frame}
W^{1,2}_{\tan}(M;\mathbb{S}^2)\neq \emptyset\,\,\, \Leftrightarrow \,\,\, \chi(M) = 0,
\end{equation}
where $\chi(M)$ is the Euler Characteristic of $M$. Consequently, the emergence
of defects is exactly related to the choice of the topology of $M$ via the Euler characteristic.
More in detail (see \cite{AGM-VMO}), the precise relation between the topology of the surface and 
the topological charge of the defects is given by
 the Poincar\'e-Hopf Theorem: If the unit-norm vector field $\v$ on $M$
 has singularities of degree $d_i$ located at the point $x_i$ for $i=1,\ldots, k$, then 
\begin{equation*}
\label{eq:p-h}
\sum_{i=1}^{k}d_i = \chi(M).
\end{equation*}

\smallskip
The goal of this paper is to understand the emergence of defects for shells of genus different from one
(that is, non zero Euler Characteristic) and their energetics.
The defect generation is related to the impossibility, for shells with $\chi\neq 0$, of supporting
a tangent, unit-norm vector field with the Sobolev regularity above. Thus, a possibile strategy would be 
to relax one the above constraints, for instance the unit-norm constraint as in the Ginzburg-Landau theory 
(see, for instance, \cite{BBH}, \cite{Jerrard}, \cite{Sandier}, \cite{JerrardSoner-GL}, \cite{SandierSerfaty}).

In this paper, we choose another point of view and instead of a continuous model we rather
consider a discrete one with the molecules sitting at the vertices of a triangular mesh approximating 
the surface $M$. 
One of the advantages of this approach is that it paves the way for a computational analysis 
in terms of finite elements. 

The model we consider here is a variant of the well-known $XY$-spin model, which is widely regarded as
a prototypical example of a discrete spin system where phase transitions that are mediated by topological defects occur.
Such phase transitions were first identified by Kosterlitz and Thouless~\cite{KosterlitzThouless} 
(also based on previous work by Berezinskii~\cite{Berezinskii}), who were awarded the 2016 Nobel Prize for Physics, together with Haldane, 
in recognition of their discoveries on topological phases of matter.  $XY$-models have also attracted
attention in the mathematical community; see, for instance, \cite{AlicandroCicalese}, \cite{ADGP}, \cite{BCSarxiv}, where 
the discrete-to-continuum limit of such models, and their connection with the continuum Ginzburg-Landau theory, is explored.
The aforementioned papers are concerned with the study of ``flat'' situations, 
i.e. the model is set on a domain~$\Omega\subseteq\R^2$; the dynamics for an $XY$-model on a ``curved''
torus has been numerically explored e.g. in~\cite{S_K_T_S11}, via a Monte-Carlo approach.

In this paper, we aim to address the mathematical analysis of an $XY$-model on surfaces.
More precisely, given a closed surface~$M\subseteq\R^3$ with $\chi(M) \neq 0$,
we consider a family of triangulations $\TT_\eps$ of~$M$ 
with the vertices $i\in \TT^{0}_\eps$ lying on $M$ and with mesh size $\eps$, i.e.
$\eps = \max_{T\in \TT_\eps}\hbox{diam}(T)$ (see Section \ref{sec:discrete_model} for the details).
Any point $i\in \TT^0_{\eps}$ is occupied by a unit-norm tangent vector $\ve(i)\in\T_i M$.
Our energy functional takes the form
\begin{equation*}
XY_\eps(\ve) := \frac{1}{2} \sum_{i\neq j\in \TT^0_\eps} \kappa^{ij}_\eps \abs{\v_\eps(i) - \v_\eps(j)}^2,
\end{equation*}
where the coefficients $\kappa_{\eps}^{ij}$ are the entries of 
the stiffness matrix of~$M$, that is, the finite-element discretization of the Laplace-Beltrami operator 
(see~\eqref{hp:weakly_acute} for the definition). 
 
As we will see later on in this paper, the $XY_\eps$ energy is intimately related
to the Dirichlet energy of the piece wise interpolation of the discrete vectors $\v_\eps(i)$. 
Consequently, when $\eps\to 0$ we expect that the behavior of $XY_\eps$ is dictated by 
\eqref{eq:functional_frame}, and a uniform bound on the energy of a minimizer can hold
if and only if the Euler Characteristic is equal to zero.    
It is important to note that the discrete (tangent) vectors $\v(i)$ are identified with 
tangent vectors in~$\mathbb{R}^3$ and thus the difference appearing in the definition 
of the energy is exactly a difference between vectors in $\mathbb{R}^3$. When $\chi(M)=0$, this
fact brings important consequences in the distorsion effect in the limit energy (see Proposition
\ref{prop:gamma_smooth}) and in particular produces a macroscopic energy that is capable 
of describing also the extrinsic effects. 

When $\chi(M)\neq 0$, configurations with defects emerge when we let $\eps\to 0$, namely in the discrete-to-continuum limit. 
Thus, the very goal of this paper is to analyze this limit in terms of $\Gamma$-convergence,
in the spirit of~\cite{AlicandroCicalese}, \cite{ADGP}. 
As it is typical in $\Gamma$-convergence results, one obtains possibly different limits according 
to the chosen scaling of the energy.

Our first main result (see Theorem~\ref{th:gamma0}) exactly relates the emergence 
of defects with the topology of the shell~$M$ and is expressed in terms of the $\Gamma$-convergence (in a suitable topology)
of $\frac{XY_\eps(\cdot)}{\vert \log\eps\vert}$.
Following the flat case (see \cite{JerrardSoner-GL}, \cite{AlicandroCicalese}, \cite{ADGP}), 
we introduce the so-called vorticity measure $\widehat{\mu}_\eps(\ve)$ of a discrete field $\ve$,
which is a kind of discrete notion of the Jacobian. This quantity captures all the relevant ``topological information'' of~$\ve$.
For sequences~$(\ve)_{\varepsilon}$ that satisfy a logarithmic energy bound (e.g., minimizers),
we show that~$\widehat{\mu}_\varepsilon(\ve)$ converges, in a suitable topology, to a measure of the form
$2\pi\sum_{i=1}^{\KK} d_i \delta_{x_i} - G\d S$, where the points~$x_i\in M$ correspond to the position of the defects,
the coefficients $d_i\in \mathbb{Z}$ are the topological charges, $G$ is the Gauss curvature and~$\d S$ is the area element on~$M$.
The proof 
follows the steps of analogous results in the Ginzburg-Landau literature
(see, in particular, \cite{JerrardSoner-GL}, \cite{Jerrard} for the continuous setting and 
\cite{AlicandroCicalese}, \cite{ADGP} for the discrete~$XY$-setting).

Our second main result (see Theorem \ref{th:gamma1}) 
is on the energetics of configurations with defects.
More precisely, we investigate the $\Gamma$-convergence of 
$XY_\eps(\cdot) - \pi\KK \vert\log\eps\vert$, where $\KK$ is the total variation of the measure
$\sum_{i=1}^{\KK} d_i \delta_{x_i}$ when $\vert d_i\vert =1$ for all $i=1, \ldots, \KK$. 
What emerges in the $\Gamma$-limit is the so called 
Renormalized Energy $\W$ that has been first studied in a rigorous way for the 
Ginzburg-Landau theory by Brezis, Bethuel and H\'elein in \cite{BBH}.
The literature on this topic is very vast (see \cite{AlicandroPonsiglione} and references therein)
and includes also results for the discrete $XY$ model on the plane, see \cite{ADGP}. 
In the euclidean situation
this energy depends on the position of the singularities via the Green function of
the laplacian on $\mathbb{R}^2$.
When dealing with a curved substrate $M$,
one may expects that the curvature properties of $M$ intervene in the limit (differently from the zeroth order
$\Gamma$-convergence in which only the topological properties of $M$ come into play). 
In particular, we expect that the curvature of $M$
(more precisely the Gaussian curvature) enters in the expression of the 
Renormalized Energy and acts as a further geometric
driving principle in the location of defects. 
This intuition is indeed true, as the analysis in \cite{VitelliNelson} shows for a corrugated plane, that is
the graph of a gaussian function. For this specific surface Vitelli and Nelson in~\cite{VitelliNelson} 
show that the Renormalized Energy is given by a sum of different terms including 
the Green function of the Laplace Beltrami on the surface and contributions coming 
from the Gaussian curvature of $M$. Interestingly, Vitelli and Nelson exhibit a term in the Renormalized energy 
given by interaction between the charge of the defects and the curvature of $M$.
Instead, our Theorem \ref{th:gamma1} deals with a compact surface $M$ and rigorously prove
the $\Gamma$-convergence of  $XY_\eps(\cdot) - \pi\KK \vert\log\eps\vert$,
to a Renormalized Energy
that takes into account also how $M$ sits in the three dimensional space.
In fact, $\mathbb{W}$
is given by the sum of a purely intrinsic part and and extrinsic
part related with the shape operator of $M$.

As a first consequence of our result, we have the following
asymptotic expansion for the minima of $XY_\eps$ (thanks to classical results in $\Gamma$-convergence):
\begin{equation}
\label{eq:conv_minima_intro}
\min XY_\eps = \pi|\chi(M)|\vert\log\eps\vert+ \W(\v) 
+ \sum_{i=1}^{|\chi(M)|}\gamma(x_i) + \mathrm{o}_{\eps\to 0}(1),
\end{equation}
where $\v\in W^{1,2}_{\tan, \mathrm{loc}}(M\setminus\{x_1,  \ldots, x_{|\chi(M)|}\}; \, \S^2)$ 
is the ``continuum limit'' of the sequence of discrete minimizers and
$\gamma(x_i)$ is a positive quantity that takes into account 
 the energy located in the core of the defects~$x_i$ of~$\v$.
It is interesting to note that $\sum_{i=1}^{|\chi(M)|}\gamma(x_i)$
has memory of the discrete structure around any singularity $x_i$ in the sense that 
$\gamma(x_i)$ will depend on the (limit) triangulation around each defect $x_i$. 
This is a purely discrete phenomenon which is essentially due 
to the fact that for a curved shell the vertices of the triangulation do not 
necessarily sit on a structured lattice. Another striking difference from the planar case, 
both continuous and discrete (see \cite{BBH} and \cite{ADGP}), and
from the curved continuous case (see \cite{JerrardIgnat}),
is that the core energy $\gamma(x_i)$
depends on the singularity $x_i$.

It is important to note 
that the Renormalized energy we introduce in this paper is defined on a proper class of vector fields 
(see \eqref{eq:classV}, \eqref{eq:renormalized_ene_def} in the present paper and \cite[Eq. 4.21]{ADGP}). 
In particular, our choice is different from the classical choice of Brezis, Bethuel and H\'elein \cite{BBH} that  
rather define the Renormalized Energy as a functional of the configuration and of the charge of the singularities. 
However, the two definitions are indeed intimately linked (at least on the planar case, see \cite[Eq. 4.23]{ADGP},
and for the intrinsic part of the Rernomalized Energy, see \eqref{eq:relation_renormalized})
by a minimization procedure. 
For a general surface $M$, the form of the Renormalized Energy $\W$ 
is rather implicit and, as expected,
one should work down its expression for any given surface $M$. In analogy with 
the well established planar case and guided by the computation in \cite{VitelliNelson}
for the corrugated plane, we expect in particular that the Renormalized energy should comprise
a term related with the Green function of the Laplace Beltrami operator on $M$ whose explicit expression
heavily depends on the form of $M$.

Even if our analysis was motivated by Nematic Shells, the study of the 
interplay between the topological properties of the domain and the possible formation of 
singularities with infinite energy, is common to other models as 
the the emergence of (topological) defects is ubiquitous in nature (see \cite{kibble03} and references therein). In particular, their topological origin often 
give the system configurations exhibiting defects a universal feature.
Consequently, the issues we aim at addressing have a more general flavor and 
are independent of the particular system.  
Moreover, energy functionals such as~\eqref{eq:energy_intro}
are commonly used also to model Amphiphilic molecules exhibiting an
hexatic bond orientational order (\cite{Mutz}, \cite{BNDT04}).  
Thus, the horizon of our analysis and results is much wider than just Nematic Shells.

\medskip
During the preparation of the present paper, we we informed that R. Jerrard and R. Ignat (see \cite{JerrardIgnat})
were obtaining, independently from us, similar results (among the others, the Renormalized Energy)
for a Ginzburg-Landau type functional on a two dimensional surface $M\subseteq\mathbb{R}^3$.

\section{Differential geometry notation}
\label{sec:diff_geo}
In this section we briefly introduce the differential geometry formalism that we use in the paper. We refer to
\cite{doCarmo} and \cite{lee} for basic references on this subject.

Let~$M\subseteq\R^3$ be a smooth, compact, connected surface without boundary, oriented by the choice of a 
smooth, unit normal field~$\ggamma\colon M \to\R^3$.
We denote by~$\boldsymbol{\omega}$ the volume form induced by this choice of the orientation
and by~$\d S$ the area element on~$M$, i.e., the positive Borel measure induced by~$\boldsymbol{\omega}$.
We let~$G$ be the Gauss curvature of~$M$. By abuse of notation, we identify~$G$
with the Borel measure~$G\d S$ and, if no confusion is possible, we write~$G$ in place of~$G\d S$.
For any point $x\in M$, we let $\T_xM$ denote the tangent plane at $x$ and we denote
with $\T M$ the tangent bundle of $M$, namely the disjoint union $\bigsqcup_{x\in M}\T_x M$. We denote with $\pi: \T M\to M$ 
the smooth map that assigns to any tangent vector its application point.
A vector field $\v$ 
on an open neighbourhood of $A\subset M$ is a section of $\T M$, namely a map 
$\v:A\to \T M$ for which $\pi\circ \v$ is the identity on $M$. 
Let~$U\subseteq\R^3$ be an open tubular neighbourhood of~$M$ of thickness $h$ such that 
\[
h \le \min_{x\in M} \Big( \max(\vert \kappa_1(x)\vert, \vert \kappa_2(x)\vert)\Big)^{-1}
\]
where $\kappa_1$ and $\kappa_2$ are the two principal curvatures of $M$. For a such a tubular 
neighbourhood the nearest-point projection~$P\colon U\to M$ is well-defined and smooth.

Let $\nabla$ be the connection with respect to the standard metric of $\Rt$,
i.e., given two smooth vector fields $Y$ and $X$ in $\R^3$ (identified with its tangent space),
$\nabla_{X}Y$ is the vector field whose components are
the directional derivatives of the components of $Y$ in the direction $X$. 
We denote with $\D_{\v}\u$ the covariant derivative of $\u$ in the direction $\v$ 
($\u$ and $\v$ are smooth tangent vector fields in $M$),
with respect to the Levi Civita (or Riemannian) connection $\D$ of the metric $g$ on $M$. 
   
Now, if $\u$ and $\v$ are extended arbitrarily to smooth vector fields on $\Rt$, we have the Gauss Formula  : 
\begin{equation}
\label{eq:gauss}
\nabla_{\v} \u = \D_{\v} \u +  \langle \d\ggamma[\u],\v\rangle \ggamma.
\end{equation}
This decomposition is orthogonal, thus there holds
\begin{equation}
\label{eq:gauss_norm}
\vert\nabla \u\vert^2 = \vert\D \u\vert^2 + \vert \d\ggamma[\u]\vert^2.
\end{equation}

Beside the covariant derivative, we introduce another differential operator for vector fields on $M$,
which takes into account also the way that $\Sigma$ embeds in $\Rt$. 
Let $\u$ be a smooth vector field on $M$. We extend it smoothly to a vector field $\tilde \u$ on $\Rt$ and we 
denote its standard gradient by $\nabla \tilde \u$ on $\Rt$. 
For $x\in M$, we define
\begin{equation}
\label{eq:surface_grad_vector}
	\nablas \u(x):=\nabla \tilde \u(x)P_{M}(x),
\end{equation}
where $P_M(x):=(\Id-\ggamma \otimes \ggamma)(x)$ is the orthogonal projection on $\T_x M$.
In other words, $\nablas$ is the restriction of the standard
derivative in $\mathbb{R}^3$ to directions that are tangent to $M$. 
Note that $\nablas \u$ is well-defined, as it does not depend on the particular extension $\tilde \u$. The object just defined is a smooth mapping 
$\nablas \u\colon M \to \R^{3 \times 3}$, or equivalently $\nablas \u\colon M\to \mathcal L(\R^3,\R^3)$ (the space of linear continuous operators on $\R^3$). In general, $\nablas \u \neq \D\u = P_M(\nabla \u)$
 since the matrix product is non commutative.
Moreover, thanks to \eqref{eq:gauss} and \eqref{eq:gauss_norm} there holds
\begin{equation}
\label{eq:gauss_surface}
\vert \nablas \u\vert^2 = \vert\D\u\vert^2 + \vert \d\ggamma[\u]\vert^2.
\end{equation}
Note that, by identifying~$\u$ with a map~$\u = (\u^1, \, \u^2, \, \u^3)\colon M\to\R^3$,
the $k$-th row of the matrix representing $\nablas\u$ coincides with the Riemannian gradient of~$\u^k$.
In other words, while~$\D$ is a connection on the tangent bundle~$\T M$,
$\nablas$ arises naturally as a connection on the trivial bundle~$M\times\R^3$.

Let~$x_0\in M$ and let~$\v$ be a vector field defined in a neighbourhood of~$x_0$ which is continuous except,
possibly, at~$x_0$ and satisfies $\inf|\v| >0$.
By taking local coordinates~$\varphi\colon B_\delta\subseteq\R^2\to M$ with~$\varphi(0) = x_0$,
we can identify~$\v$ with a map $\varphi^*\v\colon B_\delta\subseteq\R^2\to\R^2$, namely
$\varphi^*\v(z) := \langle \d\varphi^{-1}(\varphi(z)), \, \v(\varphi(z))\rangle$ for~$z\in B_\delta$. We define 
the local index~$\ind(\v, \, x_0)$ as the topological degree of the map
\[
 \frac{\varphi^*\v}{\abs{\varphi^*\v}}_{|\partial B_\delta} \colon \partial B_\delta\simeq\S^1 \to \S^1.
\]
It is easily checked that~$\ind(\v, \, x_0)$ does not depend on
the choice of~$\delta$ nor~$\varphi$. 
The index is well-defined even if~$\v$ is a field of Sobolev regularity~$W^{1,2}$
because the restriction of $\varphi^*\v$ to the circle~$\partial B_\delta$ is continuous
for a.e.~$\delta$, thanks to Fubini theorem and Sobolev embedding.
For further properties of the index, we refer the reader to, e.g.,~\cite{AGM-VMO}.

\section{The discrete model and main results}
\label{sec:discrete_model}
In this section we introduce the discrete setting we will use in the rest of the paper and we state our main results. 
This Section is organized as follows. 
First of all, we will introduce the discretization of the surface $M$. 
As the mathematical analysis of \eqref{eq:energy_intro} bears some analogy with 
the analysis of harmonic maps (see \cite{ssvM3AS}) the discretization of the surface
is based on the formalism developed in \cite{Bartels}. 
Then, we define the starting point of our analysis, namely 
the discrete energy \eqref{XY}. Finally, we discuss the (simpler) case of defects-free textures and then we state our Main results
on the emergence of defects (Theorem \ref{th:gamma0}) and on their energetics (Theorem \ref{th:gamma1}).
To ease the presentaion, we will briefly introduce two fundamental objects, namely 
a discrete version of the jacobian and the Renormalized Energy. Their rigorous definitions together with 
their properties will be postponed in the forthcoming Subsections \ref{ssec:dic_meas} and \ref{sect:renormalized}.

\subsection{Triangulations of a surface}
\label{sect:triangulations}

For any $\eps\in(0, \, \varepsilon_0]$, 
we let~$\TT_\varepsilon$ be a triangulation of~$M$, that is, a finite collection of non-degenerate affine triangles $T\subseteq\R^3$
with the following property: the intersection of any two triangles~$T$, $T^\prime\in\TT_\varepsilon$
is either empty or a common subsimplex of~$T$, $T^\prime$. The parameter~$\varepsilon$ is the mesh size,
namely we assume $\eps=\max_{T\in\TT_\varepsilon}\hbox{diam}(T)$.
The sets of vertices and edges of~$\TT_\varepsilon$ will be denoted by $\TT^0_\varepsilon$, $\TT^1_\varepsilon$, respectively.
We will always assume that $\TT^0_\varepsilon\subseteq M$.
We set~$\widehat{M}_\varepsilon := \cup_{T\in\TT_\varepsilon} T$, so $\widehat{M}_\varepsilon$
is the piecewise-affine approximation of~$M$ induced by~$\TT_\varepsilon$.
Given an open set $\Omega\subseteq M$ (or possibly in $\mathbb{R}^2$ when expressed in local coordinates),
we set
$\Omega_\eps:=\left\{T\in \TT_\eps: T\subset \Omega \right\}$. Moreover, we denote with 
$\partial_\eps \Omega$ the discrete boundary of $\Omega$, namely
\begin{equation}
\label{eq:discrete_bd}
\partial_\eps\Omega:= \partial \Omega_\eps \cap \TT^0_{\eps}.
\end{equation}
Given a piecewise-smooth function~$\u\colon\widehat{M}_\varepsilon\to\R^k$, we denote by
$\nablae\u$ the restriction of the derivative $\nabla\u$ to directions that
lie in the triangles of $\widehat{M}_\eps$.



We assume that the family of triangulations~$(\TT_\varepsilon)$ satisfies the following conditions.
\begin{enumerate}[label=(H\textsubscript{\arabic*}), ref=H\textsubscript{\arabic*}]
 \item \label{hp:quasiuniform} Let~$T_{\mathrm{ref}}\subseteq\R^2$ be a reference triangle of vertices~$(0, \, 0)$, $(1, \, 0)$ and~$(0, \, 1)$.
 There exists a constant~$\Lambda >0$ such that, for any $\varepsilon\in(0, \, \varepsilon_0]$ and any~$T\in\TT_\varepsilon$,
 the (unique) affine bijection~$\phi\colon T_{\mathrm{ref}}\to T$ satisfies 
 \[
  \Lip(\phi)\leq\Lambda\varepsilon, \qquad \Lip(\phi^{-1})\leq \Lambda\varepsilon^{-1}.
 \]
 Here~$\Lip(\phi)$ denotes the Lipschitz constant of~$\phi$, $\Lip(\phi) := \sup_{x\neq y} |x-y|^{-1} |\phi(x) - \phi(y)|$.
 
 \item \label{hp:weakly_acute} For any~$\varepsilon\in(0, \, \varepsilon_0]$ and any~$i$, $j\in \TT^0_\varepsilon$ with $i\neq j$, there holds
 \[
  \kappa_\varepsilon^{ij} := -\int_{\widehat{M}_\varepsilon} \nablae\widehat{\varphi}_{\varepsilon, i}
  \cdot\nablae\widehat{\varphi}_{\varepsilon, j} \,\d S \geq 0,
 \]
 where the hat function $\widehat{\varphi}_{\varepsilon, i}$
 is the unique piecewise-affine, continuous function $\widehat{M}_\varepsilon\to\R$
 such that~$\widehat{\varphi}_{\varepsilon, i}(j) = \delta_{ij}$ for any~$j\in \TT^0_\varepsilon$.
 
 \item \label{hp:bilipschitz} For any~$\varepsilon\in(0, \, \varepsilon_0]$, $\widehat{M}_\varepsilon\subseteq U$
 and the restriction of the nearest-point 
 projection~$\widehat{P}_\varepsilon := P_{|\widehat{M}_\varepsilon}\colon\widehat{M}_\varepsilon\to M$ has a Lipschitz
 inverse. Moreover, we have $\Lip(\widehat{P}_\varepsilon) + \Lip(\widehat{P}_\varepsilon^{-1})\leq\Lambda$
 for some~$\varepsilon$-independent constant~$\Lambda$.
\end{enumerate}

\begin{remark}
  \eqref{hp:quasiuniform} is equivalent to the following condition: there exists a constant~$\Lambda>0$ 
  such that, for any~$\varepsilon\in (0, \, \varepsilon_0]$ and any triangle~$T\in\TT_\varepsilon$, there holds
  \[
   \Lambda^{-1}\varepsilon\leq\diam(T)\leq\Lambda\varepsilon \qquad \textrm{and} \qquad
   \alpha_{\min}(T) \geq \Lambda^{-1},
  \]
  where $\alpha_{\min}(T)$ stands for the minimum of the angles of~$T$.
  Meshes that satisfy this condition are called \emph{quasi-uniform} in the numerical literature.
  Since the manifold~$M$ is compact and smooth, and hence has bounded curvature, $\alpha_{\min}(T)\geq\Lambda^{-1}\varepsilon$
  implies that the number of neighbours of a given vertex is uniformly bounded with respect to~$\varepsilon$.
\end{remark}
\begin{remark} 
  A sufficient condition for~\eqref{hp:weakly_acute} is the following: for any pair of triangles~$T_1$, $T_2\in\TT_\varepsilon$ 
  that share a common edge $e$, let $\alpha_i$ be the angle in $T_i$ opposite to $e$ (for~$i\in\{1, \, 2\}$). If $\alpha_1 + \alpha_2 \leq\pi$
  for every edge~$e$ as above, then~\eqref{hp:weakly_acute} holds (see e.g.~\cite[Lemma~1.4.1]{Bartels}).
  Triangular meshes that satisfy~\eqref{hp:weakly_acute} are called \emph{weakly acute}.
\end{remark}
\begin{remark} 
  If~$\TT_\varepsilon$ satisfies~\eqref{hp:weakly_acute} and if~$\widehat{\varphi}$, $\widehat{\tau}\in C(\widehat{M}_\varepsilon, \,\R)$ 
  are piecewise-affine functions on the triangles of~$\TT_\varepsilon$, then 
  \begin{equation} \label{exact_integration}
   \int_{\widehat{M}_\varepsilon} \nablae\widehat{\varphi}\cdot\nablae\widehat{\tau} \,\d S
   = \sum_{i, j\in \TT^0_\varepsilon} \kappa_\varepsilon^{ij} \left(\widehat{\varphi}(i) - \widehat{\varphi}(j)\right) 
   \left(\widehat{\tau}(i) - \widehat{\tau}(j)\right) \! .
  \end{equation}
\end{remark}
\begin{remark} 
  The condition~\eqref{hp:bilipschitz} is introduced to rule out pathological examples, and to make sure
  that~$\widehat{M}_\varepsilon$ is indeed a good approximation of~$M$. 
  It is not meant to be sharp. There are algorithmic ways to construct triangulations that are quasi-uniform, weakly acute
  and satisfy~\eqref{hp:bilipschitz}, for instance, Delaunay meshes (see e.g.~\cite{Shewchuk}, \cite{Amenta_et_al}).
\end{remark}

Beside Hypothesis \eqref{hp:quasiuniform}-\eqref{hp:bilipschitz}, for the validity 
of Theorem \ref{th:gamma1} we will need a refined control on the triangulation
$\TT_\eps$ around the singularities (see Proposition \ref{prop:core_energy}) in the limit $\eps\to 0$.
At base, we require that 
that our triangulation $\TT_\eps$  is somehow scale invariant. We express this requirement as follows.
Fix~$x\in M$ and let~$\delta>0$ be smaller than the injectivity radius of~$M$.
Using geodesic coordinates, $\varphi\colon B_\delta\subseteq\T_xM\simeq\mathbb{R}^2\to M$ 
such that~$\varphi(0) = \bar x$, we pull~$\TT_\eps$ back and define a triangulation~$\overline{\TT}_\eps$ 
on~$B_\delta\subseteq\R^2$. (The set of vertices of~$\overline{\TT}_\eps$ 
is~$\varphi^{-1}(B_\delta(\bar x)\cap\TT_\eps^0)$; three vertices in~$\overline{\TT_\eps}$ 
span a triangle in~$\overline{\TT}_\eps$ if and only if their images via~$\varphi$ do.) 
We scale the triangulation~$\overline{\TT}_\eps$ of a factor~$1/\eps$ and define a triangulation
on~$B_{\delta/\varepsilon}\subseteq\R^2$, namely
\[
 \SS_\eps := \left\{\frac{1}{\varepsilon}T \colon T\in\overline{\TT}_\eps \right\} \! .
\]
%
Given another triangulation~$\SS$ on~$\R^2$, we denote by~$\SS_{|B_{\delta/\eps}}$
the collection of triangles~$T\in\SS_\eps$ such that~$T\subseteq B_{\delta/\eps}$.
We express the distance between~$\SS_\eps$ and~$\SS_{|B_{\delta/\eps}}$ as
\[
 d(\SS_\eps, \, \SS_{|B_{\delta/\eps}}) := \inf_{\phi} \max_{i\in \SS_\eps^0} |i - \phi(i)|,
\]
where the infimum is taken over all simplicial isomorphisms~$\phi$ from~$\SS_\eps$ to~$\SS_{|B_{\delta/\eps}}$
(that is, piecewise-affine maps that preserves the combinatorial structure of the mesh;
see Section~\ref{sect:mesh_distance} for more details).
In addition to \eqref{hp:quasiuniform}, \eqref{hp:weakly_acute} and \eqref{hp:bilipschitz}, 
for the validity of Theorem \ref{th:gamma1} we assume that the following condition holds.
\begin{enumerate}[label=(H\textsubscript{\arabic*}), ref=H\textsubscript{\arabic*}, resume]
 \item \label{hp:convergence} For any~$x\in M$ there exists a triangulation~$\SS = \SS(x)$ on~$\R^2$
 such that, for any~$\delta > 0$ smaller than the injectivity radius of~$M$, there holds
 \[
  \lim_{\eps\searrow 0} \  d(\SS_{\eps}, \, \SS_{|B_{\delta/\eps}}) \, \abs{\log\eps} = 0.
 \]
\end{enumerate}
This assumption is only used in the arguments of Section~\ref{sect:core}, and in particular in the proof 
of Proposition~\ref{prop:core_energy}.
\begin{remark}
 The construction of a sequence of triangulations satisfying~\eqref{hp:convergence},
 as well as~\eqref{hp:quasiuniform}--\eqref{hp:bilipschitz}, is illustrated in Figure~\ref{fig:h4},
 in case~$M$ is a sphere centred at the origin.
 We take a cube~$Q$ inscribed in the sphere, and subdivide each face into a square grid
 with spacing comparable to~$\varepsilon$. This induces a triangulation on~$\partial Q$ by isosceles right triangles,
 whose restriction to each face satisfies~\eqref{hp:convergence}.
 We pull back this triangulation to the sphere using the Lipschitz map~$x\in \partial Q\mapsto x/|x|$. 
 A similar construction can be carried out for any closed surface~$M$, by noting that~$M$ is bilipschitz
 equivalent to the boundary of a polyhedron whose faces are finite unions of squares. 
 In the left Figure~\ref{fig:ico-uv} we present another triangulation of the sphere that satisfy our set 
 of hypothesis \eqref{hp:quasiuniform}--\eqref{hp:convergence}. 
 A triangulation that does not satisfies~\eqref{hp:convergence} is illustrated in the Figure~\ref{fig:ico-uv} on the 
 right.
\end{remark}

\begin{figure}[t] 
\centering
 \begin{minipage}{.47\textwidth}
  \includegraphics[height=.3\textheight]{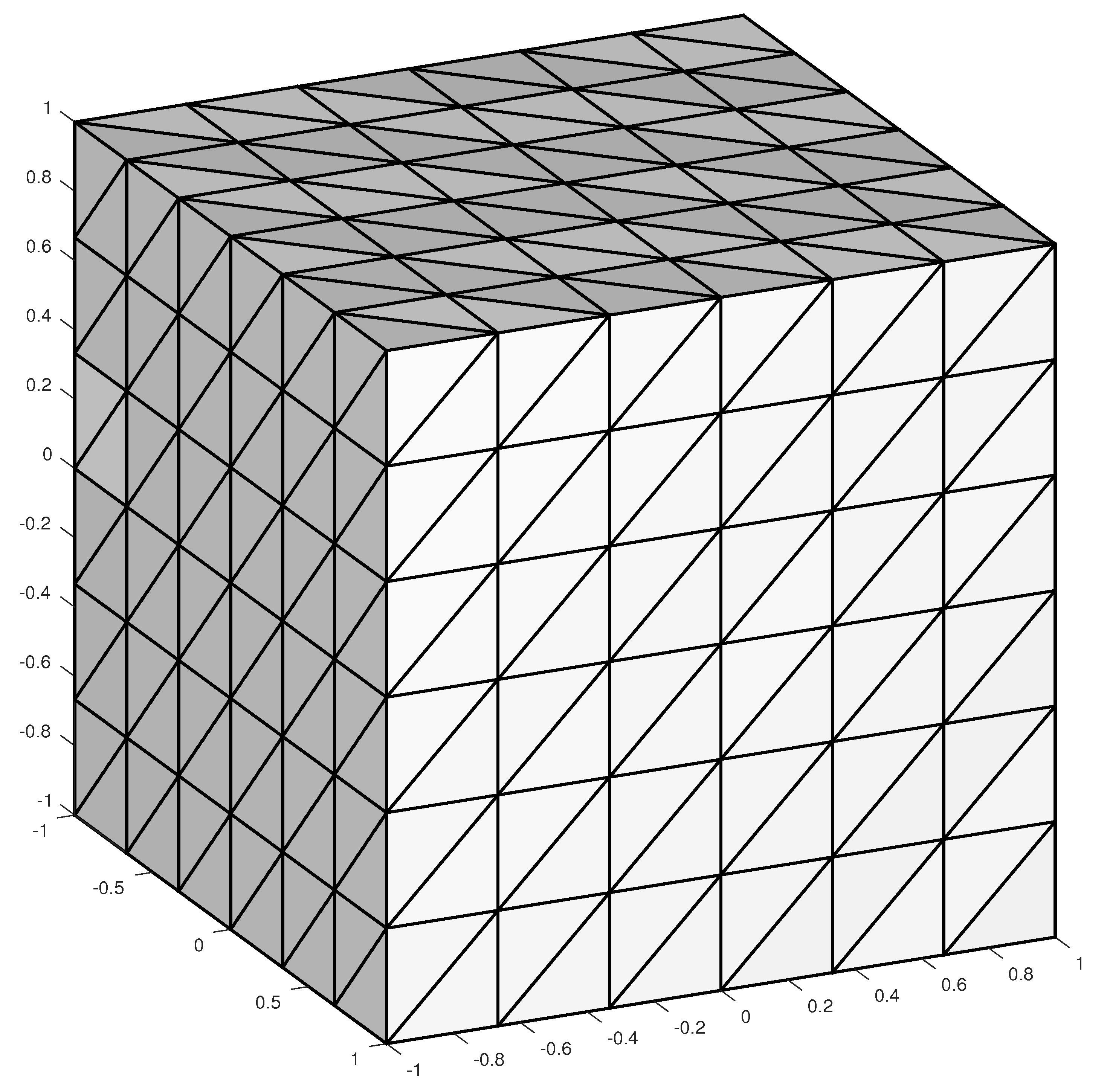}
 \end{minipage}%
 \quad
 \begin{minipage}{.47\textwidth}%
  \includegraphics[height=.3\textheight]{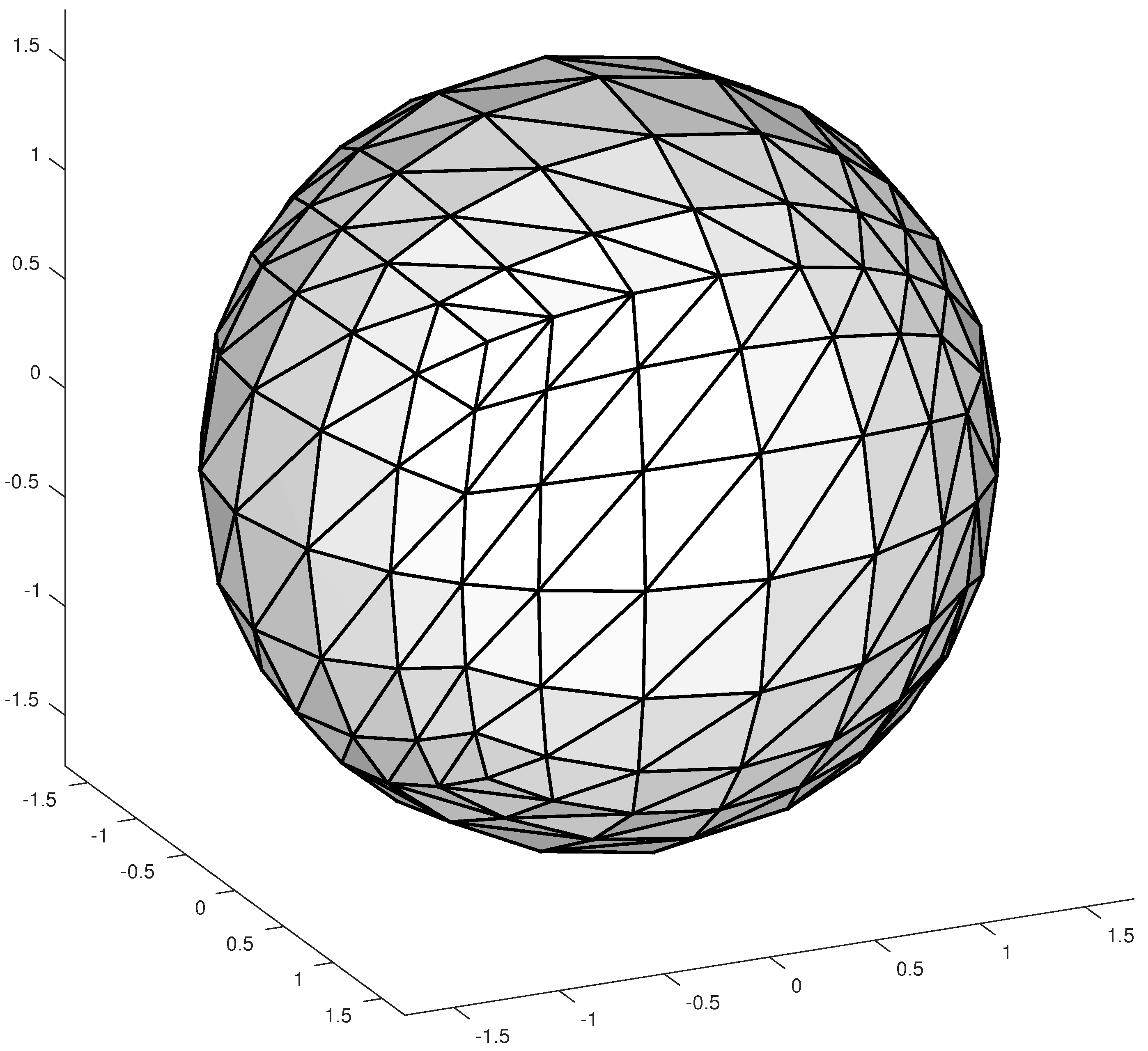}
 \end{minipage}%
        \caption{Left: a uniform triangulation by isosceles right triangles on the boundary of a cube.
        Right: the same triangulation is mapped to a sphere, by renormalizing the coordinates of each vertex.}
        \label{fig:h4}
\end{figure}

\begin{figure}[ht] 
\centering
 \begin{minipage}{.47\textwidth}
  \includegraphics[height=.28\textheight]{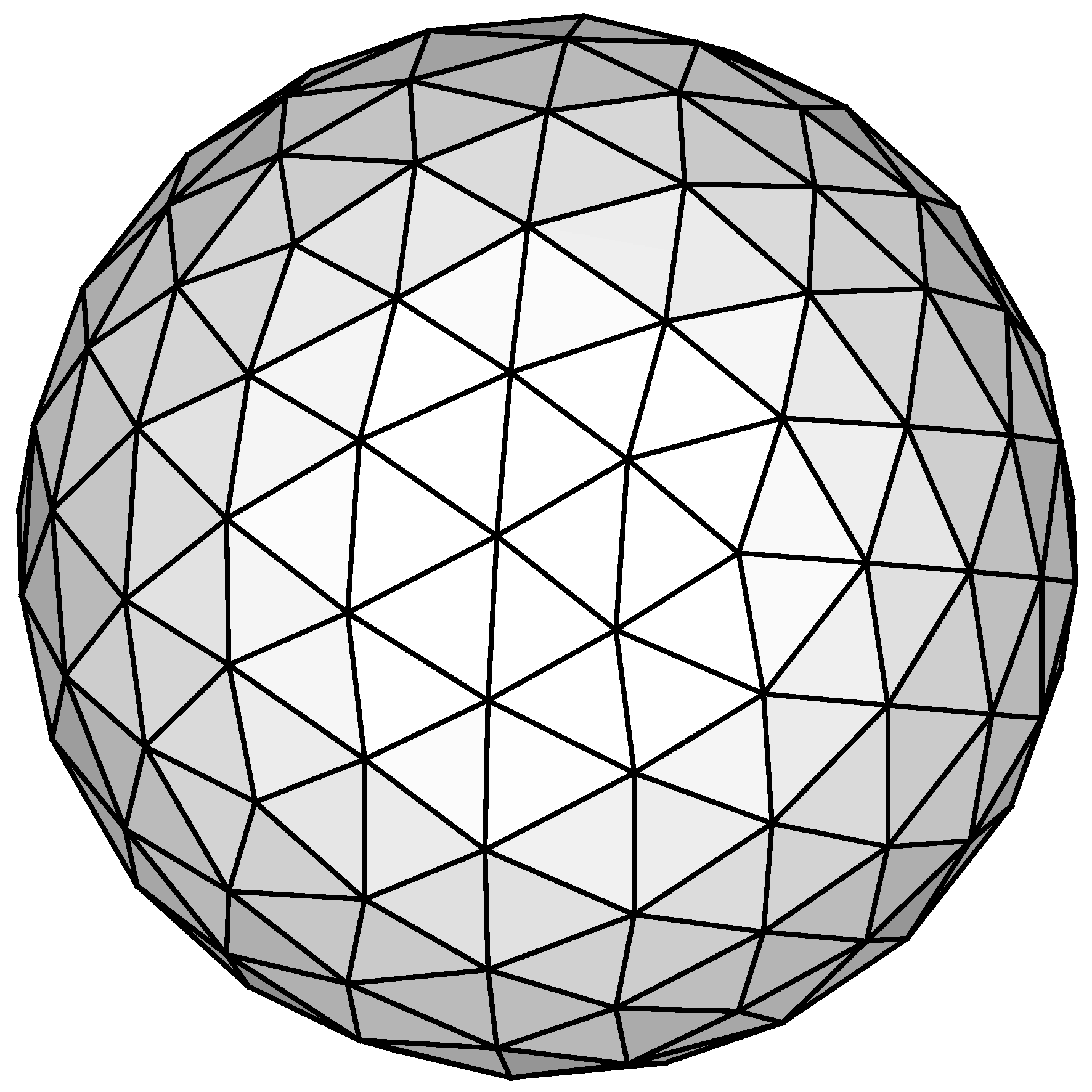}
 \end{minipage}%
 \quad
 \begin{minipage}{.47\textwidth}%
  \includegraphics[height=.28\textheight]{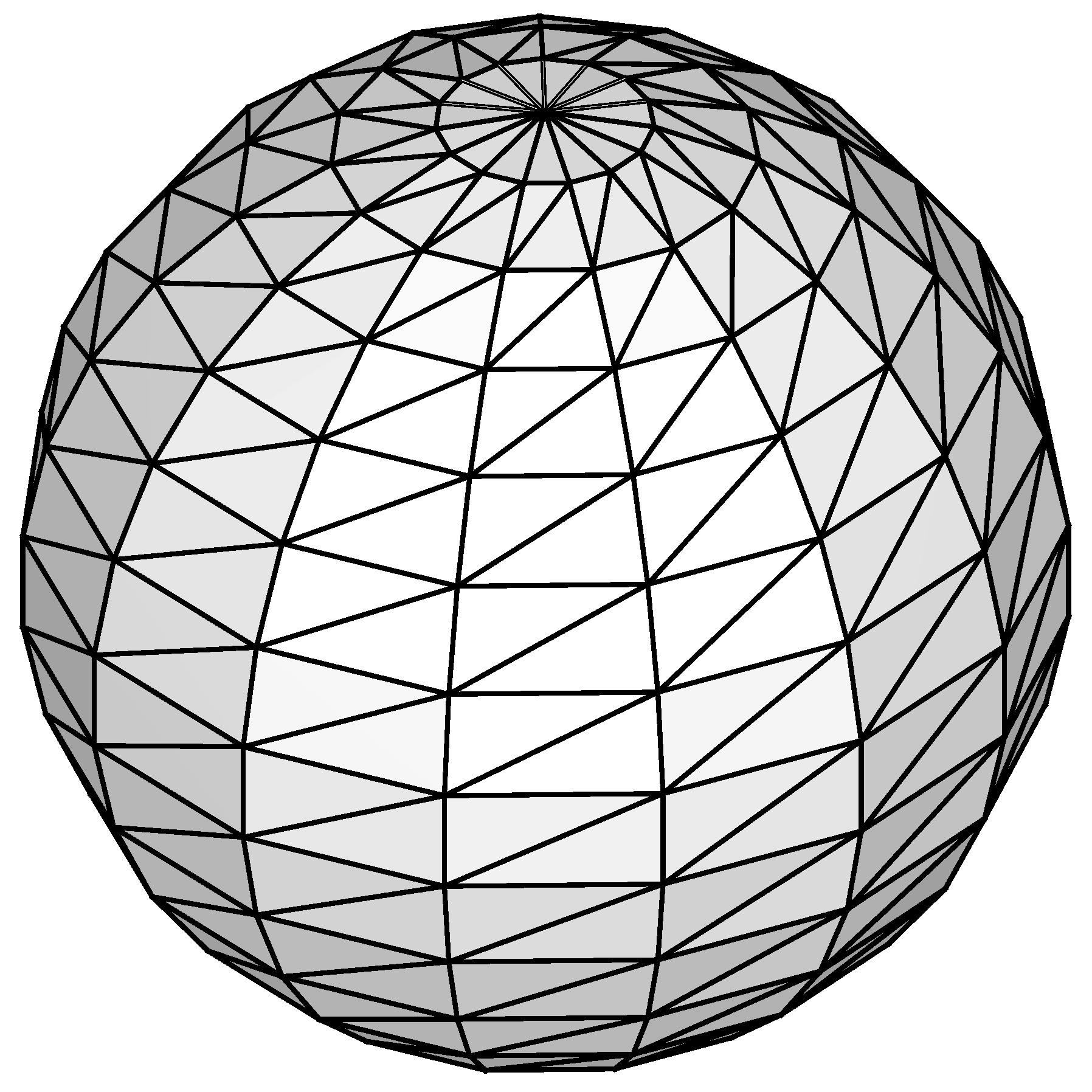}
 \end{minipage}%
        \caption{Left: an \emph{icosphere}. This triangulation is obtained by subdividing the faces 
        of a regular icosahedron into smaller triangles, then projecting the vertices to the sphere. 
        By refining this construction, one obtains a sequence of triangulations that satisfies
        \eqref{hp:quasiuniform}--\eqref{hp:convergence}. Right: an~\emph{UV-sphere}, obtained by
        mapping a uniform grid on the square via spherical coordinates. The meshes produced by this method 
        do not satisfy~\eqref{hp:convergence} nor~\eqref{hp:quasiuniform}, because the number of neighbours
        of the north and south poles blows up as the mesh is refined.}
        \label{fig:ico-uv}
\end{figure}

The main characters of our analysis will be unit-norm tangent discrete
vector fields on $M\cap \TT^{0}_{\eps}$, namely maps
\begin{equation} \label{discrete_fields}
 \v_\varepsilon\colon \TT^0_\varepsilon\to \mathbb{R}^3 \hbox{ s.t. }\vert \v\vert = 1 \,\,\hbox{ and }
 \v_\varepsilon(i)\cdot\ggamma(i) = 0 \quad \textrm{for any } i\in \TT^{0}_{\eps}.
\end{equation}
We will denote with $\T(\TT_\varepsilon; \, \S^2)$ the space of such discrete vector fields.
Given~$\v_\varepsilon\in\T(\TT_\varepsilon;\, \S^2)$, $\widehat{\v}_\varepsilon\colon\widehat{M}_\varepsilon\to\R^3$ 
denotes the piecewise-affine interpolant of~$\v_\varepsilon$. Note that $\widehat{\v}_\eps$
can be represented using the basis functions $\widehat{\varphi}_{\varepsilon, i}$ in this way:
\begin{equation}
 \label{eq:aff_interpol}
 \widehat{\v}_\eps = \sum_{j\in \TT^0_{\eps}}\ve(j) \widehat{\varphi}_{\varepsilon, j}.
\end{equation}
In a similar way, $\widehat{\ggamma}_\varepsilon$ denotes the
piecewise-affine interpolant of~$\ggamma$ restricted to~$\TT_\varepsilon^0$.

\begin{remark}
 In computational applications, it might be convenient to define the set of discrete vector fields~\eqref{discrete_fields}
 using some numerical approximation~$\ggamma_\varepsilon$ of~$\ggamma$, instead of~$\ggamma$ itself. 
 The arguments in this paper could easily be adapted to cover this case as well, provided that the approximation~$\ggamma_\varepsilon$
 satisfies an a priori bound such as
 \[
  \sup_{i\in\TT^0_\varepsilon} \abs{\ggamma_\varepsilon(i) - \ggamma(i)} \leq C\varepsilon.
 \]
\end{remark}

\subsection{The discrete energy}
\label{ssec:discrete_energy}
Given a discrete field~$\v_\varepsilon\in\T(\TT_\varepsilon, \, \S^2)$, we consider the discrete XY-energy
\begin{equation} \label{XY} \tag{XY$_\varepsilon$}
 XY_\varepsilon(\v_\varepsilon) := \frac{1}{2} \sum_{i\neq j\in \TT^0_\varepsilon} \kappa^{ij}_\varepsilon \abs{\v_\varepsilon(i) - \v_\varepsilon(j)}^2.
\end{equation}
Because the support of the hat function~$\widehat{\varphi}_{\varepsilon,i}$ only intersects the triangles that are adjacent to~$i$,
we have~$\kappa_\varepsilon^{ij}=0$ if the vertices~$i$, $j$ are not adjacent. Hence, the XY-energy is indeed defined 
by a nearest-neighbours interaction. 
Moreover, due to~\eqref{eq:aff_interpol} we have
\begin{equation} \label{affine_dirichlet}
 XY_\varepsilon(\v_\varepsilon) = \frac{1}{2} \int_{\widehat{M}_\varepsilon} |\nablae\widehat{\v}_\varepsilon|^2 \,\d S.
\end{equation}

\subsection{The defect-free case}

Before addressing the problem of generation of defects, it is important to understand what happens
for a shell $M$ with $\chi(M)=0$. 
In this case, the topology of $M$ does not forces the formation of singularities 
as the Poincar\'e-Hopf Theorem admits configurations with $d_i=0$ for any $i$. 
Moreover, the results in \cite{ssvM3AS} and \cite{AGM-VMO} guarantee that
the set 
\begin{equation*}
W^{1,2}_{\tan}(M;\mathbb{S}^2):= \left\{\v: M\to \mathbb{R}^3, \,\,\,\vert \v\vert =1, \,\,
\v(x)\in \T_x M\,\,\hbox{ for a.a.}\,\, x\in M, \,\,\vert \D\v\vert\in L^2(M)\right\} \! ,
\end{equation*}
is non empty. 
Consequently, for a compact surface without boundary and $\chi(M)=0$, the following
energy $E_{\mathrm{extr}}$ is well defined
\begin{equation}
\label{eq:extrinsic_energy}
E_{\mathrm{extr}}(\v) : = \int_{M}e(\v) \d S, \,\,\,\, e(\v) := \vert\D\v\vert^2 + \frac 12\vert \d \ggamma[\v]\vert^2.
\end{equation}
When $\v$ is of unit norm, this energy governs 
the statics of a Nematic Shell (of genus $1$)
in the one constant approximation when one takes into account also extrinsic effects
(see \cite{NapVer12L}, \cite{NapVer12E} and \cite{ssvM3AS}). 
Indeed, differently from \eqref{eq:energy_intro}, this energy includes the 
term with $\d\ggamma$ that takes into account how the surface $M$ sits 
in the tridimensional space $\mathbb{R}^3$. We refer to the papers 
\cite{NapVer12L}, \cite{NapVer12E} and \cite{ssvM3AS} for a detailed analysis
of $E_{\mathrm{extr}}$ and for a discussions about 
the differences in selecting minimizers between \eqref{eq:energy_intro} and 
\eqref{eq:extrinsic_energy}. Interestingly, the energy \eqref{eq:extrinsic_energy}
emerges as the discrete to continuum limit of the energies $XY_\eps$. 
In fact, suppose $M$ is a compact surface without boundary with $\chi(M)=0$. Let
$\TT_\varepsilon$ be a triangulation of $M$ satisfying the conditions 
\eqref{hp:quasiuniform}, \eqref{hp:weakly_acute} and~\eqref{hp:bilipschitz}. Now, 
given a smooth vector field $\v$ of unit norm we consider the discrete vector field
given by the restriction of $\v$ to the nodes of the triangulation, namely
$\ve(i):= \v(i)$ for $i\in \TT^{\eps}_{0}$. It is not difficult to realize that 
\begin{equation*}
\lim_{\eps\to 0}XY_\eps(\ve) = \lim_{\eps\to 0}
\frac{1}{2} \int_{\widehat{M}_\varepsilon} |\nablae\widehat{\v}_\varepsilon|^2 \,\d S
=\frac 12 \int_{M}\vert \nablas \v\vert^2 \d S = E_{\mathrm{extr}}(\v),
\end{equation*}
where we used the fact that the directions tangent to $\widehat{M}_\varepsilon$
uniformly converge to directions that are tangent to $M$ together with \eqref{eq:gauss_surface}.
This discrete to continuum limit can be actually made rigorous in terms
of $\Gamma$-convergence. 
More precisely, 
we have 
\begin{prop}
\label{prop:gamma_smooth}
Suppose that the assumptions~\eqref{hp:quasiuniform}, \eqref{hp:weakly_acute} and~\eqref{hp:bilipschitz} are satisfied.
 Then, $XY_\eps$ $\Gamma$-converges with respect to weak convergence of $L^2(M;\mathbb{R}^3)$ 
 to the functional
 \begin{equation*}
 E_{\mathrm{extr}}(\v):=
 \begin{cases}
 \int_{M}\vert\D\v\vert^2 + \vert \d\ggamma[\v]\vert^2 \d S, &\,\,\hbox{ if } \v\in
  W^{1,2}_{\tan}(M;\mathbb{S}^2)\\
 + \infty, &\,\,\hbox{ otherwise in } L^2(M;\mathbb{R}^3).
 \end{cases}
 \end{equation*}
 \end{prop}
The proof of this Proposition follows routine arguments in the analysis 
of discrete to continuum limits. Therefore we skip it and we refer to \cite{AlicandroCicalese},
\cite{BCSarxiv} and \cite[Theorem~III.2]{LM}.

\subsection{Configuration with defects: the zeroth order $\Gamma$-limit}

\smallskip
Towards the analysis of defects, we 
briefly introduce the important notion of discrete vorticity measure.
This measure will be a kind of discrete notion of jacobian for
the discrete vector field $\ve$ in $\T(\TT_\varepsilon; \, \S^2)$. 
As it happens for the discrete flat case and for Ginzburg Landau case, the vorticity measure
of the sequence $\ve$ will provide all the informations regarding the emergence of the defects in 
the $\eps\searrow 0$ limit.  
Even if we will precisely introduce this measure in the next Subsection \ref{ssec:dic_meas}, 
we briefly present it now for the sake of clarity.
Given a triangle ~$T\in\TT_\varepsilon$ we let~$(i_0, \, i_1, \, i_2)$ be the vertices of~$T$,
sorted in counter-clockwise order with respect to the orientation induced by~$\ggamma$,
and let~$i_3 := i_0$. For any triangle $T$, $\widehat{\mu}_\varepsilon(\ve)\mres T$ is supported on the barycenter of~$T$ and 
\begin{equation} \label{mu_hat_def}
 \begin{split}
  \widehat{\mu}_\varepsilon(\v_\varepsilon)[T] := 
  \sum_{k = 0}^2 \left(\frac{\ggamma(i_k) + \ggamma(i_{k+1})}{2}, \, \v_\varepsilon(i_k)\times\v_\varepsilon(i_{k+1})\right) \! .
 \end{split}
\end{equation}
In the limit $\eps\to 0$, the appearance of defects is related to a measure
concentrated on a finite number of points $\left\{x_1,\ldots,x_k\right\}$ in $M$. 
We will denote by~$X$ the set of measures on~$M$ of the form
\begin{equation*} 
 \mu = \sum_{i = 1}^k d_i\delta_{x_i},
\end{equation*}
where~$k\in\N$, $d_i\in\Z$ are such that~$\sum_i d_i = \chi(M)$
and~$x_i\in M$ for~$i\in\{1, \, \ldots, \, k\}$. The space~$X$ will be endowed with the topology of flat convergence, 
that is, the topology induced by the dual norm of Lipschitz functions.

Here is the precise statement of the first main result of the paper.

\begin{maintheorem}{A} \label{th:gamma0}
Suppose that the assumptions~\eqref{hp:quasiuniform}, \eqref{hp:weakly_acute} and~\eqref{hp:bilipschitz} are satisfied.
 Then, the following results hold.
 \begin{enumerate}[label=(\roman*), leftmargin=*]
  \item \emph{Compactness.} If~$(\v_\varepsilon)$ is a sequence in~$\T(\TT_\varepsilon; \, \S^2)$
  that satisfies the energy bound
  \begin{equation} \label{H} \tag{H}
   XY_\varepsilon(\v_\varepsilon) \leq \Lambda |\log\varepsilon|
  \end{equation}
  then, up to subsequences, $\hat{\mu}_\varepsilon(\v_\varepsilon) \xrightarrow{\fflat}2\pi\mu - G\d S$ for some~$\mu\in X$.
  
  \item \emph{$\Gamma$-liminf inequality.} Let~$(\v_\varepsilon)$ be a sequence in~$\T(\TT_\varepsilon; \, \S^2)$
  such that $\hat{\mu}_\varepsilon(\v_\varepsilon) \xrightarrow{\fflat}2\pi\mu - G\d S$ for some~$\mu\in X$. Then, there holds
  \[
   \liminf_{\varepsilon\to 0} \frac{XY_\varepsilon(\v_\varepsilon)}{|\log\varepsilon|} \geq \pi|\mu|(M).
  \]
  
  \item \emph{$\Gamma$-limsup inequality.} For any~$\mu\in X$ there exists a sequence~$(\v_\varepsilon)$ in~$\T(\TT_\varepsilon; \, \S^2)$
  such that $\hat{\mu}_\varepsilon(\v_\varepsilon) \xrightarrow{\fflat}2\pi\mu - G \d S$ and
  \[
   \limsup_{\varepsilon\to 0} \frac{XY_\varepsilon(\v_\varepsilon)}{|\log\varepsilon|} \leq \pi|\mu|(M).
  \]
 \end{enumerate}
\end{maintheorem}
This Theorem will be proved in the next Section \ref{sec:provaGamma0} (see Proposition~\ref{prop:0Gamma-nu}).
In this Proposition we will actually prove a slightly stronger result, 
where the $\Gamma$-liminf inequality~(ii) is replaced by (a local version of)
\[
 XY_\varepsilon(\v_\varepsilon) \geq \pi|\mu|(M) |\log\varepsilon| - C.
\]
%

\subsection{Location and energetics of defects: the Renormalized Energy and the core energy}
\label{ssec:renormalized_energy}

Beside discussing the generation of defects, we are interested in understanding 
the energetics of defects configurations and, consequently, locate the defects
on the surface $M$. This program is achieved by the analysis of the so called
Renormalized Energy $\mathbb{W}$ introduced by Brezis, Bethuel and H\'elein 
for the Ginzburg Landau equation in \cite{BBH}. In this paper, we obtain 
the Rernomalized Energy as the (first order) $\Gamma$-limit of the discrete energy $XY_\eps$ 
as in \cite{SandierSerfaty}, \cite{AlicandroPonsiglione}, \cite{ADGP} for the euclidean case. 

The Renormalized Energy emerges as the $\Gamma$-limit with respect to the strong $L^2$
convergence of the following rescaled energy
\begin{equation}
\label{eq:rescaled_energy}
XY_\eps (\ve ) - \pi\KK \vert \log\eps\vert,
\end{equation}
where $\KK$ is a positive, even integer such that $\vert \chi(M)\vert \le \KK$.
Now, we introduce the concept of Renormalized Energy. 
Following \cite{ADGP}, we introduce the following class of vector fields in $M$:
\begin{equation} \label{eq:classV}
 \begin{split}
  \mathcal{V}_k: = \left\{\v\in L^2(M; \, \S^2)\colon \textrm{there exist } (x_i)_{i=1}^k\in M^k,
  \ (d_i)_{i=1}^k \in \{-1, \, 1\} \textrm{ such that } \right. \\
  \left. \v\in W^{1,2}_{\tan, \mathrm{loc}}\left(M\setminus \bigcup_{i=1}^{k}x_i; \, \mathbb{S}^2\right) \textrm{ and }
  \star\d\j(\v) = 2\pi\sum_{i = 1}^k d_i \delta_{x_i} - G \right\} \! ,
 \end{split}
\end{equation}
for any~$k\in\N$. Here $\star\d\j$ is a differential operator, defined in Section~\ref{sect:jacobians},
which generalizes the distributional Jacobian, that is, $\star\d\j(\v) = 2\det\nabla\v$ 
if~$\v\in C^2(\R^2, \, \R^2)$.
Given an even number~$\KK\in\N$ such that~$\KK\geq |\chi(M)|$,
we define the intrinsic Renormalized Energy as (see \cite[Eq.~(4.22)]{ADGP}):
\begin{equation}
\label{eq:renormalized_ene_def}
\W_{\mathrm{intr}}(\v) := 
\begin{cases}
\lim_{\delta \to 0} \left(\dfrac{1}{2}\displaystyle\int_{M_\delta}\vert \D\v\vert^2\d S - \KK\pi\vert \log\delta\vert\right) &\hbox{for }\v\in \mathcal{V}_{\KK}\\
-\infty &\hbox{for }\v\in \mathcal{V}_{k},\,\,k <\KK,\\
+\infty &\hbox{otherwise in } L^2(M; \, \mathbb{R}^3),
\end{cases}
\end{equation}
where, given~$\v\in\mathcal{V}_{\KK}$ and $\delta>0$ so small that the balls $B_\delta(x_i)$
are pairwise disjoint, we have set $M_\delta:=M\setminus\bigcup_{i=1}^{\KK}B_\delta(x_i)$.
As we will see in Subsection \ref{sect:renormalized}, the functional $\W_{\mathrm{intr}}$ is well defined.

Now, assume that $\v\in \mathcal{V}_{\KK}$.  
Since in particular, $\vert\v\vert=1$, we have 
\begin{equation}
\label{eq:stima_shape}
\vert \d \ggamma[\v]\vert \le C \,\,\,\,\hbox{ a.e. in } M,
\end{equation}
where the constant $C$ depends only on $M$. 
 Thus, the following quantity exists in $[-\infty,+\infty]$:
\begin{equation}
\label{eq:reno_intr/ex}
 \W(\v) := \W_{\mathrm{intr}}(\v) + \frac{1}{2}\int_{M}\vert \d \ggamma[\v]\vert^2 \d S.
\end{equation}
$\W$ will be called the Renormalized Energy. Note that~$\W$ contains both an intrinsic and an extrinsic term
but, due to~\eqref{eq:stima_shape}, the latter is always finite.
This shows, as expected, that the concentration of the energy is due to the Dirichlet part of
$E_{\mathrm{extr}}$ in \eqref{eq:extrinsic_energy}. 

\begin{remark}
 The Renormalized Energy, as defined by Bethuel, Brezis and H\'elein in~\cite{BBH},
 only depends on the locations~$x_i$ of the defects and their charge~$d_i$. The relation between 
 their Renormalized Energy~$W$ and ours is the following (compare with~\cite[Eq.~(4.23)]{ADGP}):
 \begin{equation}
 \label{eq:relation_renormalized}
  W(x_1, \, \ldots, x_{\KK}; \, d_1, \, \ldots, \, d_{\KK}) = 
  \inf\left\{\W_{\mathrm{intr}}(\v)\colon \v\in\mathcal{V}_{\KK}, \:
  \star\d\j(\v) = 2\pi\sum_{i = 1}^k d_i \delta_{x_i} - G \right\} \! .
 \end{equation}
 The proof of this property is analogous to the proof of the corresponding relation in Euclidean setting, therefore 
 we skip it. 
 Interestingly, 
 the quantity~$W$ can be expressed in terms of the Green function for the Laplace-Beltrami operator of~$M$; 
 see~\cite[Proposition~2]{JerrardIgnat}.
\end{remark}

If we assume that also Hypothesis \eqref{hp:convergence} holds,
for each singularity $x_i$ we can construct 
the so-called core energy, a positive number that takes into account
the energy in the core of the defect at~$x_i$
(see Section~\ref{sect:core}, and in particular Proposition~\ref{prop:core_energy}, for the definition).
As anticipated in the Introduction, the core energy is influenced by the discrete behavior of the triangulation
in the vicinity of each singular point $x_i$ and hence depends explicitly on the particular defect $x_i$.  
The core energy around the singularity $x_i$ will be named $\gamma(x_i)$. 

To correctly state our result, we finally need a proper continuous interpolation of the discrete vector field $\ve$. 
Thus, we define (see \eqref{eq:def_w}) $\w_\varepsilon\colon M\to\R^3$ as 
\[
\w_\varepsilon:=\widehat{\v}_\varepsilon\circ\widehat{P}_\varepsilon^{-1}, 
\]
where $\widehat{\v}_\varepsilon\colon\widehat{M}_\varepsilon\to\R^3$ is the affine interpolant of~$\v_\varepsilon$
and $\widehat{P}_\varepsilon^{-1}$ is the inverse of the nearest-point projection, see \eqref{hp:bilipschitz}.

The second main result is thus the following
\begin{maintheorem}{B}
\label{th:gamma1}
Suppose that the assumptions \eqref{hp:quasiuniform}, \eqref{hp:weakly_acute},
\eqref{hp:bilipschitz} and~\eqref{hp:convergence} are satisfied.
Then the following $\Gamma$-convergence result holds.
\begin{itemize}
 \item[(i)] \emph{Compactness.}
 Let $\KK\in \mathbb{N}$ 
 and let $\ve$ be a sequence in~$\T(\TT_\varepsilon; \, \S^2)$ for which
 there exists a positive constant $C_{\KK}$ such that
 \begin{equation}
 \label{eq:energy_estimate}
 XY_\epsi(\ve) - \KK\pi \vert \log\epsi\vert \le C_{\KK}.
 \end{equation}
 Then, up to a subsequence, there holds
 \begin{equation}
 \label{eq:conv_vort}
 \hat{\mu}_\epsi(\ve)\xrightarrow{\fflat} 2\pi\mu - G \d S
 \end{equation}
  for some 
 $\mu\in X$ with $\sum_{i=1}^k \vert d_i\vert \le \KK$. If $\lvert\mu\rvert = \KK$, 
 then $k=\KK\equiv\chi(M) \mod 2$, $\vert d_i\vert =1$ for any $i$. Moreover,
 there exists $\v\in \mathcal{V}_{\KK}$ and a subsequence such that 
 \begin{equation}
 \label{eq:conv_forte/debole}
 \we \to \v \hbox{ strongly in } L^2(M;\mathbb{R}^3) \hbox{ 
 and weakly in } W^{1,2}_{\mathrm{loc}}(M\setminus \bigcup_{i=1}^{\KK}x_i;
 \mathbb{R}^3).
 \end{equation}
 \item[(ii)] \emph{$\Gamma$-$\liminf$ inequality.}
 Let $\ve\in \T(\TT_\varepsilon; \, \S^2)$ be a sequence
 satisfying \eqref{eq:energy_estimate} with $\KK\equiv\chi(M)\mod 2$
 and converging to some $\v\in \mathcal{V}_{\KK}$
 as in \eqref{eq:conv_vort}-\eqref{eq:conv_forte/debole}. Then, there holds
 \begin{equation}
 \label{eq:gamma_liminf1}
  \liminf_{\epsi\to 0} \left(XY_\epsi(\ve) - \KK \pi\vert \log\epsi\vert\right)
  \ge \WW(\v) +  \sum_{i=1}^{\KK}\gamma(x_i),
 \end{equation}
 where $\gamma(x_i)$ is the core energy around each defect $x_i$ (see Proposition \ref{prop:core_energy}).
\item[(iii)] \emph{$\Gamma$-$\limsup$ inequality.}
 Given $\v\in \mathcal{V}_{\KK}$, 
 there exists $\ve\in \T(\TT_\varepsilon; \, \S^2)$ such that
 $\hat{\mu}_\epsi(\ve)\xrightarrow{\fflat}\star\d\j(\v)$, $\we \to \v$ as in~\eqref{eq:conv_forte/debole} and
 \begin{equation}
 \label{eq:gamma_limsup1}
  \lim_{\eps\to 0}\left(XY_\epsi(\ve) - \KK \pi\vert \log\epsi\vert\right) 
  = \WW(\v) + \sum_{i=1}^{\KK}\gamma(x_i).
 \end{equation}
\end{itemize}
\end{maintheorem}
As anticipated in the Introduction of the paper, the above Theorem
entails a precise convergence result for the minima of $XY_\eps$.
More precisely, the Fundamental Theorem of $\Gamma$-convergence gives that
if we set $\ve^{*}\in \argmin_{\ve \in \T(\TT_\varepsilon; \, \S^2)}XY_\eps(\ve)$
and $\we^*$ as in \eqref{eq:def_w}, there exist points $x_1, \, \ldots, \, x_{|\chi(M)|}$
and a unit-valued tangent field $\v^*$ such that
\begin{equation}
\label{eq:conv_minimizers}
\we^* \to \v^* \qquad
 \hbox{weakly in } W^{1,2}_{\mathrm{loc}}(M\setminus \bigcup_{i=1}^{\vert \chi(M)\vert}x_i;
\mathbb{R}^3).
\end{equation}
Moreover, there holds
\begin{equation}
\label{eq:conv_minima}
\min_{\ve \in \T(\TT_\varepsilon; \, \S^2) } 
XY_\eps(\ve) = \pi\vert \chi(M)\vert\vert\log\eps\vert+\W(\v^*) + \sum_{i=1}^{\vert \chi(M)\vert} \gamma(x_i) + \mathrm{o}_{\eps\to 0}(1).
\end{equation}
Note that the fact that $\KK = \vert \chi(M)\vert$ implies 
that there are only two circumstances for the charge of the defects: either $d_i=1$ for any $i=1\ldots,\KK$ 
either $d_i=-1$ for any $i=1,\ldots, \KK$.

\section{Preliminaries}
\subsection{Distance between triangulations}
\label{sect:mesh_distance}

Let~$\SS$, $\TT$ be two (finite) triangulations on~$\R^2$.
We assume that~$\SS$, $\TT$ are quasi-uniform of size~$1$, i.e., there exists a constant~$\Lambda$
such that, for any~$T\in\SS\cup\TT$, the affine bijection~$\phi_T$ 
from the reference triangle~$T_{\mathrm{ref}}$ to~$T$ satisfies
\begin{equation} \label{quasiuniform-1}
 \max\left\{\Lip(\phi_T), \, \Lip(\phi_T^{-1}) \right\} \leq \Lambda
\end{equation}
(compare with~\eqref{hp:quasiuniform}). We define an isomorphism from~$\SS$ to~$\TT$
as a transformation
\[
 \phi\colon \bigcup_{T\in\SS} T \to \bigcup_{T\in\TT} T
\]
that satisfies the following conditions:
\begin{enumerate}[label = (\roman*)]
 \item for any~$T\in\SS$, restricts to an affine map~$T\to\R^2$;
 \item $\phi$ restricts to a bijection $\SS^0\to\TT^0$;
 \item any three vertices~$i$, $j$, $k\in\SS^0$ span a triangle 
 in~$\SS$ if and only if~$\phi(i)$, $\phi(j)$, $\phi(k)$ span a triangle in~$\TT$.
\end{enumerate}
We denote by~$\Iso(\SS, \, \TT)$ the set of isomorphisms from~$\SS$ to~$\TT$. 
Note that~$\Iso(\SS, \, \TT)\subseteq C\cap W^{1,2}$. We also define
\begin{equation} \label{mesh_distance}
 d(\SS, \, \TT) := \inf_{\phi\in\Iso(\SS, \, \TT)} \max_{i\in\SS^0} |i - \phi(i)|
\end{equation}
(the infimum being equal to~$+\infty$ if~$\Iso(\SS, \, \TT)=\emptyset$). 
The function~$d$ defines a metric on the triangulations of the plane.

\begin{lemma} \label{lemma:mesh_distance}
 Let~$\SS$, $\TT$ be two triangulations such that~$\Iso(\SS, \, \TT)\neq\emptyset$.
 Suppose~\eqref{quasiuniform-1} is satisfied for some~$\Lambda >0$.
 Then, there exists~$\phi\in\Iso(\SS, \, \TT)$ such that
 \[
  \max\left\{\Lip(\phi), \, \Lip(\phi^{-1}) \right\} \leq 1 + C \, d(\SS, \, \TT),
 \]
 where~$C$ is a positive constant that depends only on~$\Lambda$.
\end{lemma}
\begin{proof}
 Let~$\phi\in\Iso(\SS, \, \TT)$ be such that~$|i - \phi(i)| \leq 2 d(\SS, \, \TT)$ for any~$i\in\SS^0$.
 Since the restriction of~$\phi$ to any triangle of~$\SS$ is affine, we deduce that
 \[
  \norm{\Id - \phi}_{L^\infty(T)} \leq 2 d(\SS, \, \TT)
 \]
 on each~$T\in\SS$. Using the assumption~\eqref{quasiuniform-1}, and up to composition with an affine 
 bijection, we can assume without loss of generality that~$T$
 is the triangle of vertices~$(0, \, 0)$, $(1, \, 0)$, $(0, \, 1)$.
 Since the space of affine functions on~$T$ is finite-dimensional,
 the $L^\infty$- and the~$W^{1,\infty}$-norm of an affine map on~$T$ are equivalent. Thus, 
 $\|\Id - \phi\|_{W^{1,\infty}(T)} \leq C d(\SS, \, \TT)$ and the lemma follows.
\end{proof}

\subsection{The metric distorsion tensor}

By the assumption~\eqref{hp:bilipschitz}, the restriction of the nearest-point projection~$\widehat{P}_\varepsilon\colon \widehat{M}_\varepsilon\to M$ 
has a Lipschitz inverse~$\widehat{P}_\varepsilon^{-1}\colon M \to \widehat{M}_\varepsilon$. 
Following~\cite{Hildebrandt-al}, we use this pair of maps to compare~$M$ with its polyhedral approximation~$\widehat{M}_\varepsilon$.
For any~$x\in M$ such that~$\widehat{P}_\varepsilon^{-1}(x)$ falls in the interior of a triangle of~$\widehat{M}_\varepsilon$
(so that~$\widehat{P}_\varepsilon^{-1}$ is smooth in a neighbourhood of~$x$), we let~$\A_\varepsilon(x)$
be the unique linear operator~$\T_x M\to \T_x M$ that satisfies
\begin{equation} \label{metric_distorsion}
 \left(\A_\varepsilon(x)\mathbf{X}, \, \mathbf{Y}\right) 
 = \left(\d \widehat{P}_\varepsilon^{-1}(x) [\mathbf{X}], \, \d \widehat{P}_\varepsilon^{-1}(x) [\mathbf{Y}]\right)
\end{equation}
for any~$\mathbf{X}$, $\mathbf{Y}\in\T_x M$. This defines (almost everywhere) a~$(1, \, 1)$-tensor
field~$\A_\varepsilon\in L^\infty(M; \, \T M\otimes\T^* M)$,
which is called \emph{metric distorsion tensor} in the terminology of~\cite{Hildebrandt-al}.
The metric distorsion tensor is symmetric and positive definite, since the right-hand side of~\eqref{metric_distorsion} is.
We introduce a norm~$\|\cdot\|_{L^\infty(M)}$ on~$L^\infty(M; \, \T M\otimes\T^* M)$ by
\[
 \|\A\|_{L^\infty(M)} := \underset{x\in M}{\supess} \ \|\A(x) \|_{\T M\otimes\T^* M},
\]
where~$\|\cdot\|_{\T M\otimes\T^* M}$ is the operator norm.

\begin{lemma} \label{lemma:metric_distorsion}
 Suppose that~$(\TT_\varepsilon)$ satisfies~\eqref{hp:quasiuniform} and~\eqref{hp:bilipschitz}. Then, there holds
 \[
  \|\A_\varepsilon - \Id\|_{L^\infty(M)} + \|\A^{-1}_\varepsilon - \Id\|_{L^\infty(M)} \leq C\varepsilon.
 \]
\end{lemma}
\begin{proof}
 Let~$\widehat{\nnu}_\varepsilon\colon\widehat{M}_\varepsilon\to\R^3$ be a unit normal field to~$\widehat{M}_\varepsilon$,
 which is well defined (and constant) in the interior of each triangle. The assumption~\eqref{hp:quasiuniform} implies that
 \[
  \|\widehat{\nnu}_\varepsilon\circ\widehat{P}_\varepsilon^{-1} - \ggamma \|_{L^\infty(M)} \leq C\varepsilon, \qquad
  \|\dist(\cdot, \, \widehat{M}_\varepsilon)\|_{L^\infty(M)} \leq C\varepsilon
 \]
 for some $\varepsilon$-independent constant~$C$. (One can write~$M$ as a smooth graph locally around a point~$x\in M$, 
 then use a Taylor expansion; the constant~$C$ can be chosen uniformly with respect to~$x$, by a compactness argument.) 
 Thanks to~\cite[Theorem~1]{Hildebrandt-al}, which gives a formula for~$\A_\varepsilon$ in terms
 of~$(\widehat{\nnu}_\varepsilon\circ\widehat{P}_\varepsilon^{-1}, \, \ggamma)$
 and~$\dist(\cdot, \, \widehat{M}_\varepsilon)$, we deduce
 \begin{equation} \label{metric_distorsion1}
  \norm{\A_\varepsilon - \Id}_{L^\infty(M)} \leq C\varepsilon.
 \end{equation}
 Now, the definition~\eqref{metric_distorsion} of~$\A_\varepsilon$, together with the fact that~$\widehat{P}_\varepsilon^{-1}$
 has a Lipschitz inverse~$\widehat{P}_\varepsilon$ and~$\Lip(\widehat{P}_\varepsilon)\leq\Lambda$ by~\eqref{hp:bilipschitz}, imply that
 \[
  |\A_\varepsilon(x)\mathbf{X}| \geq C|\mathbf{X}|
 \]
 for some constant~$C= C(\Lambda)$, a.e.~$x\in M$ and all~$\mathbf{X}\in\T_x M$, whence~$\|\A_\varepsilon^{-1}\|_{L^\infty(M)}\leq C$.
 Thus, we have
 \[
  \|\A^{-1}_\varepsilon - \Id\|_{L^\infty(M)} \leq \|\A^{-1}\|_{L^\infty(M)} \|\Id - \A_\varepsilon\|_{L^\infty(M)} 
  \stackrel{\eqref{metric_distorsion1}}{\leq} C\varepsilon. \qedhere
 \]
\end{proof}

Let~$g_\varepsilon\in L^\infty(M;\, \T^* M^{\otimes 2})$ be the metric on~$M$ defined 
by~$g_\varepsilon(\mathbf{X}, \, \mathbf{Y}) := (\A_\varepsilon\mathbf{X}, \, \mathbf{Y})$, for any smooth fields~$\mathbf{X}$ and~$\mathbf{Y}$ on~$M$.
Given a function~$u\in W^{1,2}(M)$, one can define the Sobolev $W^{1,2}$-seminorm of~$u$ with respect to~$g_\varepsilon$, i.e.
\begin{equation} \label{Sobolev_eps}
 |u|_{W^{1,2}_\varepsilon(M)}^2 := \int_M \left(\A_\varepsilon^{-1}\nablas u, \, \nablas u\right) (\det \A_\varepsilon)^{1/2} \, \d S.
\end{equation}
By construction~\eqref{metric_distorsion}, the map $\widehat{P}_\varepsilon^{-1}$ is an isometry between~$M$, equipped with the metric~$g_\varepsilon$,
and~$\widehat{M}_\varepsilon$, with the metric induced by~$\R^3$.
Therefore, given  $v\in W^{1, 2}(\widehat{M}_\varepsilon; \, \R)$ and a Borel set~$U\subseteq M$, there holds
\begin{equation} \label{sobolev_isometry}
 |v\circ\widehat{P}_\varepsilon^{-1}|_{W^{1,2}_\varepsilon(U)}^2 =
 \int_{\widehat{P}_\varepsilon^{-1}(U)} |\nablas v|^2 \,\d S.
\end{equation}
Arguing component-wise, we see that the same equality holds for a (not necessarily tangent) vector field~$\v\colon\widehat{M}_\varepsilon\to\R^3$ in place of~$v$. 

\subsection{Interpolants of discrete fields}

Using assumption~\eqref{hp:bilipschitz}, to any discrete vector field~$\v_\varepsilon\in\T(\TT_\varepsilon; \, \S^2)$
we can associate a continuous field~$\w_\varepsilon\colon M\to\R^3$ by setting
\begin{equation}
\label{eq:def_w}
\w_\varepsilon:=\widehat{\v}_\varepsilon\circ\widehat{P}_\varepsilon^{-1}, 
\end{equation}
where $\widehat{\v}_\varepsilon\colon\widehat{M}_\varepsilon\to\R^3$ is the affine interpolant of~$\v_\varepsilon$.
The field~$\w_\varepsilon$ is Lipschitz-continuous and satisfies~$\w_\varepsilon = \v_\varepsilon$ on~$\TT^0_\varepsilon$,
but it is not tangent to~$M$ nor unit-valued, in general. However, one can still prove some useful properties.

\begin{lemma} \label{lemma:norm-interpolant}
 Suppose that~\eqref{hp:quasiuniform}, \eqref{hp:weakly_acute}, \eqref{hp:bilipschitz} are satisfied. Then, for any~$\varepsilon\in(0, \, \varepsilon_0]$
 and any discrete field~$\v_\varepsilon\in\T(\TT_\varepsilon; \, \S^2)$, 
 $\we$ is Lipschitz-continuous with Lipschitz constant 
 \begin{equation}
 \label{eq:lip_we}
 \Lip(\w_\varepsilon)\leq C\varepsilon^{-1}
 \end{equation}
 and, for any subset~$\widehat{U}\subseteq\widehat{M}_\varepsilon$ that can be written 
 as union of triangles of~$\TT_\eps$, there holds
 \begin{equation}
 \label{eq:energy_normwe}
  XY_\varepsilon(\v_\varepsilon, \, \widehat{U}) := 
  \frac12 \sum_{i, j\in \TT^0_\eps\cap\widehat{U}}\kappa_\eps^{ij} \abs{\v_\eps(i) - \v_\eps(j)}^2
  = \frac12 |\w_\varepsilon|^2_{W^{1,2}_\varepsilon(P(\widehat{U}))}.
 \end{equation}
\end{lemma}
\begin{proof}
   From the very definition of~$\w_\varepsilon:=\widehat{\v}_\varepsilon\circ\widehat{P}_\varepsilon^{-1}$, it follows that
  \begin{equation*} 
   \Lip(\w_\varepsilon) \stackrel{\eqref{hp:bilipschitz}}{\leq} C\Lip(\widehat{\v}_\varepsilon)
   \leq C \sup_{[i, \, j]\in\TT^1_\varepsilon}\frac{|\v_\varepsilon(i) - \v_\varepsilon(j)|}{|i-j|}
   \stackrel{\eqref{hp:quasiuniform}}{\leq} C\varepsilon^{-1}.
  \end{equation*}
 To prove \eqref{eq:energy_normwe}, it is enough to combine~\eqref{exact_integration} with~\eqref{sobolev_isometry}.
\end{proof}


\begin{lemma} \label{lemma:proiection-pointwise}
 Suppose that~$(\TT_\varepsilon)$ satisfies~\eqref{hp:quasiuniform}, \eqref{hp:bilipschitz}. Then, there exists a contant~$C$ such that, 
 for any~$\varepsilon\in(0, \, \varepsilon_0]$ and any~$\v_\varepsilon\in\T(\TT_\varepsilon; \, \S^2)$, there holds
 \[
  \norm{\left(\w_\varepsilon, \, \ggamma\right)}_{L^\infty(M)} \leq C\varepsilon.
 \]
\end{lemma}
\begin{proof}
 Every point~$x\in\widehat{M}_\varepsilon$ can be written in the form~$x = \lambda_0 i_0 + \lambda_1 i_1 + \lambda_2 i_2$,
 where~$i_k\in\TT^0_\varepsilon$ and~$\lambda_k\geq 0$, $\lambda_0 + \lambda_1+ \lambda_2 = 1$.
 Using the definition of the affine interpolant, and the fact that~$(\v_\varepsilon(i_k), \, \ggamma(i_k)) = 0$, we can write
 \[
  \begin{split}
   \abs{\left(\widehat{\v}_\varepsilon(x), \, (\ggamma\circ\widehat{P}_\varepsilon)(x)\right)}
   &\leq \sum_{k=0}^2 \lambda_k \abs{\left(\v_\varepsilon(i_k), \,
   (\ggamma\circ\widehat{P}_\varepsilon)(x) - (\ggamma\circ\widehat{P}_\varepsilon)(i_k) \right)} \\
   &\leq \|\nabla(\ggamma\circ P)\|_{L^\infty(U)} \sup_{T\in\TT_\varepsilon}\diam(T).
  \end{split}
 \]
 Thus, using the smoothness of~$\ggamma$ and the assumptions~\eqref{hp:quasiuniform}, \eqref{hp:bilipschitz}, we deduce
 \[
  \norm{\left(\w_\varepsilon, \, \ggamma\right)}_{L^\infty(M)} =
  \norm{\left(\widehat{\v}_\varepsilon, \, \ggamma\circ\widehat{P}_\varepsilon\right)}_{L^\infty(\widehat{M}_\varepsilon)} 
  \leq C\varepsilon. \qedhere
 \]
\end{proof}

The following property is well-known in the flat case (see e.g.~\cite[Lemma~2]{AlicandroCicalese}).

\begin{lemma} \label{lemma:GL-pointwise}
 Suppose that~\eqref{hp:quasiuniform} is satisfied.
 Then, there exists a positive constant~$C$ such that, for any~$0 < \varepsilon\leq\varepsilon_0$ 
 and any discrete field~$\v_\varepsilon\in\T(\TT_\varepsilon; \, \S^2)$, there holds
 \[
  \frac{1}{\varepsilon^2} \left(1 - |\w_\varepsilon|^2\right)^2
  \leq C \abs{\nablas\w_\varepsilon}^2 \qquad \textrm{pointwise on } M.
 \]
\end{lemma}
\begin{proof}
 Thanks to~\eqref{hp:bilipschitz}, it suffices to show that
 \begin{equation} \label{GL-pointwise1}
  \frac{1}{\varepsilon^2} \left(1 - |\widehat{\v}_\varepsilon|^2\right)^2
  \leq C \abs{\nablae\widehat{\v}_\varepsilon}^2 \qquad \textrm{on } \widehat{M}_\varepsilon.
 \end{equation}
 Let~$T\in\TT_\varepsilon$ be a triangle with vertices~$i_0$, $i_1$, $i_2$. Any point~$x\in T$ can be 
 written as $x = i_0 +\lambda_1 (i_1 - i_0) + \lambda_2(i_2 - i_0)$, where~$\lambda_1$, $\lambda_2$
 are positive numbers such that~$\lambda_1 + \lambda_2 \leq 1$.
 Using the definition of affine interpolant and that~$|\v_\varepsilon(i_0)| = 1$, we obtain that
 \[
  1 - |\widehat{\v}_\varepsilon(x)| \leq |\widehat{\v}_\varepsilon(x) - \v_\varepsilon(i_0)| 
  \leq \sum_{k=1}^2 \lambda_k |\v_\varepsilon(i_k) - \v_\varepsilon(i_0)|,
 \]
 whence, using that~$|\widehat{\v}_\varepsilon|\leq 1$ and that~$\nablae\widehat{\v}_\varepsilon$ is constant on~$T$, we deduce
 \[
  \left(1 - |\widehat{\v}_\varepsilon(x)|^2\right)^2 \leq 4 \left( 1- |\widehat{\v}_\varepsilon(x)|\right)^2
  \leq 8\sum_{k=1}^2 |\v_\varepsilon(i_k) - \v_\varepsilon(i_0)|^2
  = 8 \sum_{k=1}^2 \abs{\nablae\widehat{\v}_\varepsilon(x) (i_k - i_0)}^2.
 \]
 Now,~\eqref{GL-pointwise1} follows because~$|i_k - i_0|\leq C\varepsilon$, due to~\eqref{hp:quasiuniform}.
\end{proof}

As a consequence of Lemmas~\ref{lemma:norm-interpolant} and~\ref{lemma:GL-pointwise},
if both~\eqref{hp:quasiuniform} and~\eqref{hp:weakly_acute} are satisfied, then we have
\begin{equation} \label{remark:GL}
 \frac{1}{\varepsilon^2} \int_{\widehat{M}_\varepsilon} \left(1 - |\w_\varepsilon|^2\right)^2
 \leq C \, XY_\varepsilon(\v_\varepsilon).
\end{equation}

\subsection{Jacobians of vector fields on~$M$}
\label{sect:jacobians}

In this section, we define the Jacobian determinant of a vector field in the sense of distributions,
and we recall a few useful properties. This notion was introduced in the context of Ginzburg-Landau functionals 
(see e.g.~\cite{JerrardSoner-Jacobians}) and in nonlinear elasticity (see e.g.~\cite{Ball76}, \cite{Mueller}).
As we are dealing with vector fields over a manifold, it will be useful
to recast the theory in the language of differential forms.

Given a map~$\u\in (W^{1, 1}\cap L^\infty)(M;\,\R^3)$, we define the ``pre-jacobian'' or vorticity of~$\u$ as the $1$-form
\begin{equation} \label{prejacobian}
 \j(\u) := \left(\ggamma, \, \u\wedge\d\u\right) \! .
\end{equation}
More explicitely, $\j(\u)$ is defined via its action on a smooth, \emph{tangent} field $\w$ on~$M$:
\begin{equation} \label{prejacobian2}
 \langle\j(\u), \, \w\rangle = \left(\ggamma, \, \u\times\nabla_{\w}\u\right) \! .
\end{equation}
We can equivalently replace~$\nabla_{\w}\u$ with the covariant derivative~$\D_{\w}\u$ 
since the scalar product in~\eqref{prejacobian2} does not depend on the component of~$\nabla_{\w}\u$ 
in the direction of~$\ggamma$.

Suppose now that~$\u\in W^{1,1}(M; \, \R^3)$ is a \emph{unit}, \emph{tangent} field on~$M$ 
(that is, $|\u|=1$ and $\u\cdot\ggamma =0$ a.e.), and let~$(\e_1, \, \e_2)$ be a local orthonormal basis for the tangent frame of~$M$.
Then, we have
\begin{equation} \label{spin_norm}
 \abs{\j(\u)}^2 = \sum_{k=1}^2 \abs{\u\times\D_{\e_k}\u}^2 = \sum_{k=1}^2 \abs{\D_{\e_k}\u}^2 = \abs{\D\u}^2,
\end{equation}
where we denote by~$|\cdot|$ both the norm on the tangent space and the induced norm on the cotangent space.
Moreover, we can write locally that
\begin{equation} \label{angular_variable}
 \u = (\cos\alpha) \e_1 + (\sin\alpha) \e_2
\end{equation}
for some scalar function~$\alpha$ with bounded variation (this follows, e.g., by \cite{DavilaIgnat}).
A \emph{formal} computation shows that
\begin{equation} \label{spin_connection}
 \j(\u) = \d\alpha - \A
\end{equation}
where~$\A$, called \emph{spin connection}, is the~$1$-form defined by $\langle\A, \, \w\rangle := \e_1\cdot\nabla_{\w}\e_2$.
Note that~$\A$ depends on the choice of the frame, but its differential is an intrinsic quantity:
\begin{equation} \label{Gauss}
 \d\A = G\, \d S,
\end{equation}
where we recall that~$G$ is the Gauss curvature of~$M$.


The differential~$\d\j(\u)$ will play an important r\^ole.
Since~$\d\j(\u)$ is a $2$-form, it can be written uniquely 
as~$\d\j(\u) = f\boldsymbol{\omega}$ where~$f\in \mathscr{D}^\prime(M)$
is scalar and~$\boldsymbol{\omega}$ is the volume form on~$M$;
we use the notation~$\star\d\j(\u) := f$. In case~$M = \R^2$ 
(embedded as the $xy$-plane in~$\R^3$, so that~$\ggamma = \e_3$) 
and~$\u$ is a smooth vector field~$\R^2\to\R^2$, we easily compute that
\[
 \star\d\j(\u) = 2\det\nabla\u,
\]
thus $\star\d\j(\u)$ can be thought as a generalization of the Jacobian determinant of~$\u$, up to a constant factor~$2$.
If~$\u$ is a unit, tangent field on~$M$, then by differentiating~\eqref{spin_connection} we see that
$\star\d\j(\u)$ contains topological information about the singularities of~$\u$.
We denote by~$\ind(\u, \, x_i)$ the local degree of~$\u$ at the point~$x_i$,
that is, the winding number of~$\u$ around the boundary of a small disk centred at~$x_i$ (see e.g.~\cite{AGM-VMO} for more details).

\begin{lemma} \label{lemma:dj}
 Let~$\u\in W^{1,1}_{\mathrm{tan}}(M; \, \S^2)$ be a unit, tangent field. 
 Suppose that there exist a finite number of points~$x_1, \, \ldots, \, x_p$
 such that $\u\in W^{1,2}_{\mathrm{loc}}(M\setminus\{x_1, \, \ldots, \, x_p\}; \, \R^3)$.
 Then
 \[
  \star\d\j(\u) = 2\pi\sum_{i = 1}^p \ind(\u, \, x_i) \delta_{x_i} - G  \qquad \textrm{in } \mathscr{D}^\prime(M).
 \]
\end{lemma}
\begin{proof}
 We can assume without loss of generality that~$\u$ is
 smooth on~$M\setminus\{x_1, \, \ldots, \, x_p\}$, 
 as smooth unit-norm tangent fields are dense in~$W^{1,2}$. 
 This follows essentially from the density
 result~\cite[Proposition p.~267]{SU}. The paper~\cite{SU} 
 is concerned with the study of maps from~$M$ to a fixed 
 target manifold; however, the arguments can be adapted to sections 
 of the unit tangent bundle (see also~\cite[Section~3]{AGM-VMO}
 for further details).
 
 For a fixed~$i$, take a test function~$\varphi\in C^\infty(M)$
 whose support is simply connected
 and does not contain any singularity of~$\v$ other than~$x_i$.
 Suppose that an orthonormal tangent frame~$(\e_1, \, \e_2)$ is defined on the support of~$\varphi$.
 Then, we can locally define an angular variable~$\alpha$ which satisfies Equation~\eqref{angular_variable}
 and is smooth, except for a jump across a smooth ray starting at the point~$x_i$.
 The size of the jump is constant along the ray, and equal to~$2\pi\ind(\u, \, x_i)$. 
 The Lebesgue-absolutely continuous part~$\d^{\mathrm{ac}}\alpha$ of the differential~$\d\alpha$
 is actually continuous across the jump, and satisfies (analogously to~\eqref{spin_connection})
 \begin{equation} \label{spin_connection2}
  \j(\u) = \d^{\mathrm{ac}}\alpha - \A
 \end{equation}
 on the support of~$\varphi$.
 Thanks to~\eqref{spin_connection2}, \eqref{spin_norm} and the fact that~$\u\in W^{1, 1}$, 
 we deduce that $\d^{\mathrm{ac}}\alpha\in L^1$. Moreover, one checks that
 $\d(\d^{\mathrm{ac}}\alpha) = 0$ on~$M\setminus\{x_i\}$.
 
 Now, we compute~$\star\d(\d^{\mathrm{ac}}\alpha)$ in the sense of distributions.
 For any~$\delta > 0$, we have
 \begin{equation} \label{dj(v)1}
  \begin{split}
   -\left\langle\d^{\mathrm{ac}}\alpha, \, \star\d\varphi \right\rangle_{L^2(M\setminus B_\delta(x_i))} 
   &= \int_{M\setminus B_\delta(x_i)} \d^{\mathrm{ac}}\alpha\wedge\d\varphi
   = -\int_{M\setminus B_\delta(x_i)} \d (\d^{\mathrm{ac}}\alpha\wedge\varphi) \\
   &= \int_{\partial B_\delta(x_i)} \d^{\mathrm{ac}}\alpha\wedge\varphi.
  \end{split}
 \end{equation}
 On the other hand, we have
 \begin{equation} \label{dj(v)2}
  \abs{\int_{\partial B_\delta(x_i)} \d^{\mathrm{ac}}\alpha\wedge\left(\varphi - \varphi(0)\right) }
  \leq \delta \norm{\nablas\varphi}_{L^\infty(M)} \int_{\partial B_\delta(x_i)} \abs{\d^{\mathrm{ac}}\alpha} \, \d s.
 \end{equation}
 We claim that the right-hand side of~\eqref{dj(v)2} converges to~$0$ at least along a subsequence~$\delta_j\searrow 0$. 
 For otherwise, there would exist positive numbers~$\eta$ and~$\delta_0$ such that
 \[
  \delta\int_{\partial B_\delta(x_i)} \abs{\d^{\mathrm{ac}}\alpha} \,\d s \geq \eta
 \]
 for any~$0 <\delta \leq \delta_0$. Dividing by $\delta$ both sides of this inequality and
 integrating over~$\delta\in(0, \, \delta_0)$, we would obtain
 \[
  \int_{B_{\delta_0}(x_i)} \abs{\d^{\mathrm{ac}}\alpha} \,\d S \geq \eta \int_0^{\delta_0} \frac{\d\delta}{\delta} = + \infty,
 \]
 which is impossible because~$\d^{\mathrm{ac}}\alpha\in L^1$.
 Then, we find a subsequence~$\delta_j\searrow 0$ along which the right-hand side of~\eqref{dj(v)2} converges to~$0$.
 Taking the limit in~\eqref{dj(v)1} along this subsequence, and using~\eqref{dj(v)2}, we obtain
 \[
  -\left\langle\d^{\mathrm{ac}}\alpha, \, \star\d\varphi \right\rangle_{L^2(M)}  
  = \lim_{j\to+\infty} \int_{\partial B_{\delta_j}(x_i)} \d^{\mathrm{ac}}\alpha\wedge\varphi(0)
  = 2\pi\ind(\u, \, x_i) \, \varphi(0).
 \]
 Since the operator~$\star\d$ is $L^2(M)$-anti-symmetric, the left-hand side of this identity
 can be interpreted as the duality pairing $\langle\star\d(\d^{\mathrm{ac}}\alpha), \, \varphi\rangle$, in the sense of distributions.
 Combining this with~\eqref{Gauss} and~\eqref{spin_connection2}, the lemma follows.
\end{proof}

We define a piecewise-continuous counterpart of~$\j$. Take a bounded, piecewise-smooth (but not necessarily tangent) 
map~$\u\colon\widehat{M}_\varepsilon\to\R^3$ such that, for any edge~$e=[i, \, j]\in\TT^1_\varepsilon$,
\begin{equation} \label{edge_continuity}
 \nabla_{j-i}\u = \nabla\u(j-i) \textrm{ is continuous across } e.
\end{equation}
For example, the affine interpolant~$\u=\widehat{\v}_\varepsilon$ of a discrete field~$\v\in\T(\TT_\varepsilon;\,\S^2)$
satisfies~\eqref{edge_continuity}. We let
\begin{equation} \label{discrete_prejacobian}
 \widehat{\j}_\varepsilon(\u) := \left(\widehat{\ggamma}_\varepsilon, \, \u\wedge\d\u\right) \! ,
\end{equation}
that is the piecewise-smooth $1$-form on~$\widehat{M}_\varepsilon$ satisfying
\begin{equation*}
 \langle\widehat{\jmath}_\varepsilon(\u), \, \w\rangle = \left(\widehat{\ggamma}_\varepsilon, \, \u\times\nabla_{\w}\u\right)
\end{equation*}
for any piecewise-smooth tangent field~$\w$ on~$\widehat{M}_\varepsilon$. This form is well-defined and continuous on each triangle of~$\widehat{M}_\varepsilon$. 
Note that~$\widehat{\j}_\varepsilon(\u)$ may not be continuous across an edge~$e = [i, \, j]$ but $\langle\widehat{\j}_\varepsilon(\u), \, i - j\rangle$ is,
therefore the integral of~$\widehat{\jmath}_\varepsilon(\u)$ along~$e$ is defined unambiguously.

\subsection{Jacobians of discrete vector-fields}
\label{ssec:dic_meas}

We want to define a notion of ``jacobian'' for a discrete field~$\v_\varepsilon\in\T(\TT_\varepsilon; \, \S^2)$ and we have two possibilities: 
either we apply~$\d\widehat{\j}_\varepsilon$ to the affine interpolant~$\widehat{\v}_\varepsilon$, or we compute~$\d\j(\u_\varepsilon)$
for a field~$\u_\varepsilon\colon M\to\R^3$ that interpolates~$\v_\varepsilon$.
The first possibility corresponds to the measure
\begin{equation} \label{mu_hat}
 \widehat{\mu}_\varepsilon(\v_\varepsilon) := \sum_{T\in\TT_\varepsilon} \left(\int_{T} \d \widehat{\j}_\varepsilon(\widehat{\v}_\varepsilon)\right) \delta_{x_T},
\end{equation}
where~$\delta_{x_T}$ is the Dirac delta measure supported by the barycentre~$x_T$ of~$T$.
Let~$(i_0, \, i_1, \, i_2)$ be the vertices of a triangle~$T\in\TT_\varepsilon$, 
sorted in counter-clockwise order with respect to the orientation induced by~$\ggamma$,
and let~$i_3 := i_0$. Using Stokes' theorem and the definition of the affine interpolant, we compute
\begin{equation} \label{mu_hat2}
 \begin{split}
  \widehat{\mu}_\varepsilon(\v_\varepsilon)[T] &= 
  \sum_{k = 0}^2 \int_{[i_k, \, i_{k+1}]} \left(\widehat{\ggamma}_\varepsilon, \, \widehat{\v}_\varepsilon\times\nabla_{i_{k+1}-i_k}\widehat{\v}_\varepsilon\right) \d s \\
  &= \sum_{k = 0}^2 \left(\frac{\ggamma(i_k) + \ggamma(i_{k+1})}{2}, \, \v_\varepsilon(i_k)\times\v_\varepsilon(i_{k+1})\right) \! .
 \end{split}
\end{equation}

As for the second approach, we construct a suitable field~$\u_\varepsilon$ in the following way.
We fix a sequence~$(t_\varepsilon)_{\varepsilon>0}$ such that
\begin{equation} \label{t_eps}
 \frac{\varepsilon|\log\varepsilon|}{t_\varepsilon} \to 0 \qquad \textrm{as } \varepsilon\to 0,
\end{equation}
e.g. $t_\varepsilon := \varepsilon|\log\varepsilon|^2$. Now, reminding that~$\w_\varepsilon := \widehat{\v}_\varepsilon\circ\widehat{P}_\varepsilon^{-1}$, 
for~$x\in M$ we define
\begin{equation} \label{u_eps}
 \tilde{\u}_\varepsilon(x) := \mathrm{proj}_{\T_x M}\w_\varepsilon(x) \qquad \textrm{and } \qquad
 \u_\varepsilon(x) := \eta_\varepsilon\left(|\tilde{\u}_\varepsilon(x)|\right) \tilde{\u}_\varepsilon(x),
\end{equation}
where~$\eta_\varepsilon(s) := \min\{t_\varepsilon^{-1}, \, s^{-1}\}$.
Note that~$\u_\varepsilon$ is a Lipschitz tangent field on~$M$ and~$\u_\varepsilon = \v_\varepsilon$ on~$\TT^0_\varepsilon$.
Next, we set
\begin{equation} \label{mu}
 \mu_\varepsilon(\v_\varepsilon) := \sum_{T\in\TT_\varepsilon} \left( \int_{P(T)} \d\j(\u_\varepsilon) \right) \delta_{P(x_T)}.
\end{equation}
Given a Borel set~$E\subseteq M$, let~$E_\varepsilon$ be the union of all the~$P(T)$'s such that~$T\in\TT_\varepsilon$,~$P(x_T)\in E$.
If $|\tilde{\u}_\varepsilon|\geq 1/4$ on~$\partial{E}_\varepsilon$, then we can find a unit tangent 
field~$\mathbf{U}_\varepsilon\in W^{1,1}_{\tan}(E_\varepsilon; \, \S^2)$ such that $\mathbf{U}_\varepsilon = \u_\varepsilon$ on~$\partial E_\varepsilon$
and~$\mathbf{U}_\varepsilon$ is smooth except at finitely many points. (One can modify~$\tilde{\u}_\varepsilon$
in such a way that it is smooth and has~$0$ as a regular value, then define~$\mathbf{U}_\varepsilon:=\tilde{\u}_\varepsilon/|\tilde{\u}_\varepsilon|$.)
Since~$\mu_\varepsilon(\v_\varepsilon)[E] = \int_{E_\varepsilon}\d\j(\u_\varepsilon)$ and, by Stokes' theorem,
the latter only depends on the restriction of~$\u_\varepsilon$ to~$\partial E_\varepsilon$,
we have $\mu(\v_\varepsilon)[E] = \int_{E_\varepsilon}\d\j(\mathbf{U}_\varepsilon)$ and hence, by Lemma~\ref{lemma:dj},
\begin{equation} \label{mu2}
 \mu_\varepsilon(\v_\varepsilon)[E] = 2\pi\ind(\u_\varepsilon, \, \partial E_\varepsilon) - \int_{E_\varepsilon} G \, \d S.
\end{equation}
In this sense, the measure~$\mu_\varepsilon(\v_\varepsilon)$ can be thought as a generalization of the discrete vorticity defined in~\cite{ADGP}
(see in particular~\cite[Remark~2.1]{ADGP}), and immediately provides information on the ``topological'' behaviour of~$\v_\varepsilon$.
On the other hand, the measure~$\widehat{\mu}_\varepsilon(\v_\varepsilon)$ has the advantage of being simpler to evaluate, thanks to~\eqref{mu_hat2}.
Luckily, if the XY-energy of the field~$\v_\varepsilon$ satisfies a logarithmic bound, then the two measures are close to each other.

\begin{prop} \label{prop:mu}
 Suppose that~\eqref{hp:quasiuniform}, \eqref{hp:weakly_acute}, \eqref{hp:bilipschitz} are satisfied.
 Let~$(\v_\varepsilon)_{0<\varepsilon\leq \varepsilon_0}$ be a sequence of discrete fields that satisfies
 \eqref{H} 
 for some $\varepsilon$-independent constant~$\Lambda$ and any~$0 < \varepsilon\leq\varepsilon_0$.
 Then, there holds
 \[
  \norm{\widehat{\mu}_\varepsilon(\v_\varepsilon) - \mu_\varepsilon(\v_\varepsilon)}_{\fflat} \leq 
  C\left(\frac{\varepsilon|\log\varepsilon|}{t_\varepsilon} + \varepsilon|\log\varepsilon|\right) \! .
 \]
\end{prop}

In particular, the difference between the two measures converges to zero in the flat norm as~$\varepsilon\to 0$, if we assume that~\eqref{t_eps} holds.
The rest of this section is devoted to the proof of Proposition~\ref{prop:mu}. The key fact is the following continuity property 
for the Jacobian, which is well-known for maps~$\u\colon\Omega\subseteq\R^n\to\R^n$
(see e.g.~\cite[Lemma~2.1]{AlicandroPonsiglione}).

\begin{lemma} \label{lemma:continuity_jacobian}
 Let~$\u$, $\w$ be (not necessarily tangent) fields in $W^{1,2}(M; \, \R^3)$. Then, there holds
 \begin{gather} 
  \|\star\d\j(\u)\|_{L^1(M)} \leq C \left(\|\u\|_{L^2(M)}^2
  + \|\nablas\u\|^2_{L^2(M)}\right) \! , \label{jacobian_L1} \\
  \norm{\star\d\j(\u) - \star\d\j(\w)}_{\fflat} \leq \norm{\u-\w}_{L^2(M)}
  \left(\norm{\nablas\u}_{L^2(M)} +  \norm{\nablas\w}_{L^2(M)}\right) \! .\label{jacobian_flat}
 \end{gather}
\end{lemma}
\begin{proof}
 By a density argument, we can assume WLOG that~$\u$, $\w$ are smooth. Using Einstein convention, we can write
 \[
  \j(\u) = \ggamma^i \j_i(\u), \quad \textrm{where} \quad \j_i(\u) := \epsilon_{ijk} \, \u^j \, \d\u^k
 \]
 and~$\epsilon_{ijk}$ is the Levi-Civita symbol, given by~$\epsilon_{ijk} := 1$ if~$(i, \, j, \, k)$ is an even permutation of~$(1, \, 2, \, 3)$,
 $\epsilon_{ijk} := -1$ if it is an odd permutation, and~$\epsilon_{ijk} := 0$ otherwise. 
 By differentiating, we deduce
 \begin{equation} \label{}
  \d\j(\u) = \epsilon_{ijk} \, \u^j \, \d\ggamma^i \wedge \d\u^k + \epsilon_{ijk} \, \ggamma^i \, \d\u^j \wedge \d\u^k,
 \end{equation}
 whence~\eqref{jacobian_L1} immediately follows by applying the H\"older inequality and using
 that~$|\nablas\ggamma|$ is bounded.
 We now prove~\eqref{jacobian_flat}. A straightforward computation shows that
 \[
  \j_3(\u) - \j_3(\w) = 
  \frac12 \left(\j_3\left(\u^1 - \w^1, \, \u^2 + \w^2\right) - \j_3\left(\u^2 - \w^2, \, \u^1 + \w^1\right)\right)
 \]
 and similar equalities hold for~$\j_1$, $\j_2$, therefore
 \begin{equation} \label{contjac1}
  \d\j(\u) - \d\j(\w) = \frac{\epsilon_{ijk}}{2} \, \d\left(\ggamma^i \, \j_i(\u^j - \w^j, \, \u^k + \w^k)\right).
 \end{equation}
 Now, fix a function~$\varphi\in C^\infty_{\mathrm{c}}(U)$. 
 Thanks to~\eqref{contjac1} and an integration by parts, we deduce
 \[
  \langle\star\d\j(\u)-\star\d\j(\w), \, \varphi\rangle
  = -\frac{\epsilon_{ijk}}{2} \langle\ggamma^i \, \j_i(\u^j - \w^j, \, \u^k + \w^k), \, \star\d\varphi \rangle \! .
 \]
 The definition of~$\j_i$ and the H\"older inequality immediately imply
 \[
  \langle\star\d\j(\u)-\star\d\j(\w), \, \varphi\rangle
  \leq \|\u-\w\|_{L^2(M)} \|\nablas\u - \nablas\w\|_{L^2(M)} \|\nablas\varphi\|_{L^\infty(M)},
 \]
 whence~\eqref{jacobian_flat} follows by taking the supremum over~$\varphi$.
\end{proof}

Lemma~\ref{lemma:continuity_jacobian} has a counterpart in the piecewise-continuous setting. For further reference, here we only mention that

\begin{lemma} \label{lemma:discrete_jacobian}
 Let~$\u\colon\widehat{M}_\varepsilon\to\R^3$ be a (not necessarily tangent) piecewise-smooth field that satisfies~\eqref{edge_continuity}.
 Then, there holds
 \[
  \|\star\d\widehat{\j}_\varepsilon(\u)\|_{L^1(T)} \leq C\left(\|\u\|_{L^2(T)}^2
  + \|\nablae\u\|^2_{L^2(T)}\right) \qquad \textrm{for any } T\in\TT_\varepsilon.
 \]
\end{lemma}
\begin{proof}
 We argue as in Lemma~\ref{lemma:continuity_jacobian}, using that the functions~$\widehat{\ggamma}_\varepsilon$ are Lipschitz continuous and
 $\|\nablae\widehat{\ggamma}_\varepsilon\|_{L^\infty(\widehat{M}_\varepsilon)}
 \leq\|\nablas\ggamma\|_{L^\infty(M)}$.
\end{proof}

\begin{proof}[Proof of Proposition~\ref{prop:mu}]
 Given a piecewise-smooth map $\u\colon\widehat{M}_\varepsilon\to\R^3$, we let $\j(\u) := (\ggamma\circ P, \, \u\wedge\d\u)$,
 i.e.~we extend the operator $\u\mapsto\j(\u)$ to fields~$\u$ that are not defined on~$M$ by pre-composing~$\ggamma$ with the projection~$P\colon U\to M$.
 When~$\u$ is a piecewise-smooth field, we denote by~$\d\j(\u)$, $\d\widehat{\j}_\varepsilon(\u)$ the Lebesgue-absolutely 
 continuous part of the distributional differential of~$\j(\u)$, $\widehat{\j}_\varepsilon(\u)$ respectively --- 
 that is, we neglect any jumps that may arise at the boundary of the regions where~$\u$ is smooth.
 
 Let~$(\v_\varepsilon)$ be a sequence of discrete fields satisfying the logaritmic energy bound~\eqref{H}.
 The assumption~\eqref{H} together with~\eqref{affine_dirichlet}, \eqref{remark:GL} and the fact that~$|\widehat{\v}_\varepsilon|\leq 1$
 implies that
 \begin{equation} \label{log_norm}
  \|\widehat{\v}_\varepsilon\|_{L^2(\widehat{M}_\varepsilon)}^2 + \|\nablae\widehat{\v}_\varepsilon\|^2_{L^2(\widehat{M}_\varepsilon)} 
  + \varepsilon^{-2} \|1 - |\widehat{\v}_\varepsilon|^2\|_{L^2(\widehat{M}_\varepsilon)}^2
  \leq C|\log\varepsilon|
 \end{equation}
 for any~$0 < \varepsilon\leq \varepsilon_0$ and some constant~$C = C(M, \, \Lambda, \, \varepsilon_0)$, provided that~$\varepsilon_0 < 1$.
 We decompose the difference~$\widehat{\mu}_\varepsilon(\v_\varepsilon) - \mu_\varepsilon(\v_\varepsilon)$
 as a sum of several terms:
 \begin{equation*} \label{mu-sum}
  \begin{split}
   \widehat{\mu}_\varepsilon(\v_\varepsilon) - \mu_\varepsilon(\v_\varepsilon) &= 
   \underbrace{\widehat{\mu}_\varepsilon(\v_\varepsilon) - \star\d\widehat{\j}_\varepsilon(\widehat{\v}_\varepsilon)}_{=:A_1} \
   + \ \underbrace{\star\d\widehat{\j}_\varepsilon(\widehat{\v}_\varepsilon) - \star\d\j(\widehat{\v}_\varepsilon)}_{=:A_2} \
   + \ \underbrace{\star\d\j(\widehat{\v}_\varepsilon) - \star\d\j(\w_\varepsilon)}_{=:A_3} \\
   &+ \ \underbrace{\star\d\j(\w_\varepsilon) - \star\d\j(\tilde{\u}_\varepsilon)}_{=:A_4} \
   + \underbrace{\star\d\j(\tilde{\u}_\varepsilon) - \star\d\j(\u_\varepsilon)}_{=:A_5} \
   + \ \underbrace{\star\d\j(\u_\varepsilon) - \mu_\varepsilon(\u_\varepsilon)}_{=:A_6}.
  \end{split}
 \end{equation*}
 Throughout the rest of the proof, we let~$\varphi\in C^\infty_{\mathrm{c}}(U)$ be an arbitrarily fixed test function.
 
 \medskip
 \noindent\emph{Analysis of~$A_1$.}
 There holds
 \[
  \begin{split}
   \langle \widehat{\mu}_\varepsilon(\v_\varepsilon) - \star\d\widehat{\j}_\varepsilon(\widehat{\v}_\varepsilon), \, \varphi\rangle
   &= \sum_{T\in\TT_\varepsilon} \int_T \left(\varphi(x_T) - \varphi\right) \, \d\widehat{\j}_\varepsilon(\widehat{\v}_\varepsilon)  \\
   &\leq \|\star\d\widehat{\j}_\varepsilon(\widehat{\v}_\varepsilon)\|_{L^1(\widehat{M}_\varepsilon)}
   \|\nabla\varphi\|_{L^\infty(U)} \sup_{T\in\TT_\varepsilon} \diam(T)
  \end{split}
 \]
 Using Lemma~\ref{lemma:discrete_jacobian}, the assumption~\eqref{hp:quasiuniform} and~\eqref{log_norm}, we deduce
 \begin{equation} \label{mu-A1}
  \|\widehat{\mu}_\varepsilon(\v_\varepsilon) - \star\d\widehat{\j}_\varepsilon(\widehat{\v}_\varepsilon)\|_{\fflat}
  \leq C\varepsilon|\log\varepsilon|.
 \end{equation}

 \medskip
 \noindent\emph{Analysis of~$A_2$.}
 By integrating by parts on each triangle of the triangulation, we can write
 \[
  \begin{split}
   \langle \star\d\widehat{\j}_\varepsilon(\widehat{\v}_\varepsilon) - \star\d\j(\widehat{\v}_\varepsilon), \, \varphi\rangle
   &= -\sum_{T\in\TT_\varepsilon} \int_T \left(\widehat{\j}_\varepsilon(\widehat{\v}_\varepsilon) 
   - \j(\widehat{\v}_\varepsilon)\right)\wedge(\star\d\varphi).
  \end{split}
 \]
 The total contribution of the boundary terms vanishes, because each edge appears in the sum twice, with opposite orientations.
 Then, by the H\"older inequality, we obtain
 \[
  \|\star\d\widehat{\j}_\varepsilon(\widehat{\v}_\varepsilon) - \star\d\j(\widehat{\v}_\varepsilon)\|_{\fflat}
  \leq \|\widehat{\ggamma}_\varepsilon - \ggamma\circ P\|_{L^\infty(\widehat{M}_\varepsilon)}
  \norm{\widehat{\v}_\varepsilon}_{L^2(\widehat{M}_\varepsilon)}
  \norm{\nablae\widehat{\v}_\varepsilon}_{L^2(\widehat{M}_\varepsilon)}.
 \]
 Using that~$|\widehat{\v}_\varepsilon|\leq 1$,
 that $\|\widehat{\ggamma}_\varepsilon - \ggamma\circ P\|_{L^\infty(\widehat{M}_\varepsilon)}\leq C\varepsilon$
 (as a consequence of~\eqref{hp:quasiuniform}) and~\eqref{log_norm}, we conclude that
 \begin{equation} \label{mu-A2}
  \|\star\d\widehat{\j}_\varepsilon(\widehat{\v}_\varepsilon) - \star\d\j(\widehat{\v}_\varepsilon)\|_{\fflat} \leq C \varepsilon|\log\varepsilon|^{1/2}.
 \end{equation}

 \medskip
 \noindent\emph{Analysis of~$A_3$.}
 Set~$\omega := \j(\w_\varepsilon)$, so~$\omega$ is a $1$-form on~$M$
 and~$\widehat{P}_\varepsilon^*(\omega) = \j(\widehat{\v}_\varepsilon)$.
 Since the sets~$\widehat{P}_\varepsilon(T)$ for~$T\in\TT_\varepsilon$ define a Borel partition of~$M$ up to sets of measure zero,
 we can write
 \[
  \langle \star\d\j(\w_\varepsilon) - \star\d\j(\widehat{\v}_\varepsilon), \, \varphi\rangle
  = \sum_{T\in\TT_\varepsilon} \left( \int_{\widehat{P}_\varepsilon(T)} \varphi \, \d\omega  
  - \int_{T} \varphi \, \widehat{P}_\varepsilon^*(\d\omega)  \right).
 \]
 Thanks to the assumption~\eqref{hp:bilipschitz}, $\widehat{P}_\varepsilon$ induces a bilipschitz equivalence of~$T$ onto its image.
 Therefore, by applying the area formula to the first integral in the right-hand side, we deduce
 \[
  \begin{split}
   \langle \star\d\j(\w_\varepsilon) - \star\d\j(\widehat{\v}_\varepsilon), \, \varphi\rangle
   &= \sum_{T\in\TT_\varepsilon} \int_{T} \left( \widehat{P}_\varepsilon^*(\varphi \, \d\omega) - \varphi \, \widehat{P}_\varepsilon^*(\d\omega) \right) \\
   &= \sum_{T\in\TT_\varepsilon} \int_{T} \left( \varphi\circ\widehat{P}_\varepsilon - \varphi \right) \, \widehat{P}_\varepsilon^*(\d\omega)
   = \int_{\widehat{M}_\varepsilon} \left( \varphi\circ\widehat{P}_\varepsilon - \varphi \right) \, \d\j(\widehat{\v}_\varepsilon) .
  \end{split}
 \]
 The H\"older inequality and Lemma~\ref{lemma:discrete_jacobian} then yield
 \[
  \langle \star\d\j(\w_\varepsilon) - \star\d\j(\widehat{\v}_\varepsilon), \, \varphi\rangle 
  \leq C\dist(\widehat{M}_\varepsilon, \, M)\|\nablae\varphi\|_{L^\infty(\widehat{M}_\varepsilon)}
  \left(\|\widehat{\v}_\varepsilon\|_{L^2(\widehat{M}_\varepsilon)}^2 + \|\nablae\widehat{\v}_\varepsilon\|_{L^2(\widehat{M}_\varepsilon)}^2\right)
 \]
 whence, by applying~\eqref{hp:quasiuniform} and~\eqref{log_norm}, we conclude
 \begin{equation} \label{mu-A3}
  \|\star\d\j(\widehat{\v}_\varepsilon) - \star\d\j(\w_\varepsilon)\|_{\fflat}
  \leq C\varepsilon|\log\varepsilon|.
 \end{equation}

 \medskip
 \noindent\emph{Analysis of~$A_4$.}
 Reminding the definition~\eqref{u_eps} of~$\tilde{\u}_\varepsilon$ and Lemma~\ref{lemma:proiection-pointwise}, we have
 \begin{equation} \label{mu4,1}
  \|\w_\varepsilon - \tilde{\u}_\varepsilon\|_{L^\infty(M)} = 
  \|(\w_\varepsilon, \, \ggamma)\|_{L^\infty(M)} \leq C\varepsilon.
 \end{equation}
 On the other hand, using~\eqref{log_norm} and the assumption~\eqref{hp:bilipschitz}, we compute that
 \begin{equation} \label{mu4,2}
  \|\nablas\w_\varepsilon\|_{L^2(M)} + \|\nablas\tilde{\u}_\varepsilon\|_{L^2(M)} \leq C|\log\varepsilon|^{1/2}.
 \end{equation}
 We combine~\eqref{jacobian_flat} in Lemma~\ref{lemma:continuity_jacobian} with~\eqref{mu4,1} and~\eqref{mu4,2} to obtain
 \begin{equation} \label{mu-A4}
  \|\star\d\j(\w_\varepsilon) - \star\d\j(\tilde{\u}_\varepsilon)\|_{\fflat} \leq C\varepsilon|\log\varepsilon|^{1/2}.
 \end{equation}
 
 \medskip
 \noindent\emph{Analysis of~$A_5$.}
 From the definition~\eqref{u_eps} of~$\u_\varepsilon$ and~\eqref{mu4,2}, we compute
 \begin{equation} \label{mu5,1}
  \|\nablas{\u}_\varepsilon\|_{L^2(M)} \leq 2t_\varepsilon^{-1} \|\nablas\tilde{\u}_\varepsilon\|_{L^2(M)} 
  \leq Ct_\varepsilon^{-1} |\log\varepsilon|^{1/2}.
 \end{equation}
 On the other hand, thanks to~\eqref{mu4,1} 
 we obtain
 \[
  |\u_\varepsilon - \tilde{\u}_\varepsilon|^2 \leq \left(1 - |\tilde{\u}_\varepsilon|\right)^2 
  \leq \left(1 - |\w_\varepsilon|\right)^2 +C\varepsilon
  \leq \left(1 - |\w_\varepsilon|^2\right)^2 +C\varepsilon.
 \]
 By integrating both sides of the inequality, and making a change of variable 
 we deduce
 \begin{equation} \label{mu5,2}
  \|\u_\varepsilon - \tilde{\u}_\varepsilon\|_{L^2(M)}^2 \leq 
  \norm{1 - |\widehat{\v}_\varepsilon|^2}^2_{L^2(\widehat{M}_\varepsilon)} + C\varepsilon\H^2(\widehat{M}_\varepsilon)
  \leq C\varepsilon^2|\log\varepsilon|.
 \end{equation}
 For the last inequality, we have used~\eqref{log_norm} and the fact that~$\H^2(\widehat{M}_\varepsilon)\leq C$, which follows from~\eqref{hp:bilipschitz}.
 Thus, by applying~\eqref{jacobian_flat} (Lemma~\ref{lemma:continuity_jacobian}), with the help of~\eqref{mu4,2}, \eqref{mu5,1} and~\eqref{mu5,2}
 we deduce that
 \begin{equation} \label{mu-A5}
  \|\star\d\j(\tilde{\u}_\varepsilon) - \star\d\j(\u_\varepsilon)\|_{\fflat} \leq C t_\varepsilon^{-1}\varepsilon|\log\varepsilon|.
 \end{equation}

 \medskip
 \noindent\emph{Analysis of~$A_6$.}
 Arguing as in the proof of~\eqref{mu-A1}, and using~\eqref{jacobian_L1} in Lemma~\ref{lemma:continuity_jacobian}
 instead of Lemma~\ref{lemma:discrete_jacobian}, we obtain
 \begin{equation} \label{mu-A6}
  \|\star\d\j(\u_\varepsilon) - \mu_\varepsilon(\v_\varepsilon)\|_{\fflat}\leq\varepsilon|\log\varepsilon|.
 \end{equation}
 
 Now, the proposition follows by combining~\eqref{mu-A1}, \eqref{mu-A2}, \eqref{mu-A3}, \eqref{mu-A4}, \eqref{mu-A5} and~\eqref{mu-A6}.
\end{proof}

\section{The zero-order $\Gamma$-convergence: emergence of defects}
\label{sec:provaGamma0}
%
%
%

\subsection{Localized lower bounds for the energy}

Thanks to Proposition~\ref{prop:mu}, the compactness of the sequence $\widehat{\mu}_\varepsilon(\v_\varepsilon)$ is 
equivalent to the compactness of~$\mu_\varepsilon(\v_\varepsilon)$.
The latter is defined in terms of the fields~$\w_\varepsilon\in W^{1, \infty}(M; \,\R^3)$, given by~\eqref{u_eps},
which interpolate continuously the discrete fields~$\v_\varepsilon$ but are not necessarily tangent nor unit-valued.
We discuss now a localized lower bound for the Dirichlet energy of~$\w_\varepsilon$.
Similar results are well-known in the continuum Ginzburg-Landau setting, where they play a major r\^ole 
(see, e.g.,~\cite[Theorems~2.1 and~4.1]{Jerrard} and~\cite[Theorem~1]{Sandier}),
and are also available for the discrete XY-energy~\cite[Proposition~3.2]{ADGP}. 
Given a point~$x_0\in M$ and a radius~$\rho >0$, we denote by~$B_\rho(x_0)$ the geodesic ball of centre~$x_0$ and radius~$\rho$.
We follow the approach in~\cite{Jerrard}. We define
\begin{equation} \label{alpha_eps}
 \alpha_\varepsilon := (1 - C\varepsilon)^{-2} \, \underset{x\in M}{\infess} \, \inf_{\mathbf{X}\in \T_x M, \, |\mathbf{X}|=1}
 \left(\A_\varepsilon^{-1}(x)\mathbf{X}, \, \mathbf{X}\right) (\det \A_\varepsilon(x))^{1/2}
\end{equation}
and, for~$0 < \varepsilon < \rho$ and~$d\in\Z$, we let
\begin{equation} \label{lambda_eps}
 \lambda_\varepsilon(\rho, \, d) := \alpha_\varepsilon \min_{0 \leq m \leq 1}\left\{ \frac{\pi |d|}{\rho + C\rho^2} m^2
 \vee \frac{C}{\varepsilon} (1 - m)^3\right\} \! .
\end{equation}
The constant~$C$, which will be selected below, does not depend on~$\varepsilon$, $\rho$ and~$d$.
Recall that, given a vector field~$\v_\varepsilon\in\T(\TT_\varepsilon;\, \S^2)$, we let
$\w_\varepsilon :=\widehat{\v}_\varepsilon\circ\widehat{P}_\varepsilon^{-1}$ and we denote by~$\u_\varepsilon$ 
the vector field defined by~\eqref{u_eps}.

\begin{lemma} \label{lemma:circle}
 There exist positive constants~$R_*$ and~$\varepsilon_*$ with the following property:
 for any~$\varepsilon\in(0, \, \varepsilon_*]$, any~$x_0\in M$, any~$\rho\in(\varepsilon, \, R_*]$
 and any~$\v_\varepsilon\in\T_\varepsilon(\TT_\varepsilon; \, \S^2)$ such that~$|\w_\varepsilon|\geq 1/2$
 on~$\partial B_\rho(x_0)$, there holds
 \begin{equation} \label{circle}
  \frac12 |\w_\varepsilon|_{W^{1,2}_\varepsilon(\partial B_\rho(x_0))}^2 \geq 
  \lambda_\varepsilon\left(\rho, \, \ind(\u_\varepsilon, \, \partial B_\rho(x_0))\right).
 \end{equation}
 Moreover, there holds
 \begin{equation} \label{lambda_bound}
  \lambda_\varepsilon(\rho, \, d) \geq \frac{(1 - C\varepsilon) \pi|d|}{\rho + C\rho^2} - C\varepsilon^{1/3}|d|^{1/3}\rho^{-4/3}
 \end{equation}
 for any~$\rho > \varepsilon$, $d\in\Z$ and some constant~$C$ which does not depend on~$\varepsilon$, $\rho$, $d$.
\end{lemma}

\begin{proof}[Proof of Lemma~\ref{lemma:circle}]
 This argument is adapted from~\cite[Theorem~2.1]{Jerrard}. Throughout the proof, the symbol~$C$ denotes several constants that
 do not depend on~$\varepsilon$ or~$x_0\in M$, but possibly on~$R_*$,~$\varepsilon_*$. We start by proving~\eqref{circle}.
 \setcounter{step}{0} 
 \begin{step}
  There exist positive numbers~$R_*$ and~$C$ such that, for any~$x_0\in M$ and any~$0 < \rho \leq R_*$, there holds
  \begin{equation} \label{circle0}
   \H^1(\partial B_\rho(x_0)) \leq 2\pi\rho + C\rho^2.
  \end{equation}
  Indeed, thanks to the area formula, the left-hand side is bounded by~$\Lip(\exp_{x_0|D_\rho})\H^1(\partial D_\rho)$, 
  where $\exp_{x_0}\colon\T_{x_0}M\to M$ is the exponential map and~$D_\rho\subseteq\T_{x_0}M$ is the disk of radius~$\rho$ centred at the origin.
  Since~$\D(\exp_{x_0}) = \Id_{\T_{x_0}M}$, by a Taylor expansion we see that $\Lip(\exp_{x_0|D_\rho}) \leq 1-C\rho$
  if~$\rho$ is small enough. By a compactness argument, the numbers~$R_*$ and~$C$ can be chosen uniformly with respect to~$x_0$.
 \end{step}
 \begin{step}
  Take a point~$x_0\in M$ and a field~$\v_\varepsilon\in T(\TT_\varepsilon;\,\S^2)$ such that~$|\w_\varepsilon| \geq 1/2$ on~$\partial B_\rho(x_0)$.
  Since~$x_0$ will be fixed for the rest of the proof, we will omit it from the notation.
  Thanks to Lemma~\ref{lemma:proiection-pointwise} and to~\eqref{t_eps},~\eqref{u_eps}
  we have that~$|\tilde{\u}_\varepsilon|\geq 1/4$ and~$\u_\varepsilon = \tilde{\u}_\varepsilon/|\tilde{\u}_\varepsilon|$
  on~$\partial B_\rho$ provided that~$\varepsilon$ is small enough, say, smaller that some number~$\varepsilon_*$. Then, we compute
  \[
   \begin{split}
    |\w_\varepsilon| \abs{\left(\ggamma, \, \u_\varepsilon\times\nabla_\ttau\u_\varepsilon\right)}
    &= \frac{|\w_\varepsilon|}{|\tilde{\u}_\varepsilon|^2} \abs{\left(\ggamma, \, \tilde{\u}_\varepsilon\times\nabla_\ttau\tilde{\u}_\varepsilon\right)} \\
    &= \frac{|\w_\varepsilon|}{|\tilde{\u}_\varepsilon|^2} \abs{\left(\ggamma, \, \w_\varepsilon\times\nabla_\ttau\w_\varepsilon\right) 
    - (\w_\varepsilon, \, \ggamma) \left(\ggamma, \, \w_\varepsilon\times\nabla_\ttau\ggamma\right)} \\
    &\leq \frac{|\w_\varepsilon|^2}{|\tilde{\u}_\varepsilon|^2} \left(\abs{\nablas\w_\varepsilon} + C|(\w_\varepsilon, \, \ggamma)|\right).
   \end{split}
  \]
  Using~\eqref{u_eps} and the fact that~$|\tilde{\u}_\varepsilon|\geq 1/4$, we can bound the ratio~$|\w_\varepsilon|/|\tilde{\u}_\varepsilon|$
  in terms of~$|(\w_\varepsilon, \, \ggamma)|$, which in turns is bounded by~$C\varepsilon$, due to  Lemma~\ref{lemma:proiection-pointwise}. This yields
  \begin{equation*}
   |\w_\varepsilon| \abs{\left(\ggamma, \, \u_\varepsilon\times\nabla_\ttau\u_\varepsilon\right)}
   \leq (1 + C\varepsilon) \left(\abs{\nablas\w_\varepsilon} + C\varepsilon\right) \! .
  \end{equation*}
  After a rearrangement, and using again that~$|\w_\varepsilon|\geq 1/2$, we obtain
  \begin{equation} \label{circle1}
   |\nablas\w_\varepsilon|^2 \geq (1 +C\varepsilon)^{-2} |\w_\varepsilon|^2 
   \left(\abs{\left(\ggamma, \, \u_\varepsilon\times\nabla_\ttau\u_\varepsilon\right)} - C\varepsilon\right)^2.
  \end{equation}
 \end{step}
 \begin{step}
  Thanks to Lemma~\ref{lemma:metric_distorsion}, there exists~$\varepsilon_*$ such that
  the quantity~$\alpha_\varepsilon$ defined in~\eqref{alpha_eps} satisfies
  \begin{equation} \label{circle2}
   \alpha_\varepsilon\geq 1 - C\varepsilon > 0 \qquad \textrm{ for any } 0 < \varepsilon\leq\varepsilon_*.
  \end{equation}
  Fix $\varepsilon$ and~$\rho$ such that $0 < \varepsilon\leq\varepsilon_*$ and $\varepsilon < \rho \leq R_*$. 
  By definition~\eqref{Sobolev_eps} of the $W^{1,2}_\varepsilon$-seminorm, there holds
  \begin{equation} \label{circle2.5}
   \frac12 |\w_\varepsilon|^2_{W^{1,2}_\varepsilon(\partial B_\rho)}
   \geq \frac{\alpha_\varepsilon (1- C\varepsilon)^2}{2} \int_{\partial B_\rho} \abs{\nablas\w_\varepsilon}^2 \,\d s .
  \end{equation}
  Set~$m(\rho) := \min_{\partial B_\rho}|\w_\varepsilon|$ (note that~$m(\rho)\in [1, \, 1/2]$), and let~$\ttau$ be a unit tangent field on~$\partial B_\rho$.
  Using~\eqref{circle1} and the definition~\eqref{prejacobian} of~$\j$, we obtain
  \[
    \frac12 |\w_\varepsilon|^2_{W^{1,2}_\varepsilon(\partial B_\rho)}
    \geq \frac{\alpha_\varepsilon m^2(\rho)}{2} \int_{\partial B_\rho}
    \abs{\langle \j(\u_\varepsilon), \, \ttau\rangle - C\varepsilon}^2 \,\d s .
  \]
  By applying Jensen inequality, we deduce
  \begin{equation*}
   \begin{split}
    \frac12 |\w_\varepsilon|^2_{W^{1,2}_\varepsilon(\partial B_\rho)}
    &\geq \frac{\alpha_\varepsilon \, m^2(\rho)}{2 \H^1(\partial B_\rho)} 
    \abs{\int_{\partial B_\rho}\j(\u_\varepsilon) - C\varepsilon\H^1(\partial B_\rho)}^2 \\
    &\stackrel{\eqref{circle0}}{\geq} \frac{\alpha_\varepsilon \, m^2(\rho)}{4\pi\rho + C\rho^2} 
    \abs{\int_{\partial B_\rho}\j(\u_\varepsilon) - C\varepsilon\H^1(\partial B_\rho)}^2.
   \end{split}
  \end{equation*}
  Using Lemma~\ref{lemma:dj}, and arguing as in the proof of~\eqref{mu2}, we can evaluate the integral of~$\j(\u_\varepsilon)$ 
  in terms of~$d := \ind(\u_\varepsilon, \, \partial B_\rho)$ and the Gauss curvature~$G$:
  \[
   \begin{split}
    \frac12 |\w_\varepsilon|^2_{W^{1,2}_\varepsilon(\partial B_\rho)}
    \geq \frac{\alpha_\varepsilon \, m^2(\rho)}{4\pi\rho + C\rho^2} \abs{2\pi d - \int_{B_\rho}G - C\varepsilon\H^1(\partial B_\rho)}^2.
   \end{split}
  \] 
  Now, the Gauss curvature~$G$ is bounded and~$\varepsilon <\rho\leq R_*$, so the integral of~$G$ is uniformly bounded by~$CR_*^2$,
  while~$\varepsilon\H^1(\partial B_\rho) \leq CR_*^2$. 
  Thus, reducing the value of~$R_*$ if necessary, we can assume that 
  \[
   u := \frac{1}{2\pi}\abs{\int_{B_\rho}G + C\varepsilon\H^1(\partial B_\rho)} < \frac{1}{2}.
  \]
  If~$|d|>1$, then $|d-u|^2 \geq d^2 - 2u|d| \geq |d|$ and hence
  \begin{equation} \label{circle3}
   \frac12 |\w_\varepsilon|^2_{W^{1,2}_\varepsilon(\partial B_\rho)}
   \geq \frac{\alpha_\varepsilon \pi \, m^2(\rho)}{\rho + C\rho^2} |d|.
  \end{equation}
  An elementary computation, based on the fact that~$u \leq C\rho^2$, shows that the same inequality is satisfied also if~$|d| = 1$,
  provided that we take a larger constant~$C$ in the denominator, and we reduce again the value of~$R_*$ if necessary.
  Finally, \eqref{circle3} is trivially satisfied if~$d = 0$.
 \end{step}
 \begin{step} \label{step:potential}
  Suppose that~$m(\rho) < 1$, and let~$x_\rho\in\partial B_\rho$ be such that~$|\w_\varepsilon(x_\rho)| = m(\rho)$. 
  We have~$|\w_\varepsilon| \leq (1+m(\rho))/2$ on~$B_{\rho^\prime}(x_\rho)$, where
  \[
   \rho^\prime := \frac{1 - m(\rho)}{2\Lip(\w_\varepsilon)}.
  \]
  Now, $\we$ has Lipschitz constant $\Lip(\w_\varepsilon)\le C\eps^{-1}$, thanks to Lemma \ref{lemma:norm-interpolant}. Thus, 
  since we have assumed that~$\rho > \varepsilon$, we conclude that
  \[
   \H^1\left(\partial B_\rho \cap B_{\rho^\prime}(x_\rho)\right)\geq C\rho^\prime\geq C\varepsilon\left(1-m(\rho)\right) \! .
  \]
  Thus, by applying Lemma~\ref{lemma:GL-pointwise}, we estimate
  \begin{equation} \label{circle4}
   \frac12 |\w_\varepsilon|^2_{W^{1,2}_\varepsilon(\partial B_\rho)} 
   \geq \frac{C}{\varepsilon^2} \int_{\partial B_\rho\cap B_{\rho^\prime}(x_\rho)} \left(1 - |\w_\varepsilon|^2\right)^2 \, \d s
   \geq \frac{C}{\varepsilon}\left(1 - m(\rho)\right)^3.
  \end{equation}
  Note that~\eqref{circle4} trivially holds when~$m(\rho)=1$. Now, \eqref{circle} follows by combining~\eqref{circle2}, ~\eqref{circle3} and~\eqref{circle4}.
 \end{step}
 \begin{step}[Proof of~\eqref{lambda_bound}]
  The function~$f\colon m\in [0, \, 1]\mapsto Am^2 \vee B(1-m)^3$ achieves its minimum value at the point~$m_0$ such that~$Am_0^2 = B(1-m_0)^3$.
  From this equality, we deduce that
  \[
   \frac{B}{A} (1-m_0)^3 = m_0^2\leq 1,
  \]
  whence~$m_0 \geq 1 - (A/B)^{1/3}$ and~$f(m_0) \geq A (1 - 2(A/B)^{1/3})$.
  Substituting for~$A$ and~$B$ the expressions in~\eqref{lambda_eps}, and using~\eqref{alpha_eps}, yields~\eqref{lambda_bound}. \qedhere
 \end{step}
\end{proof}

%

Following Jerrard~\cite{Jerrard}, it will be useful to reformulate the lower bound~\eqref{circle}
in terms of a function~$\Lambda_\varepsilon$, defined by
\begin{equation} \label{Lambda_eps}
 \Lambda_\varepsilon(r) := \int_0^r \lambda_\varepsilon(\rho, \, 1)\wedge \frac{C_*}{\varepsilon} \, \d\rho
 \qquad \textrm{for } r > 0.
\end{equation}
We first collect a few properties of~$\Lambda_\varepsilon$ (see also~\cite[Proposition~3.1]{Jerrard}).

\begin{lemma} \label{lemma:Lambda}
 The function~$\Lambda_\varepsilon$ satisfies
 \begin{equation} \label{prop:lambda}
  \Lambda_\varepsilon(r + s) \leq \Lambda_\varepsilon(r) + \Lambda_\varepsilon(s), \qquad \Lambda_\varepsilon(r) \leq \Lambda_\varepsilon(s),
  \qquad \frac{\Lambda_\varepsilon(r)}{r} \geq \frac{\Lambda_\varepsilon(s)}{s}
 \end{equation}
  for any~$0 < r \leq s$. Moreover, there holds
 \begin{equation} \label{Lambda_bound}
  \Lambda_\varepsilon(r) \geq (1 - C\varepsilon)\pi\log\frac{r}{\varepsilon} - C
 \end{equation}
 for any~$r \in(\varepsilon, \, R_*]$ (where~$R_*$ is given by Lemma~\ref{lemma:circle}) and some $\varepsilon$-independent constant~$C$.
\end{lemma}
\begin{proof}
 It is clear by the definition~\eqref{lambda_eps} that~$\lambda_\varepsilon$ is positive and decreasing;
 then~\eqref{prop:lambda} follows by elementary calculus. As for the lower bound~\eqref{Lambda_bound}, Equation~\eqref{lambda_bound} implies that
 \[
  \lambda_\varepsilon(\rho, \, 1) \wedge \frac{C}{\varepsilon} \geq \frac{(1 - C\varepsilon) \pi}{\rho + C\rho^2} - C\varepsilon^{1/3}\rho^{-4/3}
 \]
 for any~$\rho\in(c_1\varepsilon, \, R_*]$, for some constant~$c_1 > 0$. By integrating both sides of this inequality
 with respect to~$\rho\in (c_1\varepsilon, \, r)$, we deduce
 \[
  \begin{split}
   \Lambda_\varepsilon(r) &\geq (1-C\varepsilon)\pi \left\{ \log\frac{r}{c_1\varepsilon} 
   + \log\frac{C\varepsilon + 1}{Cr + 1} \right\} + C\varepsilon^{1/3} \left(R^{1/3} - r^{1/3}\right) \\
   &\stackrel{c_1\varepsilon < r < R_*}{\geq} (1-C\varepsilon)\pi \log\frac{r}{\varepsilon} - C,
   \end{split}
 \]
 where the constant $C$ in the right-hand side depends only on~$c_1$ and~$R_*$.
 If~$c_1\leq1$, then the lemma follows immediately. Otherwise, we note that, by choosing~$C$ large enough, 
 the right-hand side of~\eqref{Lambda_bound} can be made non-positive for every~$r \in [\varepsilon, \, c_1\varepsilon]$,
 so that~\eqref{Lambda_bound} holds trivially.
\end{proof}

We state a lower bound for the energy on annuli in terms of the function~$\Lambda_\varepsilon$.

\begin{lemma} \label{lemma:Lambda_annulus}
 For any~$\varepsilon\in(0, \, \varepsilon_*]$, any~$x_0\in M$, any~$\varepsilon < r < R \leq R_*$
 (where~$\varepsilon_*$,~$R_*$ are given by Lemma~\ref{lemma:circle}) and any~$\v_\varepsilon\in\T_\varepsilon(\TT_\varepsilon; \, \S^2)$ 
 such that~$|\w_\varepsilon|\geq 1/2$ on~$A_{r, R} := B_{R}(x_0)\setminus B_{r}(x_0)$, there holds
 \begin{equation*}
  \frac12 |\w_\varepsilon|_{W^{1,2}_\varepsilon(A_{r, R})}^2 \geq 
  |d| \left\{\Lambda_\varepsilon\left(\frac{R}{|d|}\right) - \Lambda_\varepsilon\left(\frac{r}{|d|}\right)\right\},
 \end{equation*}
 where $d := \ind(\u_\varepsilon, \, \partial B_R(x_0))$.
\end{lemma}

The proof of this lemma follows by integrating the lower bound~\eqref{circle}; see~\cite[Proposition~3.2]{Jerrard} for details.

\subsection{The ball construction}
\label{sect:ball}

In this section, we recall the ``ball construction'' as presented by Jerrard~\cite{Jerrard}
(a similar construction was independently introduced by Sandier~\cite{Sandier}). 
In contrast with~\cite{Jerrard}, our lower bound (Lemma~\ref{lemma:Lambda_annulus}) is only valid for annuli with outer radius~$\leq R_*$,
so we need to make sure that this constraint is preserved by the construction.

Throughout this section, we fix a sequence of discrete fields~$\v_\varepsilon\in \T(\TT_\varepsilon; \, \S^2)$,
for~$\varepsilon\in (0, \, \varepsilon_*]$, that satisfies the logarithmic energy bound~\eqref{H}.
We define the set~$S_\varepsilon := \{x\in M\colon |\w_\varepsilon(x)|\leq 1/2\}$ and the measure
\begin{equation} \label{nu_eps}
 \nu_\varepsilon(\v_\varepsilon) := \frac{1}{2\pi} \left(\mu_\varepsilon(\v_\varepsilon) + G \,\d S \right) \! ,
\end{equation}
where~$\mu_\varepsilon(\v_\varepsilon)$ is given by~\eqref{mu}. Since~$(\v_\varepsilon)$ is fixed, throughout this section we write~$\nu_\varepsilon$
instead of~$\nu_\varepsilon(\v_\varepsilon)$. If~$E\subseteq M$ is a Borel set with~$\partial E\subseteq M\setminus S_\varepsilon$,
Equation~\eqref{mu2} implies that~$\nu_\varepsilon(E) = \ind(\u_\varepsilon, \, \partial E)$.
The sequence~$(\nu_\varepsilon)$ is precompact in the flat topology if and only if~$(\mu_\varepsilon(\v_\varepsilon))$ is,
and hence (by Proposition~\ref{prop:mu}), if and only if~$(\widehat{\mu}_\varepsilon(\v_\varepsilon))$ is.
We will also need the following notation: given a closed ball~$B$, we denote by~$\rad(B)$ its radius.
If~$\mathscr{B}$ is a finite collection of closed balls, we set 
$\spt\mathscr{B} := \cup_{B\in\mathscr{B}} B$.

\begin{lemma} \label{lemma:beta}
 There exists an $\varepsilon$-independent constant~$\beta$ such that, for any~$T\in\TT_\varepsilon$,
 \[
  P(T)\cap S_\varepsilon\neq\emptyset \quad \textrm{implies} \quad 
  \frac12 \abs{\w_\varepsilon}^2_{W^{1,2}_\varepsilon(P(T))} \geq \beta.
 \]
\end{lemma}
\begin{proof}
 Suppose that there is a point~$x_0\in P(T)$ such that~$|\w_\varepsilon(x_0)|\leq 1/2$.
 Arguing as in the proof of Lemma~\ref{lemma:circle}, Step~\ref{step:potential},
 we deduce that~$|\w_\varepsilon| \leq 3/4$ on a ball~$B_{\rho^\prime}(x_0)$ with~$\rho^\prime\geq C\varepsilon$. 
 Then, using also the assumptions~\eqref{hp:quasiuniform} and~\eqref{hp:bilipschitz},
 we see that~$\H^2(P(T)\cap B_{\rho^\prime}(x_0))\geq C\varepsilon$ and hence, by Lemma~\ref{lemma:GL-pointwise}, we estimate
  \begin{equation*}
   \frac12 |\w_\varepsilon|^2_{W^{1,2}_\varepsilon(P(T))} 
   \geq \frac{C}{\varepsilon^2} \int_{P(T)\cap B_{\rho^\prime}(x_\rho)} \left(1 - |\w_\varepsilon|^2\right)^2 \, \d S
   \geq C. \qedhere
  \end{equation*}
\end{proof}
\begin{lemma} \label{lemma:ni_triangle}
 There exists an~$\varepsilon$-independent constant~$C$ such that, for any~$T\in\TT_\varepsilon$, there holds $|\nu_\varepsilon(P(T))|\leq C$.
\end{lemma}
\begin{proof}
 We have that
 \[
  |\nu_\varepsilon(P(T))| \stackrel{\eqref{nu_eps}}{\leq} \abs{\mu_\varepsilon(\v_\varepsilon)[P(T)]} + C\varepsilon^2
  \stackrel{\eqref{mu}}{=} \abs{\int_{\partial(P(T))} \j(\u_\varepsilon)} + C\varepsilon^2 
 \]
 (we have used that the Gauss curvature is bounded and the surface area of~$P(T)$ is $\leq C\varepsilon^2$,
 which follows from~\eqref{hp:quasiuniform} and~\eqref{hp:bilipschitz}).
 Using now the definition~\eqref{u_eps} of~$\u_\varepsilon$, we compute that
 $\j(\u_\varepsilon) = \eta_\varepsilon(|\tilde{\u}_\varepsilon|)\j(\tilde{\u}_\varepsilon)$ and hence
 $|\j(\u_\varepsilon)| \leq \Lip(\tilde{\u}_\varepsilon)$. Combining~\eqref{u_eps} with the Lipschitz bound~\eqref{eq:lip_we} for~$\w_\varepsilon$,
 we see that $\Lip(\tilde{\u}_\varepsilon) \leq \Lip(\w_\varepsilon) \leq C\varepsilon^{-1}$. Thus
 \[
  |\nu_\varepsilon(P(T))| \leq C\varepsilon^{-1}\H^1\left(\partial P(T)\right) + C\varepsilon^2 \leq C,
 \]
 where we have used that $\H^1(\partial P(T)) \leq C\varepsilon$, due to~\eqref{hp:quasiuniform} and~\eqref{hp:bilipschitz}.
\end{proof}

For any $T\in\TT_\varepsilon$ such that~$P(T)\cap S_\varepsilon\neq\emptyset$,
we consider the smallest closed ball~$B$ of centre~$P(x_T)$
such that~$P(T)\subseteq B$. Let~$\mathscr{B}_\varepsilon$ be the collection of such balls.
Thanks to the assumption~\eqref{hp:quasiuniform}, any~$B\in\mathscr{B}_\varepsilon$ satisfies
\begin{equation} \label{radius_B0}
 C^{-1}\varepsilon \leq \rad(B) \leq C\varepsilon.
\end{equation}
Moreover, each ball~$B\in\mathscr{B}_\varepsilon$ intersects~$P(T)$ for at most~$C$ triangles~$T\in\TT_\varepsilon$, 
where~$C$ is an $\varepsilon$-independent constant. Therefore, from~\eqref{radius_B0} and Lemma~\ref{lemma:ni_triangle}
we deduce that
\begin{equation} \label{s_eps}
 C^{-1}\varepsilon \leq s_\varepsilon := \min_{B\in\mathscr{B}_\varepsilon} \frac{\rad(B)}{|\nu_\varepsilon(B)|} \leq C\varepsilon.
\end{equation}
(To prove the upper bound, note that~$\spt\mathscr{B}_\varepsilon\supseteq\spt(\nu_\varepsilon)$ 
and that~$|\nu_\varepsilon(B)|\geq 1$ as soon as~$B\cap\spt(\nu_\varepsilon)\neq\emptyset$. 
Here, we are assuming WLOG that $\nu_\varepsilon\not\equiv 0$, otherwise~$s_\varepsilon = +\infty$).
Finally, as a consequence of Lemma~\ref{lemma:beta} and the energy bound~\eqref{H}, we obtain
\begin{equation} \label{cardinal_B0}
 \#(\mathscr{B}_\varepsilon) \leq  C|\log\varepsilon|.
\end{equation}
The following proposition is adapted from~\cite[Proposition~4.1]{Jerrard} (see also~\cite[Proposition~5.4]{SandierSerfaty11}).

\begin{prop} \label{prop:ball}
 There exists an ($\varepsilon$-independent) positive constant~$C$ such that, for any $s\in [s_\varepsilon, \, C R_*\#(\mathscr{B}_\varepsilon)^{-1}]$,
 there exists a family of pairwise disjoint, closed balls~$\mathscr{B}_\varepsilon(s)$ with the following properties.
 \begin{enumerate}[label=(\roman*)]
  \item \label{item:ball_monotone} $\spt\mathscr{B}_\varepsilon\subseteq\spt\mathscr{B}_\varepsilon(s)\subseteq\spt\mathscr{B}_\varepsilon(t)$ 
  for any~$s_\varepsilon \leq s \leq t \leq C R_*\#(\mathscr{B}_\varepsilon)^{-1}$.
  \item \label{item:ball_energy} For any~$B\in\mathscr{B}_\varepsilon(s)$, there holds
  \[
   \frac12 \abs{\w_\varepsilon}^2_{W^{1,2}_\varepsilon(B\setminus\spt\mathscr{B}_\varepsilon)} \geq \frac{\rad(B)}{s}
   \left(\Lambda_\varepsilon(s) - \Lambda_\varepsilon(s_\varepsilon)\right) \! .
  \]
  \item \label{item:ball_lower_radius} For any~$B\in\mathscr{B}_\varepsilon(s)$, there holds~$\rad(B)\geq s|\nu_\varepsilon(B)|$.
  \item \label{item:ball_radius} There holds
  \[
   \sum_{B\in\mathscr{B}_\varepsilon(s)}\rad(B) \leq \frac{s}{s_\varepsilon} \sum_{B\in\mathscr{B}_\varepsilon} \rad(B).
  \]
 \end{enumerate}
\end{prop}
\begin{proof}[Sketch of the proof]
 If the balls in~$\mathscr{B}_\varepsilon$ are not pairwise disjoint, the construction starts with a \emph{merging} phase, 
 that is, we select a pair of balls~$B$, $B^\prime\in\mathscr{B}_\varepsilon$ such that~$B\cap B^\prime\neq\emptyset$ 
 and we replace them with a new ball~$B^{\mathrm{new}}$ such that~$B^{\mathrm{new}}\supseteq B\cup B^\prime$,
 $\rad(B^{\mathrm{new}}) = \rad(B) + \rad(B^\prime)$. We repeat this operation until we obtain a collection of pairwise disjoint balls,
 which we call~$\mathscr{B}_\varepsilon^\prime$. If all the balls in the original collection~$\mathscr{B}_\varepsilon$ were pairwise disjoint,
 then~$\mathscr{B}_\varepsilon = \mathscr{B}_\varepsilon^\prime$.
 Set~$\mathscr{B}_\varepsilon(s_\varepsilon) := \mathscr{B}_\varepsilon$, so (i), (iii), (iv) are trivially satisfied and~(ii)
 is also satisfied, because of~\eqref{s_eps}.
 
 Now we perform an expansion phase, i.e. we let the parameter~$s$ grow continuously, and we let the ``minimizing balls''
 (i.e., the balls~$B$ such that~$\rad(B) = s|\nu_\varepsilon(B)|$) grow, leaving the other unchanged.
 More precisely, if~$\mathscr{B}_\varepsilon^\prime = \{B_i\}_{i=1}^k$ and~$x_i$ is the centre of~$B_i$,
 then the elements of~$\mathscr{B}_\varepsilon(s)$ are defined by
 \begin{equation} \label{expansion}
  B_i(s) := \begin{cases}
             B_i & \textrm{if } \rad(B_i) > s|\nu_\varepsilon(B_i)| \\
             B_i(x_i, \, s|\nu_\varepsilon(B_i)|) & \textrm{otherwise}.
            \end{cases}
 \end{equation}
 For $s$ small enough, the balls $B_i(s)$'s are pairwise disjoint. We also have $|\nu_\varepsilon(B_i)| = |\nu_\varepsilon(B_i(s))|$,
 because $(B_i(s)\setminus B_i)\cap\spt\mathscr{B}_\varepsilon=\emptyset$ and $\spt(\nu_\varepsilon)\subseteq\spt\mathscr{B}_\varepsilon$.
 If for some~$s^*$ there happens $B_i(s^*)\cap B_j(s^*)\neq\emptyset$ for~$i\neq j$, then we stop the expansion phase.
 We define $\mathscr{B}_\varepsilon(s^*)$ as the family of balls obtained from $\{B_i(s^*)\}_{i=1}^k$ 
 via merging. For $s>s^*$, we repeat an expansion phase according to the same rule as~\eqref{expansion}, until two or more balls touch and 
 we perform a merging phase again, and so on. 
 
 Arguing as in~\cite[Proposition~4.1]{Jerrard}, one can show that~$\mathscr{B}_\varepsilon(s)$ 
 satisfies~\ref{item:ball_monotone}, \ref{item:ball_energy} and~\ref{item:ball_lower_radius}. 
 (Actually, \ref{item:ball_energy} appears in a slightly different form, but the same argument applies.)
 The proof of~\ref{item:ball_energy} relies on Lemma~\ref{lemma:Lambda_annulus}; 
 in order to apply this lemma, we need to make sure that the radii of all the balls we consider are~$\leq R_*$.
 However, if we temporarily assume that~\ref{item:ball_radius} holds, then (using~\eqref{radius_B0} and~\eqref{s_eps} as well) we see that
 \begin{equation*} \label{bound_radius}
  \sum_{B\in\mathscr{B}_\varepsilon(s)}\rad(B) \leq C s \, \#(\mathscr{B}_\varepsilon).
 \end{equation*}
 Therefore, we have~$\rad(B)\leq R_*$ for any~$B\in\mathscr{B}_\varepsilon(s)$ and any~$s\leq C^{-1}R_* \#(\mathscr{B}_\varepsilon)$.
 
 To prove~\ref{item:ball_radius}, we note that the quantity $\sum_{B\in\mathscr{B}_\varepsilon(s)}\rad(B)$ is preserved 
 during each merging phase. Then, for a fixed~$s$, let~$s_1 < \ldots < s_k < s$ be the values of the parameter
 when merging occurred, and take~$B_i(s)\in\mathscr{B}_\varepsilon(s)$. From~\eqref{expansion}, we see that
 \[
  \rad(B_i(s)) = \min\left\{\rad(B_i(s_k)), \, s|\nu_\varepsilon(B_i(s))| \right\} \leq s|\nu_\varepsilon(B_i(s_k))|
  \stackrel{\ref{item:ball_lower_radius}}{\leq} \frac{s}{s_k} \rad(B_i(s_k))
 \]
 (we have used that~$s\mapsto|\nu_\varepsilon(B_i(s))|$ is constant during each expansion phase). Thus,
 \[
  \sum_{B\in\mathscr{B}_\varepsilon(s)}\rad(B)\leq \frac{s}{s_k}\sum_{B\in\mathscr{B}_\varepsilon(s_k)}\rad(B).
 \]
 Now, we complete the proof of~\ref{item:ball_radius} arguing by induction.
\end{proof}

As an immediate consequence of Proposition~\ref{prop:ball}, using~\eqref{radius_B0}, \eqref{s_eps}
the definition~\eqref{Lambda_eps} of~$\Lambda_\varepsilon$ and~\eqref{Lambda_bound} in Lemma~\ref{lemma:Lambda}, we obtain

\begin{corollary} \label{cor:ball}
 For any~$s\in [s_\varepsilon, \, CR_*\#(\mathscr{B}_\varepsilon)^{-1}]$, there exists a family of pairwise disjoint, 
 closed balls~$\mathscr{B}_\varepsilon(s)$ which satisfies the following properties:
 \begin{enumerate}[label=(\roman*)]
  \item \label{cor:ball_monotone} $\spt\mathscr{B}_\varepsilon\subseteq\spt\mathscr{B}_\varepsilon(s)\subseteq\spt\mathscr{B}_\varepsilon(t)$ 
  for any~$s_\varepsilon \leq s \leq t \leq  R_*\#(\mathscr{B}_\varepsilon)^{-1}$;
  \item \label{cor:ball_energy} for any~$B\in\mathscr{B}_\varepsilon(s)$, there holds
  \[
   \frac12 \abs{\w_\varepsilon}^2_{W^{1,2}_\varepsilon(B\setminus\spt\mathscr{B}_\varepsilon)} \geq |\nu_\varepsilon(B)|
   \left(\pi(1-C\varepsilon)\log\frac{s}{\varepsilon} - C\right) ;
  \]
  \item \label{cor:ball_radius} there holds $\sum_{B\in\mathscr{B}_\varepsilon(s)}\rad(B) \leq C s \, \#(\mathscr{B}_\varepsilon)$.
 \end{enumerate}
\end{corollary}

\subsection{Proof of the zero-order $\Gamma$-convergence}
\label{sect:0Gamma}

We state and prove a zero-order $\Gamma$-convergence result in terms of the measures~$\nu_\varepsilon(\v_\varepsilon)$.
Given a measure~$\mu\in X$ with $\mu = \sum_i d_i\delta_{x_i}$, we set
\[
 \sigma_0(\mu) := \frac12\min\left\{\min_{j\neq i}\dist(x_i, \, x_j), \, \textrm{injectivity radius of } M\right\} \! .
\]


\begin{prop} \label{prop:0Gamma-nu} 
 Suppose that the assumptions~\eqref{hp:quasiuniform}, \eqref{hp:weakly_acute} and~\eqref{hp:bilipschitz} are satisfied.
 Then, the following results hold.
 \begin{enumerate}[label = (\roman*), leftmargin=*]
  \item \emph{Compactness.} If~$(\v_\varepsilon)$ is a sequence in~$\T(\TT_\varepsilon; \, \S^2)$
  that satisfies the energy bound~\eqref{H} then, up to subsequences, 
  $\nu_\varepsilon(\v_\varepsilon) \xrightarrow{\fflat}\mu$ for some~$\mu\in X$.
  
  \item \emph{Localized $\Gamma$-liminf inequality.} Let~$(\v_\varepsilon)$ be a sequence in~$\T(\TT_\varepsilon; \, \S^2)$
  such that $\nu_\varepsilon(\v_\varepsilon) \xrightarrow{\fflat}\mu$ for some~$\mu\in X$, $\mu = \sum_{i=1}^K d_i\delta_{x_i}$.
  Then, there exists a constant~$C$ such that, for any~$i\in\{1, \, \ldots, \, K\}$
  and any~$0 < \sigma \leq \sigma_0(\mu)$, there holds
  \[
   \liminf_{\varepsilon\to 0} \left(\frac12 |\w_\varepsilon|^2_{W^{1,2}_\varepsilon(B_\sigma(x_i))}
   - \pi|d_i|\log\frac{\sigma}{\varepsilon}\right)\geq C.
  \]
  
  \item \emph{$\Gamma$-limsup inequality.} For any~$\mu\in X$ there exists a sequence~$(\v_\varepsilon)$ in~$\T(\TT_\varepsilon; \, \S^2)$
  such that $\nu_\varepsilon(\v_\varepsilon) \xrightarrow{\fflat}\mu$ and
  \[
   \limsup_{\varepsilon\to 0} \frac{|\w_\varepsilon|^2_{W^{1,2}_\varepsilon(M)}}{2|\log\varepsilon|} \leq \pi|\mu|(M).
  \]
 \end{enumerate}
\end{prop}

The proof of this Proposition is adapted from~\cite[Theorem~3.1]{ADGP}.
Throughout the proof, we write~$\nu_\varepsilon$ instead of~$\nu_\varepsilon(\v_\varepsilon)$ when no confusion is possible.

\begin{proof}[Proof of~(i) --- Compactness] 
 Let~$\mathscr{B}_\varepsilon^1 := \mathscr{B}_\varepsilon(s_\varepsilon^1)$ be the family of balls given by Corollary~\ref{cor:ball}
 for the choice of parameter~$s_\varepsilon^1 := \varepsilon^{1/2}$. If~$\varepsilon$ is small enough, we have
 \[
  s_\varepsilon \stackrel{\eqref{s_eps}}{\leq} \varepsilon^{1/2} \leq \frac{CR_*}{|\log\varepsilon|} 
  \stackrel{\eqref{cardinal_B0}}{\leq} \frac{CR_*}{\#(\mathscr{B}_\varepsilon)},
 \]
 so~$s_\varepsilon^1$ satisfies the assumptions of Corollary~\ref{cor:ball}.
 By~\ref{cor:ball_monotone} in Corollary~\ref{cor:ball}, we have that
 $\spt(\nu_\varepsilon)\subseteq\spt\mathscr{B}_\varepsilon\subseteq\spt\mathscr{B}_\varepsilon^1$,
 while (ii) implies
 \[
  \frac12 \abs{\w_\varepsilon}^2_{W^{1,2}_\varepsilon(B)} \geq |\nu_\varepsilon(B)|
   \left(\frac{\pi}{2}(1-C\varepsilon)|\log\varepsilon| - C\right)
 \]
 for any~$B\in\mathscr{B}_\varepsilon^1$. Summing up this inequality over all the~$B$'s,
 using Lemma~\ref{lemma:norm-interpolant} and the energy bound~\eqref{H}, one sees that
 \begin{equation*} 
  \sum_{B\in\mathscr{B}_\varepsilon^1} \abs{\nu_\varepsilon (B)} \leq C
 \end{equation*}
 for an~$\varepsilon$-independent constant~$C$. Therefore, the measures
 $\nu^1_\varepsilon := \sum_{B\in\mathscr{B}_\varepsilon^1} \nu_\varepsilon(B)\delta_{x(B)}\in X$,
 where~$x(B)$ denotes the centre of the ball~$B$, have uniformly bounded mass and flat-converge to an element of~$X$,
 up to extraction of a subsequence. On the other hand, \ref{cor:ball_radius} in Corollary~\ref{cor:ball} implies
 \[
  \sum_{B\in\mathscr{B}_\varepsilon^1} \rad(B) \leq C \varepsilon^{1/2} \#(\mathscr{B}_\varepsilon)
  \stackrel{\eqref{cardinal_B0}}{\leq} C \varepsilon^{1/2} |\log\varepsilon|.
 \]
 Then, arguing as in~\cite[Theorem~3.1.(i)]{ADGP}, one can show that~$\|\nu_\varepsilon - \nu_\varepsilon^1\|_{\fflat}\to 0$,
 which yields compactness for the sequence~$(\nu_\varepsilon)$.
 (see also~\cite[Theorem~3.3]{AlicandroPonsiglione} for more details). 
\end{proof}
\begin{proof}[Proof of~(ii) --- $\Gamma$-liminf] 
 Fix~$i\in\{1, \, \ldots, \, K\}$ and~$0 < \sigma\leq\sigma_0(\mu)$.
 By extraction of a non-relabelled subsequence, we can assume WLOG that
 \begin{equation} \label{0liminf-1}
  \lim_{\varepsilon\to 0} \left(\frac12 |\w_\varepsilon|^2_{W^{1,2}_\varepsilon(B_\sigma(x_i))} - \pi|d_i||\log\varepsilon|\right) < +\infty
 \end{equation}
 (and, in particular, the limit exists). 
 Arguing as in~\cite[Theorem~3.1.(ii)]{ADGP}, we see that 
 $\|\nu_\varepsilon(\v_\varepsilon) - \nu_\varepsilon(\bar{\v}_\varepsilon)\|_{\fflat}\to 0$, 
 where~$\bar{\v}_\varepsilon$ denotes the restriction of~$\v_\varepsilon$ to~$B_\sigma(x_i)$, 
 and that $\nu_\varepsilon(\bar{\v}_\varepsilon)$ flat-converges to~$d_i\delta_{x_i}$.
 Therefore, we can repeat the ball construction of Section~\ref{sect:ball} with~$M$ replaced by~$B_{\sigma}(x_i)$
 and~$\v_\varepsilon$ replaced by~$\bar{\v}_\varepsilon$. (This guarantees that no ball ``coming from outside'' enters~$B_\sigma(x_i)$.)
 We still write~$\nu_\varepsilon$ instead of~$\nu_\varepsilon(\bar{\v}_\varepsilon)$.
 
 For a fixed~$\gamma\in (0, \, 1)$, we apply Corollary~\ref{cor:ball} with~$s = s^2_\varepsilon := \varepsilon^\gamma$.
 (One can check, arguing as in the proof of~(i), that the assumptions of Corollary~\ref{cor:ball} are satisfied.)
 The collection of balls~$\mathscr{B}_\varepsilon(s^2_\varepsilon)$ satisfies $\spt(\nu_\varepsilon)\subseteq\spt\mathscr{B}_\varepsilon(s^2_\varepsilon)$,
 \begin{equation*} 
  \sum_{B\in\mathscr{B}_\varepsilon(s^2_\varepsilon)} \rad(B) \leq C \varepsilon^{\gamma} \#(\mathscr{B}_\varepsilon)
  \stackrel{\eqref{cardinal_B0}}{\leq} C \varepsilon^{\gamma} |\log\varepsilon|
 \end{equation*}
 and
 \begin{equation} \label{0liminf-3}
  \frac12 \abs{\w_\varepsilon}^2_{W^{1,2}_\varepsilon(B)} \geq |\nu_\varepsilon(B)|
   \left( \pi(1-\gamma)(1-C\varepsilon)|\log\varepsilon| - C\right)
 \end{equation}
 for any~$B\in\mathscr{B}_\varepsilon(s^2_\varepsilon)\setminus\mathscr{B}_\varepsilon$. 
 Let $\mathscr{B}^2_\varepsilon := \{B\in\mathscr{B}_\varepsilon(s^2_\varepsilon)\colon B\subseteq B_{\sigma}(x_i)\}$
 and~$\nu^2_\varepsilon := \sum_{B\in\mathscr{B}_\varepsilon^2} \nu_\varepsilon(B)\delta_{x(B)}$.
 Arguing again as in~\cite{ADGP}, we see that~$\|\nu^2_\varepsilon - \nu_\varepsilon\|_{\fflat}\to 0$, so~$\nu^2_\varepsilon$
 flat-converges to~$d_i \delta_{x_i}$ and, in particular, $\liminf_{\varepsilon\to 0}|\nu^2_\varepsilon|(B_\sigma(x_i)) \geq|d_i|$.
 Now, by summing up the inequality~\eqref{0liminf-3} with respect to~$B\in\mathscr{B}_\varepsilon^2$, we deduce
 \[
  \frac12 \abs{\w_\varepsilon}^2_{W^{1,2}_\varepsilon(B_\sigma(x_i))} \geq 
  \pi(1-\gamma)(1-C\varepsilon)\abs{\nu_\varepsilon^2}(B_\sigma(x_i))\abs{\log\varepsilon} - C.
 \]
 If~$\liminf_{\varepsilon\to 0} |\nu_\varepsilon^2|(B_\sigma(x_i)) > |d_i|$, then in fact
 $\liminf_{\varepsilon\to 0} |\nu_\varepsilon^2|(B_\sigma(x_i)) \geq |d_i| + 1$ (because $\nu^2_\varepsilon$ is integer-valued)
 and hence the $\Gamma$-liminf inequality~(ii) follows, provided that we choose $\gamma$ such that $(1- \gamma) (|d_i| + 1) > |d_i|$.
 Otherwise, we have that~$|\nu^2_\varepsilon|(B_\sigma(x_i)) = |d_i|$ for~$\varepsilon$ small enough.
 Then, we can write
 \[
  \nu^2_\varepsilon = \sum_{j=1}^k p^\varepsilon_j \, \delta_{y^\varepsilon_i},
 \]
 where the numbers~$p^\varepsilon_j\in\mathbb{Z}$ all have the same sign and satisfy $\sum_j p^\varepsilon_j = d_i$,
 and~$y^\varepsilon_j \to x_i$ as~$\varepsilon\to 0$. By taking~$\varepsilon$ small enough,
 we can assume that~$y_j^\varepsilon\in B_{\sigma/2}(x_i)$ for all~$j$.
 
 We fix a positive number~$\eta > 0$ apply Corollary~\ref{cor:ball} with~$s = s_\varepsilon^3 := \eta \, \#(\mathscr{B}_\varepsilon)^{-1}$.
 (We choose~$\eta$ small enough that~$s_\varepsilon^3 \leq CR_*\#(\mathscr{B}_\varepsilon)^{-1}$, so the assumptions 
 of Corollary~\eqref{cor:ball} are satisfied.) We find a collection of balls~$\mathscr{B}_\varepsilon^3 := \mathscr{B}_\varepsilon(s_\varepsilon^3)$
 that satisfies $\spt(\nu^2_\varepsilon)\subseteq\spt\mathscr{B}^3_\varepsilon$,
 \begin{gather}
  \sum_{B\in\mathscr{B}_\varepsilon^3} \rad(B) \leq C\eta \label{0liminf-4} \\
  \frac12 \abs{\w_\varepsilon}^2_{W^{1,2}_\varepsilon(B\setminus\spt\mathscr{B}_\varepsilon)} \geq |\nu_\varepsilon(B)|
   \left( \pi(1-C\varepsilon)\log\frac{\eta}{\varepsilon\#(\mathscr{B}_\varepsilon)} - C\right) \label{0liminf-5}
 \end{gather}
 Thanks to~\eqref{0liminf-4} and the fact that $\spt(\nu^2_\varepsilon)\subseteq\spt\mathscr{B}^3_\varepsilon$,
 $\dist(\spt\nu^2_\varepsilon, \, \partial B_\sigma(x_i))\geq \sigma/2$,
 we can choose~$\eta$ so small that~$B\subseteq B_\sigma(x_i)$ for any~$B\in\mathscr{B}^3_\varepsilon$.
 Then, using also the fact that all the~$p^\varepsilon_j$ have the same sign and sum up to~$d_i$,
 we see that~$\sum_{B\in\mathscr{B}^3_\varepsilon} |\nu^2_\varepsilon(B)| = |d_i|$ and hence, by~\eqref{0liminf-5},
 \[
  \frac12 \abs{\w_\varepsilon}^2_{W^{1,2}_\varepsilon(B_\sigma(x_i)\setminus\spt\mathscr{B}_\varepsilon)} \geq
   \pi|d_i|(1-C\varepsilon)\log\frac{\eta}{\varepsilon\#(\mathscr{B}_\varepsilon)} - C.
 \]
 On the other hand, Lemma~\ref{lemma:beta} implies that
 \[
  \frac12 \abs{\w_\varepsilon}^2_{W^{1,2}_\varepsilon(\spt\mathscr{B}_\varepsilon)} \geq
  \sum_{T\in\TT_\varepsilon\colon P(T)\cap S_\varepsilon\neq\emptyset} 
  \frac12 \abs{\w_\varepsilon}^2_{W^{1,2}_\varepsilon(P(T))} \geq \beta\#(\mathscr{B}_\varepsilon).
 \]
 Thus, we have
 \[
  \begin{split}
   \frac12 \abs{\w_\varepsilon}^2_{W^{1,2}_\varepsilon(B_\sigma(x_i))} &\geq
   \pi|d_i|(1-C\varepsilon)\abs{\log\varepsilon} - \pi|d_i|(1-C\varepsilon)\log\frac{\#(\mathscr{B}_\varepsilon)}{\eta} 
   + \beta\#(\mathscr{B}_\varepsilon) - C \\
   &\geq \pi|d_i|\abs{\log\varepsilon} - C. \qedhere
  \end{split}
 \]
\end{proof}
\begin{proof}[Proof of~(iii) --- $\Gamma$-limsup] 
 Fix~$\mu = \sum_{i=1}^K d_i\delta_{x_i}\in X$, and suppose that~$d_i\neq 0$ for any~$i$.
 By a diagonal argument, it suffices to show the following: for any~$\delta$ and any countable subsequence of~$\varepsilon\searrow 0$,
 there exists a (non-relabelled) subsequence such that
 \begin{equation} \label{0limsup-0}
  \limsup_{\varepsilon\to 0} \frac{\abs{\w_\varepsilon}^2_{W^{1,2}_\varepsilon(M)}}{2|\log\varepsilon|}
  \leq \pi |\mu|(M) + \delta.
 \end{equation}
 Let us fix~$\delta >0$ and the countable subsequence of~$\varepsilon$. We also fix a small parameter~$0 < \sigma \leq \sigma_0(\mu)$.
 For $i\in\{1, \, \ldots, \, K\}$ and~$j\in \{1, \, \ldots, |d_i|\}$,
 we let~$z_{i,j}^\varepsilon := \varepsilon^\sigma\exp(2\pi\iota j|d_i|^{-1})\in\C$ (where~$\iota$ is the imaginary unit).
 By taking~$\varepsilon$ small enough, we can assume that $|z^\varepsilon_{i,j}| < \sigma$.
 We define a map~$\tilde{\u}^\varepsilon_i\colon B_\sigma\subseteq\C\to\C$ by
 \begin{equation} \label{0limsup-recovery}
  \tilde{\u}^\varepsilon_i(z) := \prod_{j = 1}^{d_i} \frac{z - z^\varepsilon_{i,j}}{|z - z^\varepsilon_{i,j}|} \quad \textrm{if } d_i > 0, \qquad
  \tilde{\u}^\varepsilon_i(z) := \prod_{j = 1}^{-d_i} \overline{\frac{z - z^\varepsilon_{i,j}}{|z - z^\varepsilon_{i,j}|}} \quad \textrm{otherwise.}
 \end{equation}
 Using normal coordinates~$\varphi_i\colon B_\sigma\subseteq\C\to M$ such that~$\varphi_i(0) = x_i$,
 we can transport~$\tilde{\u}^\varepsilon_i$ to a vector field on~$B_\sigma(x_i)\subseteq M$, 
 i.e. we define $\u^\varepsilon_i(\varphi_i(z)) := \langle \d\varphi_i(z), \, \tilde{\u}^\varepsilon_i(z)\rangle$ for~$z\in B_\sigma\subseteq\C$.
 Since~$\sum_i \ind(\u^\varepsilon_i, \, B_\sigma(x_i)) = \sum_i d_i = \chi(N)$, 
 we find a smooth vector field~$\u$ on~$M_\sigma := M\setminus \cup_i B_\sigma(x_i)$
 that satisfies~$\u = \u^\varepsilon_i$ on~$\partial B_\sigma(x_i)$ for each~$i$.
 We define~$\u^\varepsilon := \u^\varepsilon_i$ on~$B_\sigma(x_i)$ and~$\u^\varepsilon := \u$ on~$M_\sigma$.
 The tangent field~$\u^\varepsilon\in W^{1, 1}(M, \, \R^3)$ is smooth except at the
 points~$x^\varepsilon_{i,j}:= \varphi_i(z^\varepsilon_{i,j})\to x_i$ and hence, by Lemma~\ref{lemma:dj}, 
 \begin{equation} \label{0limsup-1}
  \frac{1}{2\pi} \left(\star\d\j(\u^\varepsilon) - G\right) = \sum_{i, \, j} \sign(d_i)\,\delta_{x_{i,j}^\varepsilon}
  \xrightarrow{\fflat} \sum_i d_i\delta_{x_i} = \mu \qquad \textrm{as } \varepsilon\to 0.
 \end{equation}
 
 Let~$\v_\varepsilon$ be the discrete field defined by $\v_\varepsilon(i) := \u^\varepsilon(i)$ for~$i\in\TT^0_\varepsilon$,
 and let~$\w_\varepsilon := \widehat{P}_\varepsilon^{-1}\circ\widehat{\v}_\varepsilon$. We have
 \[
  \abs{\nablas\w_\varepsilon} \stackrel{\eqref{hp:bilipschitz}}{\leq} C\abs{\nablae\widehat{\v}_\varepsilon} \leq C\abs{\nablas\u^\varepsilon},
 \]
 where the last inequality follows by basic properties of the affine interpolant. 
 Setting $D_i^\varepsilon := B_\sigma(x_i)\setminus \cup_j B_\varepsilon(x^\varepsilon_{i,j})$, 
 and using~\eqref{Sobolev_eps} and Lemma~\ref{lemma:metric_distorsion}, we obtain
 \[
  \begin{split}
   \frac12 \abs{\w_\varepsilon}&^2_{W^{1,2}_\varepsilon(M)} = \frac12 \sum_{i=1}^K \abs{\w_\varepsilon}^2_{W^{1,2}_\varepsilon(D^\varepsilon_i)} 
   + \frac12 \sum_{i, \, j} \abs{\w_\varepsilon}^2_{W^{1,2}_\varepsilon(B_\varepsilon(x^\varepsilon_{i,j}))}
   + \frac12 \abs{\w_\varepsilon}^2_{W^{1,2}_\varepsilon(M_\sigma)} \\
   &\leq \frac{1 + C\varepsilon}{2} \sum_{i=1}^K \int_{D^\varepsilon_i} \abs{\nablas\u^\varepsilon}^2 \, \d S 
   + \sum_{i=1}^K \Lip(\w_\varepsilon)^2 \H^2(B_\varepsilon(x_i)) + \frac{1 + C\varepsilon}{2} \int_{M_\sigma} \abs{\nablas\u}^2 \, \d S \\
   &\stackrel{\eqref{eq:lip_we}}{\leq} \frac{1 + C\varepsilon}{2} \sum_{i=1}^K \int_{D^\varepsilon_i} \abs{\nablas\u^\varepsilon}^2 \, \d S + C_\sigma,
  \end{split}
 \]
 where~$C_\sigma$ is a positive constant, depending on~$\sigma$.
 The integral of~$|\nablas\v|^2$ on each~$D^\varepsilon_i$ can be evaluated using~\eqref{0limsup-recovery}
 and the fact that $\Lip({\varphi_i}_{|B_\sigma(x_i)}) \leq 1 + C\sigma$:
 \[
  \begin{split}
  \frac{1 + C\varepsilon}{2} \int_{D^\varepsilon_i} \abs{\nablas\u^\varepsilon}^2 \, \d S 
  &\leq  \left(\pi (1 + C\sigma)|d_i| + C\sigma|d_i|^2\right) |\log\varepsilon| + C_\sigma,
  \end{split}
 \]
 whence
 \[
  \limsup_{\varepsilon\to 0} \frac{\abs{\w_\varepsilon}^2_{W^{1,2}_\varepsilon(M)}}{2|\log\varepsilon|}
  \leq  \pi(1 + C\sigma) |\mu|(M) + C\sigma\left(|\mu|(M)\right)^2
 \]
 and, choosing~$\sigma$ so small that $C\sigma|\mu|(M) + C|\mu|(M)^2\leq\delta$, \eqref{0limsup-0} follows.
 
 To conclude the proof, we only need to show that~$\nu_\varepsilon(\v_\varepsilon)$ flat-converges to~$\mu$.
 Using~\eqref{hp:quasiuniform}, the definition of affine interpolant, and the fact that
 \[
  \abs{\nablas\u^\varepsilon(x)}\leq \frac{C_\sigma}{\dist(x, \, \{x^\varepsilon_{i,j}\})} \qquad \textrm{for } x\in B_\sigma(x_i)
 \]
 (as a consequence of~\eqref{0limsup-recovery}), one finds positive numbers $\lambda$, $\varepsilon_1$ 
 such that, for any~$0 < \varepsilon\leq\varepsilon_1$, there holds $|\w_\varepsilon|\geq 1/2$ on 
 $A_\varepsilon := M \setminus \cup_{i,j} B_{\lambda\varepsilon}(x^\varepsilon_{i,j})$. 
 Thanks to Lemma~\ref{lemma:proiection-pointwise}, this implies $|\u_\varepsilon|\geq 1/4$ if~$\varepsilon$ is small enough,
 where~$\u_\varepsilon$ is the field defined by~\eqref{u_eps}. Then, using~\eqref{mu2} and~\eqref{nu_eps},
 we obtain that~$\nu_\varepsilon(\v_\varepsilon)[B] = 0$ if~$B\subseteq A_\varepsilon$. 
 On the other hand, we also have~$\mu[B] = 0$ if~$B\subseteq A_\varepsilon$, due to~\eqref{0limsup-1}.
 Thus, for any Lipschitz function~$\varphi$ on~$M$ such that~$\sup|\varphi| + \Lip(\varphi)\leq 1$, there holds
 \begin{equation} \label{0limsup-2}
  \begin{split}
   \langle \nu_\varepsilon(\v_\varepsilon) - \mu, \, \varphi\rangle 
   &= \sum_{i, \, j}\int_{B_{\lambda\varepsilon}(x^\varepsilon_{i,j})} \varphi \, \d (\nu_\varepsilon(\v_\varepsilon) - \mu) \\
   &= \sum_{i, \, j} \int_{B_{\lambda\varepsilon}(x^\varepsilon_{i,j})} \left(\varphi - \varphi(x_i)\right) 
   \, \d (\nu_\varepsilon(\v_\varepsilon) - \mu)
   \leq C\lambda\varepsilon \left(|\nu_\varepsilon(\v_\varepsilon)| + |\mu|\right)(M),
  \end{split}
 \end{equation}
 and Lemma~\ref{lemma:ni_triangle} implies
 \begin{equation} \label{0limsup-3}
  |\nu_\varepsilon(\v_\varepsilon)|(M) \leq C \#\left\{T\in\TT_\varepsilon\colon P(T)\setminus A_\varepsilon\neq\emptyset\right\}
  \stackrel{\eqref{hp:quasiuniform}}{\leq} C.
 \end{equation}
 Combining~\eqref{0limsup-2} and~\eqref{0limsup-3}, we conclude that $\|\nu(\v_\varepsilon) - \mu\|_{\fflat}\leq C\lambda\varepsilon$.
\end{proof}
\section{The first-order $\Gamma$-convergence: location of defects and their energetics}
\label{sec:Gamma1}

\subsection{The renormalized energy}
\label{sect:renormalized}
In this Subsection we resume the concept of Renormalized energy that we 
have introduced in \eqref{eq:renormalized_ene_def} and we state its main properties. We recall that
we have set
\begin{equation}
\label{eq:renormalized_ene_def_bis}
\W_{\mathrm{intr}}(\v) := 
\begin{cases}
\lim_{\delta \to 0} \left(\dfrac{1}{2}\displaystyle\int_{M_\delta}\vert \D\v\vert^2\d S - \KK\pi\vert \log\delta\vert\right) &\hbox{for }\v\in \mathcal{V}_{\KK}\\
-\infty &\hbox{for }\v\in \mathcal{V}_{k},\,\,k <\KK,\\
+\infty &\hbox{otherwise in } L^2(M; \, \mathbb{R}^3),
\end{cases}
\end{equation}
where for $k\in \mathbb{N}$ we have set (see \eqref{eq:classV})
\[
 \begin{split}
  \mathcal{V}_k: = \left\{\v\in L^2(M; \, \S^2)\colon \textrm{there exist } (x_i)_{i=1}^k\in M^k,
  \ (d_i)_{i=1}^k \in \{-1, \, 1\} \textrm{ such that } \right. \\
  \left. \v\in W^{1,2}_{\tan, \mathrm{loc}}\left(M\setminus \bigcup_{i=1}^{k}x_i; \, \mathbb{S}^2\right) \textrm{ and }
  \star\d\j(\v) = 2\pi\sum_{i = 1}^k d_i \delta_{x_i} - G \right\} \! ,
 \end{split}
\]
for any~$k\in\N$. 
The object in \eqref{eq:renormalized_ene_def_bis} is well defined for the following reasons.

First of all, a standard construction based on the Poincar\'e-Hopf theorem (see, e.g., \cite[Proposition~1.5]{AGM-VMO})
shows that the set $\mathcal{V}_{\KK}$ is non empty if and only if~$\KK\geq |\chi(M)|$
and~$\KK\equiv\chi(M) \mod 2$, that is, $\KK$ is even. Second, for any
$\v \in \mathcal{V}_{\KK}$ and any~$\delta>0$ there holds that 
\[
 \dfrac{1}{2}\int_{M_\delta}\vert \D\v\vert^2\d S - \KK\pi\vert \log\delta\vert <+\infty.
\]
Third, for any $\v\in \mathcal{V}_{\KK}$, there holds
\[
 \begin{split}
  \frac{\d}{\d\delta}\left(\dfrac{1}{2}\int_{M_\delta}\vert\D\v\vert^2\d S - \KK\pi\vert \log\delta\vert\right)
  &= \sum_{i=1}^{\KK} \left(-\dfrac{1}{2}\int_{\partial B_\delta(x_i)} \abs{\D\v}^2 \, \d S + \frac{\pi}{\delta} \right) \\
  &\leq \sum_{i=1}^{\KK} \left(-\frac{(2\pi + C\delta^2)^2}{4\pi\delta + C\delta^2} + \frac{\pi}{\delta} \right) \! ,
 \end{split}
\]
where the last inequality follows by Lemma~\ref{lemma:lower_bound_buccia} below.
Since the righ-hand side is bounded from above as~$\delta\searrow 0$, we deduce that the limit 
in~\eqref{eq:renormalized_ene_def_bis} exists and belongs to~$(-\infty, \, +\infty]$.


Similar to the euclidean flat case (see \cite{ADGP}) we have the following dyadic decomposition
of the intrinsic Renormalized Energy.
\begin{lemma}
\label{lem:dyadic_ren}
Fix $\rho>$ small enough in such a way that $B_\rho(x_i)$ are pairwise disjoint for $i=1,\ldots,\KK$. 
Then, for any $\v\in \mathcal{V}_{\KK}$ there holds
\begin{equation*}
 \begin{split}
  \W_{\mathrm{intr}}(\v) &=  \dfrac{1}{2}\int_{M\setminus \bigcup_{i=1}^{\KK}B_{\rho}(x_i)} \vert \D\v\vert^2 \d S -\pi\KK\vert \log\rho\vert \\
  &\quad\quad + \sum_{i=1}^{\KK}\sum_{j=0}^{+\infty}\left(\dfrac{1}{2}\int_{B_{2^{-j}\rho}(x_i)\setminus B_{2^{-(j+1)}\rho}(x_i)}\vert \D\v\vert^2 \d S -\pi \log 2\right) \! .
 \end{split}
\end{equation*}
\end{lemma}
\begin{proof}
For $l\in \mathbb{N}$ and some fixed $\rho>0$ as in the statement,  
we consider $\delta = \rho 2^{-(l+1)}$. Note that the limit $\delta\to 0$ 
corresponds to the limit $l\to +\infty$. 
Moreover, we dyadically decompose $B_{\rho}(x_i)\setminus B_{\rho 2^{-(l+1)}}$ 
for any $i=1, \ldots, \KK$. We thus obtain  
\[
 M\setminus \bigcup_{i=1}^{\KK} B_{\delta}(x_i) =
 \left(M\setminus \bigcup_{i=1}^{\KK}B_{\rho}(x_i)\right) \cup
 \bigcup_{i=1}^{\KK} \bigcup_{j=0}^l 
 \left(B_{2^{-j}\rho}(x_i)\setminus B_{2^{-(j+1)}\rho}(x_i)\right) \! .
\]
Thus, we have 
\begin{align*}
  \int_{M\setminus \bigcup_{i=1}^{\KK} B_{\delta}(x_i)} \vert \D\v\vert^2 \d S -\pi\KK\vert \log\delta\vert = 
  \int_{M\setminus \bigcup_{i=1}^{\KK}B_{\rho}(x_i)} \vert \D\v\vert^2 \d S -\pi\KK\vert \log\rho\vert \\
  \qquad \qquad + \sum_{i=1}^{\KK}\sum_{j=0}^l \int_{B_{2^{-j}\rho}(x_i)\setminus B_{2^{-(j+1)}\rho}(x_i)}\vert \D\v\vert^2 \d S 
  -\pi\KK(l+1) \log2 \\
  = \int_{M\setminus \bigcup_{i=1}^{\KK}B_{\rho}(x_i)} \vert \D\v\vert^2 \d S -\pi\KK\vert \log\rho\vert
  +\sum_{i=1}^{\KK}\sum_{j=0}^l\left( \int_{B_{2^{-j}\rho}(x_i)\setminus B_{2^{-(j+1)}\rho}(x_i)}\vert \D\v\vert^2 \d S -\pi \log 2\right) \! . 
\end{align*}
Therefore, by sending $l\to +\infty$ (i.e. $\delta\to 0$), the lemma follows.
\end{proof}
Thanks to Lemma~\ref{lem:dyadic_ren}, the extrinsinc Renormalized Energy defined by~\eqref{eq:reno_intr/ex} satisfies
\begin{equation}
\label{eq:ren_energybis2_intr/ex}
 \begin{split}
  \W(\v) &= \frac 12\int_{M\setminus \bigcup_{i=1}^{\KK}B_{\rho}(x_i)} \vert \D\v\vert^2 \d S -\pi\KK\vert \log\rho\vert
  + \frac 12\int_{M} \vert \d\ggamma[\v]\vert^2\d S \\
  &\qquad\qquad + \sum_{i=1}^{\KK}\sum_{j=0}^{+\infty}\left(\dfrac{1}{2}\int_{B_{2^{-j}\rho}(x_i)\setminus B_{2^{-(j+1)}\rho}(x_i)}\vert \D\v\vert^2 \d S
  -\pi \log 2\right) \! .
 \end{split}
\end{equation}
An interesting consequence of the above representation is that, for $\v\in \mathcal{V}_{\KK}$ such that~$\W(\v) <+\infty$
there holds that, analogously to the euclidean case (see \cite[Remark~4.4]{ADGP}),
\begin{equation}
\label{eq:conc_energy_intr}
\lim_{j\to \infty}\frac 12\int_{B_{2^{-j}\rho}(x_i)\setminus B_{2^{-(j+1)}\rho}(x_i)}
\vert \D\v\vert^2\d S = \lim_{j\to \infty}\frac 12\int_{B_{2^{-j}\rho}(x_i)\setminus B_{2^{-(j+1)}\rho}(x_i)}e(\v) \d S =\pi\log 2.
\end{equation}
Consequently, we have 
\begin{lemma}
\label{lem:W11}
The effective domain of $\W$ in $\mathcal{V}_{\KK}$ is included in~$W^{1,1}_{\tan}(M; \ \mathbb{S}^2)$,\
namely
\[
\{\v\in\mathcal{V}_{\KK}\colon \W(\v) < +\infty\}\subseteq W^{1,1}_{\tan}(M; \, \mathbb{S}^2).
\]
\end{lemma}
\begin{proof}
It is clearly sufficient to show that any $\v\in\mathcal{V}_{\KK}$ with~$\W(\v) <+\infty$
is in $W^{1,1}(B_\rho(x_i))$ ($\rho>0$ smaller than the injectivity radius) for any $i=1,\ldots,\KK$.
We set $A_{2^{-(j+1)}\rho,2^{-j}\rho}^{i}:=B_{2^{-j}\rho}(x_i)\setminus B_{2^{-(j+1)}\rho}(x_i)$. There holds
\[
 \int_{A_{2^{-(j+1)}\rho,2^{-j}\rho}^{i}}\vert\D\v\vert \, \d S \le 
 \left(\H^2(A_{2^{-(j+1)}\rho,2^{-j}\rho}^{i})\right)^{1/2}
 \left(\int_{A_{2^{-(j+1)}\rho,2^{-j}\rho}^{i}}\vert\D\v\vert^2 \, \d S\right)^{1/2}
\]
and thus, thanks to \eqref{eq:conc_energy_intr}
\[
 \int_{A_{2^{-(j+1)}\rho,2^{-j}\rho}^{i}}\vert\D\v\vert \, \d S \le C \rho 2^{-j}, \qquad \textrm{for all } i=1,\ldots, \KK.
\]
Consequently, if we dyadically decompose $B_{\rho}(x_i)$ for any $i$ we get
\[
 \int_{B_{\rho}(x_i)}\vert\D\v\vert \,  \d S\le C\sum_{j=1}^{+\infty}2^{-j} <+\infty, 
\]
which clearly gives the result.
\end{proof}

\subsection{The core energy}
\label{sect:core}

In this subsection we rigorously define the concept of core energy
and discuss some of its properties. 
This object was introduced by Bethuel, Brexis and H\'elein~\cite{BBH} and 
later extended to the discrete setting by Alicandro et al.~\cite{ADGP}.

Given a point~$\bar x\in M$ and radii~$\delta_1$, $\delta_2$ such that $\delta_1<\delta_2$,
we denote by~$A_{\delta_1,\delta_2} (\bar x)$ the geodesic annulus
\[
 A_{\delta_1,\delta_2} (\bar x) := B_{\delta_2}(\bar x)\setminus B_{\delta_2}(\bar x) \subseteq M.
\]
Let us fix a positive number~$\delta$, smaller than the injectivity radius of~$M$.
We consider the minimization problem 
\begin{equation}\label{def:min_prob}
 \eta(\delta, \, \bar x) := \min_{\w \in W^{1,2}_{\tan}(A_{\delta/2,\delta}(\bar x); \, \mathbb{S}^2)}
 \left\{\frac{1}{2}\int_{A_{\delta/2,\delta}(\bar x)}\vert\D\w\vert^2 + \vert \d \ggamma [\w]\vert^2\d S,
 \quad\hbox{ind}(\w,\bar x) = 1\right\}.
\end{equation}
We denote with $\mathcal{H}(\delta, \, \bar x)$ the set of its minimizers.
$\mathcal{H}(\delta, \, \bar x)$ is non-empty, as follows by standard arguments in the Calculus of Variations.
We fix a minimizer~$\g\in\mathcal{H}(\delta, \, \bar x)$ for Problem~\eqref{def:min_prob} and,
for $\eps>0$, $\delta>0$, we set 
\begin{equation} \label{eq:core}
 \gamma_{\bar x}(\eps, \delta):= \min_{\v \in \T(\TT_\eps, S^2)}
 \left\{\frac{1}{2}\int_{\widehat{B_\delta(\bar x)}_\eps} \abs{\nablae\widehat{\v}}^2 \, \d S \colon
 \quad \v =\g  \,\hbox{ on } \partial_\eps B_\delta(\bar x) \right\}.
\end{equation}
Since~$\bar x$ is fixed throughout this section, we will write~$\gamma(\eps, \, \delta)$
instead of~$\gamma_{\bar x}(\eps, \, \delta)$.
We recall that $\widehat{B_\delta(\bar x)}_\eps$ is the union of the triangles~$T\in\TT_\eps$ such that
$P(T)\subseteq B_\delta(\bar x)$, 
$\partial_\varepsilon B_\delta(\bar x) := \partial (\widehat{B_\delta(\bar x)}_\eps)\cap \TT_\eps^0$,
and~$\widehat{\v}$ is the affine interpolant of the discrete field~$\v$.
$\gamma(\eps, \delta)$ depends on the choice of the point~$\bar x$, of~$\g\in\mathcal{H}(\delta, \, \bar x)$ and of the
triangulation~$\TT_\eps$, even though we have dropped this dependence in the notation.
We are interested in the asymptotic behaviour of~$\gamma(\eps, \delta)$ as~$\eps\searrow 0$, $\delta\searrow 0$.
\begin{prop} 
 \label{prop:core_energy}
 Suppose that the sequence~$(\TT_\eps)$ satisfies \eqref{hp:quasiuniform}, \eqref{hp:weakly_acute}.
 \eqref{hp:bilipschitz} and~\eqref{hp:convergence}. Then, for any~$\bar x\in M$ the following limits,
 are finite and coincide:
 \begin{equation*} 
  \gamma(\bar x) := \lim_{\delta\to 0} \, \liminf_{\eps\to 0} \, 
  \left(\gamma(\eps, \delta) - \pi\log\frac{\delta}{\eps}\right) \! 
  = \lim_{\delta\to 0} \, \limsup_{\eps\to 0} \, 
  \left(\gamma(\eps, \delta) - \pi\log\frac{\delta}{\eps}\right) \! .
 \end{equation*}
\end{prop}

It will be clear from the proof that the number~$\gamma(\bar{x})$ depends 
on~$\bar{x}$ and on the sequence~$(\TT_\eps)$, but \emph{not} on the choice
of~$\g\in\mathcal{H}(\delta, \, \bar x)$.
To prove Proposition~\ref{prop:core_energy}, we compare~$\gamma(\eps, \delta)$ with the solution of an auxiliary problem, 
defined as the ``flat'' counterpart of Problem~\eqref{eq:core}. We know that the solution of the latter
converges to a finite limit as~$\varepsilon\searrow 0$, thanks to the analysis in \cite{BBH} and~\cite{ADGP},
and hence we will be able to prove convergence for~$\gamma(\eps, \delta)$.

Before moving to the proof of Proposition~\ref{prop:core_energy}, let us fix some notation.
Let~$\delta > 0$ be smaller than the injectivity radius of~$M$.
Let~$\bar{A}_{\delta/2,\delta}\subseteq\R^2$ be the Euclidean annulus, 
centred at the origin, with radii $\delta/2$ and~$\delta$.
The geodesic coordinates $\varphi\colon B_{\delta}\subseteq\R^2\to M$ 
induce a bijection~$W^{1, 2}(\bar{A}_{\delta/2, \delta}; \, \S^1)\to
W^{1, 2}_{\tan}(A_{\delta/2, \delta}(\bar x); \, \S^2)$: 
the push-forward $\varphi_*\overline{\w}$ of a 
field~$\overline{\w}\in W^{1, 2}(\bar{A}_{\delta/2, \delta}; \, \S^1)$ is defined by 
\begin{equation} \label{eq:tildeh}
 (\varphi_*\overline\w)(\varphi(x)) := \frac{\langle \d\varphi(x), \, \overline{\w}(x)\rangle}
 {\abs{\langle \d\varphi(x), \, \overline{\w}(x)\rangle}},
 \qquad \hbox{for } x \in \bar{A}_{\delta/2, \delta}. 
\end{equation}
The pull-back of a field~$\w\in W^{1, 2}_{\tan}(A_{\delta/2, \delta}(\bar x); \, \S^2)$
is defined by $\varphi^*\w := (\varphi^{-1})_*\w$. A straightforward computation,
based on the fact that~$\d\varphi(0) = \Id_{\T_{\bar x} M}$, shows that
\begin{equation} \label{energy_w}
 E_{\mathrm{extr}}(\varphi_*\overline{\w}; \, A_{\delta/2,\delta}(\bar x)) 
 \leq \left(\frac12 + \mathrm{O}(\delta)\right)
 \int_{\bar A_{\delta/2,\delta}} \abs{\nabla\overline{\w}}^2 \, \d S
\end{equation}
as~$\delta\to 0$; in a similar way, there holds
\begin{equation} \label{energy_bar_w}
 \frac12 \int_{\bar A_{\delta/2,\delta}} \abs{\nabla(\varphi^*\w)}^2 \, \d S
 \leq \left(1 + \mathrm{O}(\delta)\right) E_{\mathrm{extr}}(\w; \, A_{\delta/2,\delta}(\bar x)).
\end{equation}
Moreover, we have $\ind(\w, \, \bar x) = \ind(\varphi^*\w, 0)$.
The push-forward of discrete fields $\varphi_*\colon\T(\overline{\TT}_\eps, \, \S^2)\to\T(\TT_\eps, \, \S^1)$,
along with its inverse~$\varphi^*$, is defined in a similar way.

Using~$\varphi$, we can compare minimizers~$\g\in\mathcal{H}(\delta, \, \bar x)$ of Problem~\eqref{def:min_prob}
with minimizers of the corresponding Euclidean problem, namely
\begin{equation} \label{def:min_prob_flat}
 \min_{\overline{\w}\in W^{1,2}(\bar A_{\delta/2,\delta}; \, \mathbb{S}^1)}
 \left\{\frac{1}{2}\int_{\bar A_{\delta/2,\delta}}\abs{\nabla\overline{\w}}^2\d S,
 \quad\hbox{ind}(\overline\w, \, 0) = 1\right\}.
\end{equation}
The minimizers of~\eqref{def:min_prob_flat} are exactly the fields of the 
form~$\overline{\h}_R(x) := Rx/|x|$, where~$R\in\mathrm{SO}(2)$ is a constant rotation matrix. 

\begin{lemma} \label{lem:hedgehogs}
 For any~$\delta >0$ (smaller than the injectivity radius of~$M$) and any~$\g\in\mathcal{H}(\delta, \, \bar x)$,
 there exist~$R = R(\delta, \, \g)\in\mathrm{SO}(2)$ such that
 \[
  \lim_{\delta\searrow 0} \, 
  \lVert\varphi^*\g - \overline{\h}_R\rVert_{W^{1,2}(A_{\delta/2, \delta}(\bar x))} = 0.
 \]
\end{lemma}
\begin{proof}
 Set~$\overline{\g} :=\varphi^*\g$ and~$\h_R := \varphi_*\overline{\h}_R$.
 By minimality of~$\g$ and~\eqref{energy_w}, we have
 \begin{equation*} 
  \begin{split}
   E_{\mathrm{extr}}(\g, \, A_{\delta/2, \delta}(\bar x)) 
  \leq E_{\mathrm{extr}}(\h_R, \, A_{\delta/2, \delta}(\bar x))
  \stackrel{\eqref{energy_w}}{\leq} \left(\frac12 + \mathrm{O}(\delta)\right) 
  \int_{\bar{A}_{\delta/2, \delta}} \abs{\nabla\overline{\h}_R}^2 \d S = 
  \pi\log 2 + \mathrm{O}(\delta).
  \end{split}
 \end{equation*}
 Then, using~\eqref{energy_bar_w}, we obtain
 \[
  \frac 12 \int_{\bar{A}_{\delta/2, \delta}} \abs{\nabla\overline{\g}}^2\d S \leq 
  \left(1 + \mathrm{O}(\delta)\right) E_{\mathrm{extr}}(\g, \, A_{\delta/2, \delta}(\bar x)) 
  \leq \pi\log 2 + \mathrm{O}(\delta).
 \]
 Let~$\overline{\g}_\delta\colon A_{1/2, 1}\to\S^1$ be defined 
 by~$\overline{\g}_\delta(x) := \overline{\g}(\delta x)$.
 From the previous inequality, it follows that
 \begin{equation} \label{min_seq}
  \lim_{\delta\searrow 0} \frac 12 \int_{\bar{A}_{1/2, 1}} \abs{\nabla\overline{\g}_\delta}^2\d S 
  = \lim_{\delta\searrow 0} \frac 12 \int_{\bar{A}_{\delta/2, \delta}} \abs{\nabla\overline{\g}}^2\d S 
  = \pi\log 2,
 \end{equation}
 where the right-hand side is exactly the minimum value for Problem~\eqref{def:min_prob_flat}.
 Thus, $\overline{\g}_\delta$ is a minimizing sequence for Problem~\eqref{def:min_prob_flat}
 and, by standard arguments in the Calculus of Variations, we find a subsequence that converges 
 strongly in~$W^{1,2}(\bar{A}_{1/2, 1})$ to a minimizer of~\eqref{def:min_prob_flat}. 
 Now, arguing by contradiction, we deduce that
 \begin{equation*}
  \lim_{\delta\searrow 0} \, \inf_{R\in \mathrm{SO}(2)} \, 
  \lVert\overline{\g}_\delta - \overline{\h}_R\rVert_{W^{1,2}(A_{1/2, 1})} = 0,
 \end{equation*}
 whence the lemma follows.
\end{proof}

We point out a couple of immediate, but useful, consequences of Lemma~\ref{lem:hedgehogs}.
\begin{lemma} \label{lem:hedgehogs_circle}
 For any~$\delta >0$ (smaller than the injectivity radius of~$M$) and any~$\g\in\mathcal{H}(\delta, \, \bar x)$,
 there exist~$R = R(\delta, \, \g)\in\mathrm{SO}(2)$ such that
 \[
  \lim_{\delta\searrow 0} \, \sup_{t\in [3\delta/4, \, \delta]}
  \lVert\varphi^*\g - \overline{\h}_R\rVert_{W^{1/2, 2}(\partial B_t)} = 0.
 \]
\end{lemma}
\begin{proof}
 By a scaling argument, we find a constant~$C$ such that the norm of the trace operator
 $T_{\delta, t}\colon W^{1,2}(\bar{A}_{\delta/2, \delta})\to W^{1/2, 2}(\partial B_t)$ is bounded by~$C$,
 for any~$t$ and~$\delta$ satisfying~$3/4 \leq t/\delta \leq 1$. Then, the lemma immediately follows from
 Lemma~\ref{lem:hedgehogs} and the continuity of~$T_{t, \delta}$.
\end{proof}

\begin{lemma} \label{lem:energia_hedgehogs}
 We have that $\eta(\delta, \, \bar x)\to \pi\log 2$ as~$\delta\searrow 0$, uniformly in~$\bar x\in M$.
\end{lemma}
\begin{proof}
 This follows from the arguments in the proof of Lemma~\ref{lem:hedgehogs}.
 Note that, since~$M$ is compact and smooth, the quantities~$\mathrm{O}(\delta)$ that appear 
 in~\eqref{energy_w}--\eqref{energy_bar_w} are bounded uniformly with respect to~$\bar x$.
\end{proof}

Lemma~\ref{lem:hedgehogs} above remains valid with a similar proof also for vector fields 
$\v\in \mathcal{V}_{\KK}$ with $\KK \equiv \chi(M) \mod 2$.
The elements
in $\mathcal{V}_{\KK}$ are not necessarily minimizers
of \eqref{def:min_prob} but they satisfy, thanks to the dyadic decomposition of the renormalized energy 
(see \eqref{eq:ren_energybis2_intr/ex} and \eqref{eq:conc_energy_intr}),  
\[
 \lim_{\delta\to 0}\frac 12\int_{A_{\delta,\delta/2}(x_i)}
 \vert \D\v\vert^2\d S = \lim_{\delta\to 0}\frac 12\int_{A_{\delta,\delta/2}(x_i)}
 \vert \nablas\v\vert^2\d S=\pi\log 2\ \quad \textrm{for any } i =1, \ldots, \KK.
\]
The above convergence replaces \eqref{min_seq} and thus we have 
\begin{lemma}
\label{lem:hedgehogsV}
 Let $\KK \equiv \chi(M) \mod 2$ and consider $\v\in\mathcal{V}_{\KK}$. Then, 
 for any~$\delta >0$ (smaller than the injectivity radius of~$M$) and any~$i=1, \ldots, \KK$ 
 there exist~$R = R(\delta, \, i, \, \g)\in\mathrm{SO}(2)$ such that
 \[
  \lim_{\delta\searrow 0} \, \lVert \nablas\v - \nablas \h_R\rVert_{L^2(A_{\delta/2, \delta}(x_i))} = 0.
 \]
\end{lemma}

The following lemma will also be useful in the proof of Proposition~\ref{prop:core_energy}.

\begin{lemma} \label{lemma:affine_pullback}
 Let~$T\subseteq\R^3$ be a triangle of vertices~$i_0$, $i_1$, $i_2$, and let $\w\colon T\to\R^3$ be an affine map
 such that~$|\w(i_k)| = 1$ for~$k\in\{0, \,1, \, 2\}$.
 Let~$\phi$ be a diffeomorphism defined in a neighbourhood of~$T$, 
 and suppose that $\|\d\phi - \Id\|_{L^\infty}\leq \delta$ where~$ 0 < \delta \leq 1/2$.
 Let~$S$ be the triangle of vertices~$\phi(i_0)$, $\phi(i_1)$, $\phi(i_2)$, and let~$\mathbf{z}$ be the unique affine map~$S\to\R^3$ such that
 \[
  \mathbf{z}(\phi(i_k)) = \frac{\langle \d\phi(i_k), \, \w(i_k)\rangle}
 {\abs{\langle \d\phi(i_k), \, \w(i_k)\rangle}}
 \qquad \hbox{for } k \in \{0, \, 1, \, 2\}.
 \]
 Then, there holds
 \[
  \int_S \abs{\nabla\mathbf{z}}^2 \, \d S \leq \left(1 + C\delta\right)
  \int_T \abs{\nabla\w}^2 \, \d S.
 \]
\end{lemma}
\begin{proof}
 Thanks to the assumptions that~$|\w| = 1$ on the vertices of~$T$
 and~$\|\d\phi - \Id\|_{L^\infty}\leq \delta \leq 1/2$, we see that~$\mathbf{z}$ is well-defined and
 \begin{gather*}
  \abs{\mathbf{z}(i_k) - \mathbf{z}(i_h)} \leq (1 + C\delta) \abs{\w(i_k) - \w(i_h)}, \\
  (1 - C\delta)\abs{i_k - i_h} \leq \abs{\phi(i_k) - \phi(i_h)} \leq (1 + C\delta)\abs{i_k - i_h}
 \end{gather*}
 for any~$k$, $h\in \{0, \, 1, \, 2\}$. Then, the lemma follows by a straighforward computation.
\end{proof}

\begin{proof}[Proof of Proposition~\ref{prop:core_energy}]
For the sake of convenience, we split the proof into steps.
\setcounter{step}{0}
 \begin{step}
  Recall from Section~\eqref{sect:triangulations} that~$\overline{\TT}_\eps$ is the pull back of~$\TT_\eps$
  via~$\varphi$, namely, $\overline{\TT}_\eps$ is a triangulation on~$B_\delta\subseteq\R^2$ with set of vertices
  $\varphi^{-1}(B_\delta(\bar x)\cap\TT_\eps^0)$; three vertices in~$\overline{\TT_\eps}$ 
  span a triangle in~$\overline{\TT}_\eps$ if and only if their images via~$\varphi$ do.
  We consider the minimization problem
  \begin{equation} \label{eq:core_1}
   \gamma_1(\eps,\delta) := \min_{\v \in \T(\overline{\TT}_\eps, S^1)}
   \left\{\frac{1}{2} \int_{\widehat{(B_\delta)}_\eps} \abs{\nablae\widehat{\v}}^2 \, \d S \colon 
   \quad \v = \overline{\g} \,\hbox{ on } \partial_\eps B_\delta \right\},
  \end{equation}
  where~$\overline{\g} := \varphi^*\g$. We wish to show that
  \begin{equation} \label{core_step1}
   (1 - C\delta) \, \gamma_1(\eps, \, \delta) \leq \gamma(\eps, \, \delta) 
   \leq (1 + C\delta) \, \gamma_1(\eps, \, \delta)
  \end{equation}
  Let~$\v\in\T(\overline{\TT}_\eps; \, \S^2)$ with~$\v = \overline{\g}$ on~$\partial_\eps B_\delta$
  be a competitor for Problem~\eqref{eq:core_1}. 
  The pull-back~$\varphi_*\v\in\T(\TT_\eps; \, \S^2)$ satisfies
  $\varphi_*\v = \g$ on $\partial_\eps B_\delta(\bar x)$, so~$\varphi_*\v$ is an admissible 
  competitor for Problem~\eqref{eq:core}. By noting that~$\d\varphi(0) = \Id_{\T_x M}$,
  and applying Lemma~\ref{lemma:affine_pullback} on each triangle of~$\overline{\TT}_\eps$, we obtain that
  \[
   \gamma(\eps, \delta) \leq \frac{1}{2} \int_{\widehat{B_\delta(\bar x)}_\eps} 
   \abs{\nablae(\widehat{\varphi_*\v})}^2 \, \d S
   \leq \left(\frac{1}{2} + C\delta\right) \int_{\widehat{(B_\delta)}_\eps} \abs{\nabla\widehat{\v}}^2 \, \d S.
  \]
  (In order to apply Lemma~\ref{lemma:affine_pullback}, we extend~$\varphi$ to a $3$-dimensional
  diffeormorphism~$\phi$ by setting
  $\phi(x_1, x_2, x_3) := \varphi(x_1, x_2) + x_3(\ggamma\circ\varphi)(x_1, x_2)$
  for $(x_1, x_2)\in B_\delta$ and~$x_3$ small enough.)
  Thus, by arbitrarity of~$\v$, we deduce
  \begin{equation*} 
   \gamma(\eps, \delta) \leq \left(1 + C\delta\right) \, \gamma_1(\eps, \delta).
  \end{equation*}
  A similar argument gives the other inequality in~\eqref{core_step1}.
 \end{step}
 
 \begin{step}
  Following the notation in Section~\ref{sect:triangulations}, we define a triangulation
  on~$B_{\delta/\varepsilon}\subseteq\R^2$ by setting
  \[
   \SS_\eps := \left\{\frac{1}{\varepsilon}T \colon T\in\overline{\TT}_\eps \right\} \! .
  \]
  Thanks to~\eqref{hp:quasiuniform}, there exists an~$\eps$-independent constant~$\Lambda$ such that, for any~$\eps$ 
  and any~$T\in\SS_\eps$, the affine bijection~$\phi_T$ from the reference triangle~$T_{\mathrm{ref}}\subseteq\R^2$ 
  spanned by~$(0, \, 0)$, $(1, \, 0)$, $(0, \, 1)$ to~$T$ satisfies
  \begin{equation} \label{eq:quasiuniform}
   \max\{\Lip(\phi_T), \, \Lip(\phi_T^{-1})\} \leq \Lambda.
  \end{equation} 
  
  By scaling, we deduce from~\eqref{eq:core_1} that
  \begin{equation} \label{eq:core_1_scaled}
   \gamma_1(\eps,\delta) = \min_{\v \in \T(\SS_\eps, \S^1)}
   \left\{\frac{1}{2} \int_{\widehat{(B_{\delta/\eps})}_{\varepsilon}} \abs{\nabla\widehat{\v}}^2 \, \d S \colon 
   \quad \v = \overline{\g}_{\eps} \,\hbox{ on } \partial_{\varepsilon} B_{\delta/\eps} \right\} \! ,
  \end{equation}
  where $\widehat{(B_{\delta/\varepsilon})}_{\varepsilon}$ is the union of the triangles~$T\in\SS_\eps$
  such that $T\subseteq B_{\delta/\varepsilon}$, $\partial_{\varepsilon} B_{\delta/\eps} 
  := \partial \widehat{(B_{\delta/\eps})}\cap \SS_\eps^0$ 
  and~$\overline{\g}_{\eps}\colon \bar{A}_{\delta/(2\eps), \, \delta/\eps}\to\R^2$ is given
  by~$\overline{\g}_{\eps}(x) := \overline{\g}(\eps x)$. Let us define
  \begin{equation} \label{eq:core_2}
   \gamma_2(\eps,\delta) = \min_{\v \in \T(\SS_\eps, \S^1)}
   \left\{\frac{1}{2} \int_{\widehat{(B_{\delta/\eps})}_{\varepsilon}} \abs{\nabla\widehat{\v}}^2 \, \d S \colon 
   \quad \v = \overline{\h} \,\hbox{ on } \partial_{\varepsilon} B_{\delta/\eps} \right\} \! ,
  \end{equation}
  where~$\overline{\h}(x) = \overline{\h}_{\Id}(x) := x/|x|$ 
  is, modulo rotations, the unique minimizer of~\eqref{def:min_prob_flat}.
  We claim that there exists postive numbers~$\sigma(\delta)$, $r(\delta)$, depending only on~$\delta$, such that
  \begin{equation} \label{core_step2}
   \gamma_2(\eps, \, \delta + \sigma(\delta)\delta) - r(\delta) \leq 
   \gamma_1(\eps, \, \delta) \leq \gamma_2(\eps, \, \delta - \sigma(\delta)\delta) + r(\delta),
  \end{equation}
  and~$\sigma(\delta)\to 0$, $r(\delta)\to 0$ as~$\delta\searrow 0$.  
  
  Thanks to~\eqref{eq:quasiuniform}, there exists a constant~$\lambda_0$
  (which does not depend on~$\delta$, $\eps$) such that~$\partial_\eps B_{\delta/\eps}\subseteq
  \bar{B}_{\delta/\eps}\setminus B_{\delta/\eps - \lambda_0}$.  Let~$\sigma\in (0, \, 1/2)$ be a parameter,
  possibly depending on~$\delta$ \emph{but not on~$\eps$}, to be chosen later. 
  By taking~$\eps$ small enough, we can assume without loss of generality that
  ~$(1-\sigma/2)\delta/\eps \leq \delta/\eps - \lambda_0$, so that
  \begin{equation} \label{partial_eps}
   \partial_\eps B_{\delta/\eps} \subseteq \bar{B}_{\delta/\eps}\setminus B_{(1 - \sigma/2)\delta/\eps}.
  \end{equation}
  We construct a function that interpolates between~$\overline{\g}_\eps$
  and~$\overline{\h}$ on the annulus~$A^{\eps,\delta} := B_{(1-\sigma/2)\delta/\eps}\setminus
  B_{(1-\sigma)\delta/\eps}$. Let~$\theta_{\overline{\g}}$ be a lifting for~$\overline{\g}$,
  that is, a map 
  $\theta_{\overline{\g}} \in W^{1,2}(A_{\delta/2,\delta} \setminus([0, \, +\infty)\times\{0\}); \, \R)$
  such that $\overline{\g} = \exp(i\theta_{\overline{\g}})$,
  and let~$\theta_{\overline{\g}_\eps}(x) := \theta_{\overline{\g}}(\varepsilon x)$.
  We also consider the function~$\theta_{\overline{\h}}(x) := \arctan(x_2/x_1)$, 
  which is a lifting for~$\overline{\h}$. Both $\theta_{\overline{\g}_\eps}$ and~$\theta_{\overline{\h}}$
  have a jump across the ray~$[0, \, +\infty)\times\{0\}$, and the size of both jumps is equal
  to~$\ind(\overline{\g}_\eps, \, 0) = \ind(\overline{\h}, \, 0) = 1$. Thus,
  $\theta_{\overline{\g}_\eps}-\theta_{\overline{\h}} \in W^{1,2}(A^{\eps,\delta};\,\R)$.
  By combining a scaling argument, Lemma~\ref{lem:hedgehogs_circle} and 
  the continuity of the lifting in~$W^{1/2, 2}$ (see~\cite[Remark 3]{BBM})
  we deduce that, modulo rotations, there holds
  \begin{equation} \label{eq:bordo_lifting}
   \norm{\theta_{\overline{\g}_\eps} - \theta_{\overline{\h}}}_
   {W^{1/2,2}(\partial B_{(1-\sigma/2)\delta/\eps})} =
   \norm{\theta_{\overline{\g}} - \theta_{\overline{\h}}}_
   {W^{1/2,2}(\partial B_{(1-\sigma/2)\delta})} \to 0 \quad \textrm{as } \delta\searrow 0.
  \end{equation}
  Let $u_{\eps,\delta}$ be the unique solution of
  \begin{equation} 
  \label{eq:harmonic_anulus}
   \begin{cases}
    \Delta u = 0 \quad \hbox{in } A^{\eps,\delta}\\
    u = \theta_{\overline{\g}_\eps} - \theta_{\overline{\h}} 
    \quad \hbox{on } \partial B_{(1-\sigma/2)\delta/\eps}, \quad
    u =  0 \quad \hbox{on } \partial B_{(1-\sigma)\delta/\eps},
   \end{cases}
  \end{equation}
  let~$\theta_{\eps,\delta} := u_{\eps,\delta} + \theta_{\overline{\h}}$ and 
  $\u_{\eps,\delta}:=\e^{i\theta_{\eps,\delta}}$. There holds $\u_{\eps,\delta} = \overline{\g}_\eps$ on
  $\partial B_{(1-\sigma/2)\delta/\eps}$, $\u_{\eps,\delta} = \overline{\h}$ 
  on $\partial B_{(1-\sigma)\delta/\eps}$
  By standard elliptic theory, we find that
  \begin{equation} \label{eq:rilevamentobordo}
   \lVert \nabla u_{\eps,\delta}\rVert_{L^2(A^{\eps,\delta})}^2
   \leq C \sigma^{-1} \lVert \theta_{\overline{\g}_\eps}-\theta_{\overline{\h}}\lVert^2
   _{W^{1/2,2}(\partial B_{(1-\sigma/2)\delta/\eps})}.
  \end{equation}
  The costant~$C\sigma^{-1}$ can be obtained via a scaling argument; one could also
  solve~\eqref{eq:harmonic_anulus} explicitely, passing to polar coordinates and using 
  the method of separation of variables.
  Now, since $\vert \nabla \u_{\eps,\delta}\vert = \vert \nabla \theta_{\eps,\delta}\vert$
  a.e. on $A^{\eps,\delta}$, we have
  \begin{equation*} 
   \lVert \nabla \u_{\eps,\delta}\rVert_{L^2(A^{\eps,\delta})}^2
   \le 2\left( \lVert \nabla \theta_{\overline{\h}}\rVert_{L^2(A^{\eps,\delta})}^2
   + \lVert \nabla u_{\eps,\delta}\rVert_{L^2(A^{\eps,\delta})}^2 \right)
  \end{equation*}
  Using~\eqref{eq:rilevamentobordo}, and computing explicitely the gradient of~$\theta_{\overline{\h}}$, 
  we obtain
  \begin{equation} \label{eq:harmonic_energy_bis}
   \lVert \nabla \u_{\eps,\delta}\rVert_{L^2(A^{\eps,\delta})}^2
   \le C \left(\log\left(1+\frac{\sigma/2}{1 - \sigma}\right)  + \sigma^{-1} 
   \lVert \theta_{\overline{\g}_\eps}-\theta_{\overline{\h}}\lVert^2
   _{W^{1/2,2}(\partial B_{(1-\sigma/2)\delta/\eps})}\right) \! .
  \end{equation}
  We choose
  \[
   \sigma = \sigma(\delta) := \lVert \theta_{\overline{\g}_\eps}-\theta_{\overline{\h}}\lVert
   _{W^{1/2,2}(\partial B_{(1-\sigma/2)\delta/\eps})} + \delta
  \]
  Note that, thanks to~\eqref{eq:bordo_lifting}, the right-hand side does not depend on~$\eps$
  and converges to~$0$ as~$\delta\searrow 0$; in particular, when~$\delta$ is small enough,
  we have~$\sigma\leq 1/2$. Then, using~\eqref{eq:harmonic_energy_bis}, we deduce that
  \begin{equation} \label{eq:harmonic_energy}
   \lim_{\delta\searrow 0} \, \sup_{\eps\in (0, \, \delta)} \,
   \lVert \nabla \u_{\eps,\delta}\rVert_{L^2(A^{\eps,\delta})}^2 = 0.
  \end{equation}
  Moreover, Lemma~\ref{lem:hedgehogs} combined with a scaling argument implies that
  \begin{equation} \label{eq:g_bordo}
   \lVert \nabla \overline{\g}_{\eps}\rVert_{L^2(B_{\delta/\eps}\setminus B_{\delta/\eps- \sigma(\delta)\delta/(2\eps)})}^2
   = \lVert \nabla \overline{\g}\rVert_{L^2(B_{\delta}\setminus B_{\delta - \sigma(\delta)\delta/2})}^2
   \to 0 \qquad \textrm{as } \delta\searrow 0.
  \end{equation}
  
  Let~$\v_\delta^*\in \T(\SS_\eps; \, \S^1)$ be a minimizer for Problem~\eqref{eq:core_2}
  on the ball~$B_{(\delta - \sigma(\delta)\delta)/\eps}$, i.e.
  the problem that defines~$\gamma_2(\eps, \, \delta-\sigma(\delta)\delta)$. 
  We construct the following discrete vector field:
  \begin{equation*} 
   \v_{\eps,\delta}:=
   \begin{cases}
    \v_\delta^{*}    &\hbox{in } B_{(\delta - \sigma(\delta)\delta)/\eps}\cap\SS_\eps^{0} \\
    \u_{\eps,\delta} &\hbox{in } A^{\eps,\delta}\cap\SS_\eps^{0} \\
    \overline{\g}_\eps    &\hbox{in } (B_{\delta/\eps}\setminus B_{\delta/\eps- \sigma(\delta)\delta/(2\eps)})\cap\SS_\eps^{0}.
   \end{cases}
  \end{equation*}
  Thanks to and standard interpolation arguments (see, e.g., \cite[Theorem 3.1.5]{Ciarlet}),
  we see that
  \[
   \frac{1}{2}\int_{\widehat{(B_{\delta/\eps})}_\eps} \abs{\nabla\v_{\eps,\delta}}^2 \, \d S 
   \leq \gamma_2(\eps, \, \delta - \sigma(\delta)\delta) 
   + C \lVert \nabla \u_{\eps,\delta}\rVert_{L^2(A^{\eps,\delta})}^2
   + C \lVert \nabla \overline{\g}_{\eps}\rVert_{L^2(B_{\delta/\eps}\setminus B_{\delta/\eps- \sigma(\delta)\delta/(2\eps)})}^2
  \]
  for some constant~$C$ that does not depend on~$\eps$, $\delta$ (this is possible because the
  sequence of triangulations~$\SS_\eps$ satisfies~\eqref{eq:quasiuniform}).
  Then, with the help of~\eqref{eq:harmonic_energy}, \eqref{eq:g_bordo}, we deduce that
  \[
   \frac{1}{2}\int_{\widehat{(B_{\delta/\eps})}_\eps} \abs{\nabla\v_{\eps,\delta}}^2 \, \d S 
   \leq \gamma_2(\eps, \, \delta - \sigma(\delta)\delta) + r(\delta)
  \]
  where~$r(\delta) \to 0$ and~$\sigma(\delta)\to 0$ as~$\delta\searrow 0$.
  However, due to~\eqref{partial_eps}, there holds $\v_{\eps,\delta} = \overline{\g}_\eps$ on~$\partial_\eps
  B_{\delta/\eps}$, so~$\v_{\eps,\delta}$ is an admissible competitor in 
  Problem~\eqref{eq:core_1_scaled} that defines~$\gamma_1(\eps, \delta)$.
  Thus, a comparison argument immediately yields the $\leq$-inequality in~\eqref{core_step2}.
  The other inequality is obtained via a similar argument.
 \end{step} 
 \begin{step}
  Remind that, in view of~\eqref{hp:convergence},
  the sequence~$\SS_\eps$ converges to a triangulation~$\SS$.
  We consider the analogue of Problem~\eqref{eq:core_2} on~$\SS$, that is,
  \begin{equation} \label{eq:core_3}
   \gamma_3(\eps,\delta) := \min_{\v \in \T(\SS, S^1)}
   \left\{\frac{1}{2} \int_{\widehat{(B_{\delta/\eps}})_*} \abs{\nabla\widehat{\v}}^2 \, \d S \colon 
   \quad \v = \overline{\h} \,\hbox{ on } \partial_* B_{\delta/\eps} \right\} \! .
  \end{equation}
  We have written $\widehat{(B_{\delta/\varepsilon})}_*$ to denote
  the union of the triangles~$T\in\SS$ such that $T\subseteq B_{\delta/\varepsilon}$, 
  and $\partial_* B_{\delta/\eps} := \partial (\widehat{B_{\delta/\eps}})\cap \SS^0$. We claim that
  there exists a positive number~$s(\eps, \, \delta)$ such that
  \begin{equation} \label{core_step3}
   \left(1 - s(\eps, \, \delta)\right) \, \gamma_3(\eps, \, \delta) 
   \leq \gamma_2(\eps, \, \delta) 
   \leq \left(1 + s(\eps, \, \delta)\right) \, \gamma_3(\eps, \, \delta)
  \end{equation}
  and
  \begin{equation} \label{core_step3_2}
   \lim_{\eps\searrow 0} \, s(\eps, \, \delta)\abs{\log\eps} = 0
   \qquad \textrm{for any } \delta. 
  \end{equation}
  Thanks to~\eqref{hp:convergence} and Lemma~\ref{lemma:mesh_distance}, 
  for any~$\varepsilon$, $\delta$ we find a quantity~$s_1(\eps, \, \delta)>0$ that satisfies~\eqref{core_step3_2} 
  and a piecewise affine map~$\phi_\eps\in\Iso(\SS_\eps, \, \SS_{|B_{\delta/\eps}})$ such that
  \begin{equation} \label{Lipschitz_phi_eps}
   \max\left\{\Lip(\phi_\eps), \, \Lip(\phi_\eps^{-1})\right\} 
   \leq 1 + s_1(\eps, \, \delta).
  \end{equation}
  If~$\v^*$ be a minimizer for Problem~\eqref{eq:core_3}, then we have 
  $\v^*\circ\phi_\eps = \overline{\h}\circ\phi_\eps$ on~$\partial_\eps B_{\delta/\eps}$
  so $\v^*\circ\phi_\eps$ may not be admissible competitor for Problem~\eqref{eq:core_2}.
  However, since~$|\nabla\overline{\h}(x)| \leq C/|x|$, for any~$i\in\partial_\eps B_{\delta/\eps}$ we have
  \begin{equation} \label{v_bordo}
   \abs{\v^*\circ\phi_\eps(i) - \overline{\h}(i)} \leq
   \frac{C\eps}{\delta} d(\SS_\eps, \, \SS_{|B_{\delta/\eps}}) =: s_2(\eps, \, \delta)
  \end{equation}
  and, thanks to~\eqref{hp:convergence}, $s_2(\eps, \, \delta)$ also satisfies~\eqref{core_step3_2}.
  We set~$s(\eps, \, \delta) := s_1(\eps, \, \delta) + s_2(\eps, \, \delta)$, and consider the discrete field
  \begin{equation*}
   \v_{\eps, \, \delta} :=
   \begin{cases}
    \v^{*}        &\hbox{in } (B_{\delta/\eps}\cap\SS_\eps^{0})\setminus\partial_\eps B_{\delta/\eps} \\
    \overline{\h} &\hbox{on } \partial_\eps B_{\delta/\eps}.
   \end{cases}
  \end{equation*}
  Then, $\v_{\eps, \, \delta}$ is admissible for Problem~\eqref{eq:core_2}, and a straightforward computation,
  based on~\eqref{eq:quasiuniform}, \eqref{Lipschitz_phi_eps} and~\eqref{v_bordo} yields that
  \begin{equation*}
   \begin{split}
    \gamma_2(\eps, \, \delta) \leq \frac12 \int_{\widehat{(B_{\delta/\eps})}_\eps}
    \abs{\nabla\widehat{\v}_{\eps,\delta}}^2 \, \d S
    \leq \frac{1 + s(\eps, \, \delta)}{2}
    \, \int_{\widehat{(B_{\delta/\eps})}_*} \abs{\nabla\widehat{\v}}^2 \, \d S
    = \left(1 + s(\eps, \, \delta)\right) \gamma_3(\eps, \, \delta).
   \end{split}
  \end{equation*}
  Again, the other inequality in~\eqref{core_step3} follows by a similar argument.
 \end{step}
 \begin{step}[Conclusion]
  Arguing as in~\cite[Theorem~4.1]{ADGP} (see also~\cite[Lemma~III.1]{BBH} for the continuous case),
  we find a number~$\gamma(\bar x)$ such that, for any~$\delta$, we have
  \[
   \lim_{\eps\searrow 0} \left(\gamma_3(\eps, \, \delta) - \pi\log\frac{\delta}{\eps}\right) = \gamma(\bar{x}).
  \]
  Then, the proposition follows by combining \eqref{core_step1}, \eqref{core_step2}, \eqref{core_step3}
  and~\eqref{core_step3_2}. \qedhere
 \end{step}
\end{proof}

Finally, we prepare the following refined lower bound on the Dirichlet energy 
of a unit norm vector field on an anulus.

\begin{lemma}
 \label{lemma:lower_bound_buccia}
 Let $C$ and $R^*$ be as in Lemma \ref{lemma:circle}. Then,
 for any given $\rho_1 < \rho_2 < R^*$
 and any given tangent, unit norm vector field~$\w$ defined 
 in $A_{\rho_1,\rho_2}(\bar x):=B_{\rho_2}(\bar x)\setminus B_{\rho_1}(\bar x)$ with
 $\w\in W^{1,2}_{\tan}(A_{\rho_1,\rho_2}(\bar x); \, \mathbb{S}^2)$
 and with $\ind(\w, \bar x) = d$, we have
 \begin{equation*}
  \frac12\int_{\partial B_{\rho}(\bar x)} \vert\D\w\vert^2 \d s
  \geq \frac{1}{4\pi\rho + C\rho^2}\left\vert 2\pi d -\int_{B_\rho(\bar x)}G \, \d s\right\vert^2 
  \quad \textrm{for any } \rho\in (\rho_1, \rho_2).
 \end{equation*}
\end{lemma}
\begin{proof}
 The proof is based on the fact that for a vector field $\w$ with the regularity of the statement there holds (see \eqref{spin_norm})
 \begin{equation*} \label{eq:Dv=Jv}
  \vert \D\w\vert^2 = \vert \j(\w)\vert^2. 
 \end{equation*}
 Then, the proof is similar, actually simpler, to the proof of Lemma \ref{lemma:circle}. 
 In particular, 
 there exist positive numbers~$R_*$ and~$C$ such that, for any~$x_0\in M$ and any~$0 < \rho \leq R_*$, there holds
 \begin{equation*} \label{circle01}
  \H^1(\partial B_\rho(x_0)) \leq 2\pi\rho + C\rho^2.
 \end{equation*}
 Thus, Jensen's inequality gives 
 \[
  \begin{split}
   \frac 12 \int_{\partial B_{\rho}(\bar x)} \vert\D\w\vert^2 \d s 
   = \frac12 \int_{\partial B_{\rho}(\bar x)} \vert\j(\w)\vert^2 \d s 
   \geq \frac{1}{4\pi\rho + C\rho^2} \abs{2\pi d - \int_{B_\rho}G }^2. \qedhere
  \end{split}
 \]
\end{proof}

\begin{lemma}
 \label{lemma:lower_bound_cont}
 For any $\rho_1, \rho_2 \in (0,R^*)$, for any $\bar x\in M$
 and for any $\w\in W^{1,2}_{\tan}(A_{\rho_1,\rho_2}(\bar x); \, \mathbb{S}^2)$, there holds
 \begin{equation*}
 \label{eq:lb_cont}
  \frac12\int_{A_{\rho_1,\rho_2}(\bar x)} \vert\D\w\vert^2 \d S \ge \pi\vert d\vert^2\log\frac{\rho_2}{\rho_1}
  -C\pi\vert d\vert^2 \log\frac{2\pi + C\rho_2}{2\pi +C\rho_1} - \vert d\vert,
 \end{equation*}
 where $d := \ind(\w, \bar x)$.
\end{lemma}
\begin{proof}
 Lemma \ref{lemma:lower_bound_buccia} gives that
 \begin{align*}
  \frac12\int_{\partial B_{\rho}(\bar x)} \vert\D\w\vert^2 \d s
  \ge \frac{2\pi^2d^2}{2\pi \rho + C\rho^2} - \frac{\vert d\vert}{\rho}\int_{B_\rho(\bar x)}G \, d S.
 \end{align*}
 To conclude, we integrate between $\rho_1$ and $\rho_2$ and note that thanks
 to the smoothness of the Gauss curvature $G$, by possibly reducing $R^*$, we can assume that 
 \[
  \int_{\rho_1}^{\rho_2}\left(\frac{1}{\rho}\int_{B_\rho(\bar x)}G \, \d S\right)\d \rho < 1. \qedhere
 \]
\end{proof}

\subsection{Proof of Theorem \ref{th:gamma1}}
\label{sect:proof_theorem}

\begin{proof}[Proof of~(i) --- Compactness]
The proof follows the line of \cite[Theorem 4.2]{ADGP}. Given $\ve\in \T(\TT_\varepsilon; \, \S^2)$
such that 
\[
XY_\epsi(\ve) - \KK\pi \vert \log\epsi\vert \le C,
\]
the existence of a subsequence of $\hat{\mu}_\eps$ and of the measure $\mu\in X$ with
$\sum_{i=1}^k \vert d_i\vert \le \KK$ follows from the compactness part of the zeroth order 
$\Gamma$-convergence result in Theorem \ref{th:gamma0}. 
Thus, we are left with the proof of the implication
\begin{equation}
\label{eq:implication}
\sum_{i=1}^k \vert d_i\vert = \KK \quad  \Longrightarrow \quad \vert d_i\vert =1 \ \textrm{ for any } i=1,\ldots,\KK,
\end{equation}
which implies that 
\begin{equation}
\label{eq:K=chi2}
\KK \equiv \chi(M) \mod 2.
\end{equation}
Now, fix a small $r>0$ in such a way that the balls $B_{r}(x_i),  i=1,\ldots,\KK$ are
pairwise disjoint. As usual we set $M_r :=M\setminus \bigcup_{i=1}^{\KK}B_{r}(x_i)$.
We have
\begin{align*}
XY_\eps(\ve) &= \frac{1}{2}\sum_{i=1}^k\vert \we\vert_{\Wude(B_{r}(x_i))}^2
+ \frac{1}{2}\vert \we\vert^{2}_{\Wude(M_r)}.
\end{align*}
Thanks to the localized $\Gamma$-$\liminf$ inequality in Proposition
\ref{prop:0Gamma-nu},
the first term in the equality above is bounded from below in the following way:
\begin{equation}
\label{eq:stima_basso_loc}
\frac{1}{2}\sum_{i=1}^k\vert \we\vert_{\Wude(B_{r}(x_i))}^2
\ge \pi \sum_{i=1}^k\vert d_i\vert \log\frac{r}{\eps} + C
\end{equation}
for a constant~$C$ that does not depend on~$r$.
Thus, the energy estimate \eqref{eq:energy_estimate} and the fact that $\sum_i|d_i|=\KK$ give 
\begin{equation}
\label{eq:H1_lont_difetti}
\frac{1}{2}\vert \we\vert_{\Wude(M_r)}^2\le \KK\pi |\log r| + C.
\end{equation} 
Consequently, we have that the sequence $\we$ is uniformly bounded (w.r.t. $\eps$)
in $\Wude(M_r)$ for any $r>0$. Hence, there exists a tangent  
vector field $\v$ and a subsequence such that
\begin{equation}
\label{eq:H1loc_conv}
\we \to \v \,\,\,\,\hbox{ strongly in } L^2(M;\mathbb{R}^3)\,\,\,\hbox{ and weakly in }
W^{1,2}_{\mathrm{loc}}(M\setminus \bigcup_{i=1}^{\KK} x_i; \mathbb{R}^3) \! .
\end{equation}
Moreover, thanks \ref{lemma:GL-pointwise} and to the strong $L^2$ convergence we have that 
$\vert\v\vert =1$. Passing to the limit as~$\eps\searrow 0$ in~\eqref{eq:H1_lont_difetti}, we also see that~$\W(\v)<+\infty$.
Finally, thanks to Fubini's Theorem, we can find some $r^\prime \in (r, 2r)$ 
such that for any $i=1,\ldots,k$
\[
\int_{\partial B_{r^\prime}(x_i)}\vert \nablas\we\vert^2 \d s\le C.
\]
implying, by compactness, that 
\[
\we \to \v \,\,\,\hbox{ uniformly on }  \partial B_{r^\prime}(x_i)
\]
and then we have that 
\begin{equation}
\label{eq:charge_limit}
\hbox{ind}(\v,x_i) = d_i.
\end{equation}
Thus, recalling that $\v\in W^{1,1}_{\tan}(M;\mathbb{S}^2)$ (see Lemma \ref{lem:W11})
we conclude thanks to Lemma \ref{lemma:dj} that 
$\star\d\j(\v) = 2\pi\mu - G\d S$.
Now, we prove that $k=\KK\equiv\chi(M)\mod 2$ and that $\vert d_i\vert =1$ for any $i=1,\ldots, \KK$. 
We fix $\rho_1$ and $\rho_2$ in the interval $(0,R^*]$ ($R^*$ is as in Lemma \ref{lemma:circle})
with  $\rho_1 < \rho_2$ and such that the geodesic balls $B_{\rho_2}(x_i)$ are pairwise disjoint. 
We have
\begin{align*}
XY_\eps(\ve) &= \frac{1}{2}\sum_{i=1}^k\vert \we\vert_{\Wude(B_{\rho_1}(x_i))}^2
+ \frac{1}{2}\vert \we\vert^{2}_{\Wude(M\setminus \bigcup_{i=1}^{k}B_{\rho_1}(x_i))} \\
&\ge \frac{1}{2}\sum_{i=1}^k\vert \we\vert_{\Wude(B_{\rho_1}(x_i))}^2
+ \frac{1}{2}\sum_{i=1}^k\vert\we\vert_{\Wude(A^i_{\rho_1,\rho_2})}^2.
\end{align*}
Now, thanks to the localized $\Gamma$-$\liminf$ inequality in Proposition
\ref{prop:0Gamma-nu},
the first term in the inequality above is, as in \eqref{eq:stima_basso_loc}, bounded from below
\begin{equation}
\frac{1}{2}\sum_{i=1}^k\vert \we\vert_{\Wude(B_{\rho_1}(x_i))}^2
\ge \pi \sum_{i=1}^k\vert d_i\vert \log\frac{\rho_1}{\eps} + C.
\end{equation}
Thus, we get (recall that $\sum_{i=1}^k \vert d_i\vert= \KK$)
\begin{equation}
\label{eq:est_anulus}
\frac{1}{2}\sum_{i=1}^k\vert \we\vert_{\Wude(A^i_{\rho_1,\rho_2})}^2 \le
XY_\eps(\ve)  - \pi \KK \vert \log\eps\vert \le C_{\KK}.
\end{equation}
As before,  we obtain that the sequence $\we$ is uniformly bounded (w.r.t. $\eps$)
 $\Wude(\bigcup_{i=1}^k A^i_{\rho_1,\rho_2})$. Let $\v$ the unit
 norm vector field identified above. 
By semicontinuity of the norm we get
\begin{equation}
\label{eq:semicont}
\liminf_{\eps\to 0}\frac{1}{2}\sum_{i=1}^k\vert \we\vert_{\Wude(A^i_{\rho_1,\rho_2})}^2
\ge 
\frac{1}{2}\sum_{i=1}^k\int_{A^i_{\rho_1,\rho_2}}\vert\D\v\vert^2 \d S + \frac12
\sum_{i=1}^k\int_{A^i_{\rho_1,\rho_2}}\vert \d\ggamma[\v]\vert^2\d S.
\end{equation}
Thus, 
\begin{align*}
XY_\eps(\ve)& \ge \pi \sum_{i=1}^k\vert d_i\vert \log\frac{\rho_1}{\eps} 
+ 
\frac{1}{2}\sum_{i=1}^k\int_{A^i_{\rho_1,\rho_2}}\vert\D\v\vert^2 \d S 
 + \frac12
\sum_{i=1}^k\int_{A^i_{\rho_1,\rho_2}}\vert \d\ggamma[\v]\vert^2\d S + C. 
\end{align*} 
 Moreover, Lemma \ref{lemma:lower_bound_cont}, together with $\sum_{i=1}^k\vert d_i\vert = \KK$,
 gives
 \begin{align*}
XY_\eps(\ve)& \ge \pi \KK \vert\log\eps\vert + 
\pi \sum_{i=1}^k\big(\vert d_i\vert^2- \vert d_i\vert \big)\log\frac{\rho_2}{\rho_1}
- C\pi \sum_{i=1}^k\vert d_i\vert^2\log\frac{2\pi + C\rho_2}{2\pi + C\rho_1} + C.
\end{align*}
Thus, letting $\rho_1\to 0$ and using the energy bound, we get that 
\begin{align*}
\lim_{\rho_1\to 0}\pi \sum_{i=1}^k\big(\vert d_i\vert^2- \vert d_i\vert \big)\log\frac{\rho_2}{\rho_1}
\le C_{\KK} + C + C\pi \sum_{i=1}^k\vert d_i\vert^2\log\frac{2\pi + C\rho_2}{2\pi},
\end{align*}
Then, since the last term is bounded by a constant depending on $\KK$ and on $R^*$ we get that
$\vert d_i\vert =1$ for any $i=1,\ldots,\KK$ and thus $\v\in \mathcal{V}_{\KK}$.
\end{proof}
\begin{proof}
[Proof of~(ii) --- $\Gamma$-liminf] 
Let $\ve$ be a sequence in $\T(\TT_\varepsilon; \, \S^2)$ satisfying 
the energy estimate (see \eqref{eq:energy_estimate})
\[
XY_\epsi(\ve) - \KK\pi \vert \log\epsi\vert \le C,
\]
with $\KK\equiv\chi(M)\mod 2$. 
Thanks to the compactness proved above, there exists a subsequence, a measure 
$\mu \in X$ with $\vert d_i\vert =1$ for all $i=1,\ldots,\KK$ and a tangent unit norm vector 
field $\v\in \mathcal{V}_{\KK}$ such that 
\begin{equation}
\label{eq:prova_liminf1}
\hat{\mu}_\eps(\ve) \xrightarrow{\fflat}\star \d\j(\v)= 2\pi\mu - G \d S, 
\end{equation}
and
\begin{equation}
\label{eq:H1loc_conv_bis}
\we \to \v \,\,\,\,\hbox{ strongly in } L^2(M;\mathbb{R}^3)\,\,\,\hbox{ and weakly in }
W^{1,2}_{\mathrm{loc}}(M\setminus \bigcup_{i=1}^{\KK} x_i; \mathbb{R}^3).
\end{equation}
%
In particular, the semicontinuity of the norm gives that, for any $r>0$ such 
that the geodesic balls $B_{2r}(x_i)$ are pairwise disjoint, there holds
\begin{equation}
\label{eq:semicontH1}
\liminf_{\eps\to 0}\frac 12 \vert \we\vert^2_{\Wude(M_r)}\ge \frac 12 \int_{M_r}\vert\D\v\vert^2 
+ \vert \d\ggamma[\v]\vert^2 \d S.
\end{equation}
In what follows, we will assume that~$d_i = 1$ for any~$i$. 
The case~$d_i = -1$ for some index~$i$ can be treated in a similar way, with straightforward modifications of the proof.



For any $i$ we consider the minimum problem \ref{def:min_prob} with $\bar x = x_i$ 
and $t\le r$. As in \cite{ADGP}, the following property holds.
For any fixed $\sigma>0$, there
exists a positive $\omega=\omega(\sigma)$ (independent of $t$ and of $i$) such that if
\[
 \sigma < d(\we, \mathcal{H}(t, \, x_i)) := \inf\left\{ \lVert\nablas \w_\eps - \nablas \v
 \rVert_{L^2(A_{t/2, t}(x_i))}: \v\in \mathcal{H}(t, \, x_i)  \right\}.
\]
then
\begin{equation}
\label{eq:omega_delta}
\liminf_{\eps\to 0}\frac{1}{2}\vert \we\vert_{\Wude(A_{t/2,t}(x_i))}^2
\ge \omega(\sigma) + \eta(t, \, x_i) \qquad \textrm{for any } i= 1, \ldots, \KK,
\end{equation}
where~$\eta(t, \, x_i)$ is the minimum value for~\ref{def:min_prob}, namely
 \begin{equation}
 \label{eq:eta}
 \eta(t, \, x_i):=\min_{\w \in W^{1,2}_{\tan}(A_{t/2,t}(x_i); \, \mathbb{S}^2)}\left\{\frac{1}{2}\int_{A_{t/2,t}(x_i)}\vert\D\w\vert^2 + \vert \d \ggamma [\w]\vert^2\d S,\,\,\,\,\hbox{ind}(\w, x_i) = 1\right\}.
 \end{equation}
By Lemma~\ref{lem:energia_hedgehogs}, if~$t\leq r$ is sufficiently small we have
\begin{equation} \label{eq:eta_lim}
 \eta(t, \, x_i) \geq \frac{\pi}{2}\log 2.
\end{equation}
Then, we fix $L\in \mathbb{N}$ in such a way that
\begin{equation}
\label{eq:L}
L\omega(\sigma)\ge \mathbb{W}(\v) + \sum_{i=1}^{\KK}\gamma(x_i) -\KK( \pi\log r  + C).
\end{equation}
where $C$ is the constant that appears in the localized liminf inequality in Proposition~\ref{prop:0Gamma-nu}.
This is clearly possible since $\mathbb{W}(\v)<+\infty$.
Moreover, we set 
$\lambda := 2^{1/(2\KK)} \in [1, \, 2]$. 
For $l=1,\ldots, L$, and $i=1,\ldots, \KK$ 
we set $A_{l}^i:= B_{\lambda^{1-l}r}(x_i)\setminus B_{\lambda^{-l}r}(x_i)$.

\smallskip
We have to face the two following situations.
\setcounter{case}{0}
\begin{case}
For a $\eps$ sufficiently small and for any $l=1,\ldots, L$, there exists
one $i$ such that $d(\we, \mathcal{H}(\lambda^{1-l}r, \, x_i))\ge \sigma$. Thus, thanks to 
\eqref{eq:omega_delta},  \eqref{eq:eta_lim}, to the localized $\liminf$ inequality in Proposition \ref{prop:0Gamma-nu} and 
to \eqref{eq:L}, we get (recall that $\KK\ge 1$)
\begin{equation}
\label{eq:first_case}
 \begin{split}
  XY_\eps(\ve) &\ge \sum_{i=1}^{\KK}\frac{1}{2}\vert \we\vert^2_{\Wude(B_{\lambda^{-L}r}(x_i))}
  + \sum_{l=1}^{L}\sum_{i=1}^{\KK}\frac12 \vert \we\vert^2_{\Wude(A^i_{l})} \\
  &\ge \pi\KK\log\frac{\lambda^{-L}r}{\eps} - C \KK + \frac{\pi}{2}L\log 2 + L \omega(\sigma) 
  + \mathrm{o}_{\eps\to 0}(1) \\
  &\ge \pi\KK \vert\log\eps\vert + \sum_{i=1}^{\KK}\gamma(x_i) + \mathbb{W}(\v) + \mathrm{o}_{\eps\to 0}(1). 
 \end{split}
 \end{equation}
\end{case}
\begin{case}
The second possibility we have to face is that (up to a subsequence) there exists
a $\bar{l}\in \left\{1,\ldots, L\right\}$ such that for every $i$ there holds
\[
d(\we, \mathcal{H}(\lambda^{1-\bar{l}}r, \, x_i))\le \sigma
\]
Now, for any $i=1,\ldots,\KK$, let $\we^{i}$ a vector field in $\mathcal{H}$ such that
\[
 d(\we, \mathcal{H}(\lambda^{1-\bar{l}}r, \, x_i)) 
 = \lVert \nablas \we - \nablas \we^{i}\lVert_{L^2(A^i_{l})}.
\]
Note that by construction, $\we^i$ is a tangent vector field with unit 
norm (defined in $A_{\bar l}^i$)
and such that $\hbox{ind}(\we^i,x_i)= 1$. Thus, by mimicking 
the cut off argument in \cite{ADGP},
we can construct a discrete vector field $\tilde{\v}_\eps \in \T(\TT_\varepsilon; \, \S^2)$
for which its corresponding $\tilde{\w}_\eps$ (see \eqref{eq:def_w}) verifies for any $i=1,\ldots,\KK$
$\tilde{\w}_{\eps} =\we^i $ on $\partial B_{\lambda^{1-\bar{l}}r}(x_i)$ and 
\begin{equation}
\label{eq:cut_off_liminf}
\frac{1}{2}\vert \we\vert^2_{\Wude(B_{\lambda^{1-\bar{l}}r}(x_i))}\ge \frac{1}{2}\vert \tilde{\w}_{\eps}
\vert_{\Wude(B_{\lambda^{1-\bar{l}}r}(x_i))}^2 + \mathfrak{r}(\eps,\sigma), 
\end{equation}
with $\lim_{\sigma\searrow 0} \lim_{\eps\searrow 0} \, \mathfrak{r}(\eps,\sigma) = 0$. 
To construct such a vector field, one can
map the lattice $\TT^\eps_{0}$ on $\mathbb{R}^2$ with  geodesic normal coordinates and then use
the construction of \cite{ADGP} (see also \cite{AlicandroPonsiglione} for an analogous construction
in the framework of the two dimensional Ginzburg Landau functional).

\smallskip
We are ready to conclude the proof of the $\Gamma-\liminf$. 
On the one hand, the construction above and Proposition \ref{prop:core_energy} give
that (we set $\bar r:=\lambda^{1-\bar{l}}r$ and we recall that $\gamma_{x_i}(\eps,\bar r)$ is the value 
of the minimum problem in \eqref{eq:core})
\begin{equation}
\label{eq:liminf1}
 \begin{split}
  \frac{1}{2}\sum_{i=1}^{\KK}\vert \we\vert_{\Wude(B_{\bar r}(x_i))}^2 &\ge 
  \frac{1}{2}\sum_{i=1}^{\KK}\vert \bar{\w}_{\eps}
  \vert_{\Wude(B_{\lambda^{1-\bar{l}}r}(x_i))}^2 + \mathfrak{r}(\eps,\sigma)
  \ge \sum_{i=1}^{\KK}\gamma_{x_i}(\eps,\bar r) + \mathfrak{r}(\eps,\sigma) \\
  &= \sum_{i=1}^{\KK}\gamma(x_i) + \pi\KK\abs{\log\frac{\eps}{\bar r}} + \mathfrak{r}(\eps,\sigma) + \mathrm{o}_{\eps\to 0}(1)+\mathrm{o}_{\bar r\to 0}(1).
 \end{split}
\end{equation}
On the other hand, \eqref{eq:semicontH1} and \eqref{eq:renormalized_ene_def} give
\begin{equation}
\label{eq:liminf2}
 \begin{split}
  \frac 12\vert \we \vert_{\Wude(M_{\bar r})}^2
  &\ge \frac{1}{2}\int_{M_{\bar r}}\vert\D\v\vert^2 \d S + \frac12
  \int_{M_{\bar r}}\vert \d\ggamma[\v]\vert^2\d S + \mathrm{o}_{\eps\to 0}(1) \\
  &\ge \pi\KK \vert\log\bar r\vert + \mathbb{W}(\v) + \mathrm{o}_{\eps\to 0}(1) + \mathrm{o}_{\bar r\to 0}(1).
 \end{split}
\end{equation}
As a result, combining \eqref{eq:liminf1} and \eqref{eq:liminf2} we get
\begin{equation*}
 \begin{split}
  XY_\eps(\ve) &= \frac{1}{2}\sum_{i=1}^{\KK}\vert \we\vert_{\Wude(B_{\bar r}(x_i))}^2 + \frac 12\vert \we \vert_{\Wude(M_{\bar r})}^2 \\
  &\ge  \pi\KK \vert\log\bar r\vert + \mathbb{W}(\v) + \sum_{i=1}^{\KK}\gamma(x_i)+ \pi\KK\abs{\log\frac{\eps}{\bar r}} 
  + \mathrm{o}_{\eps\to 0}(1) + \mathfrak{r}(\eps,\sigma) + \mathrm{o}_{\bar r\to 0}(1) \\
  &= \pi\KK \vert\log \eps\vert  + \mathbb{W}(\v) + \sum_{i=1}^{\KK}\gamma(x_i) +\mathrm{o}_{\eps\to 0}(1) + \mathfrak{r}(\eps,\sigma) + \mathrm{o}_{\bar r\to 0}(1).
 \end{split}
\end{equation*}
Thus, sending $\eps\to 0, \sigma\to 0, r\to 0$ we get the $\Gamma-\liminf$ inequality \eqref{eq:gamma_liminf1}. \qedhere
\end{case}
\end{proof}

\begin{proof}[Proof of~(iii) --- $\Gamma$-limsup] 
Given $\v\in \mathcal{V}_{\KK}$, 
the goal is to construct a sequence $\ve$ such that $\we\to \v$ weakly in $L^2(M;\mathbb{R}^3)$,
such that $\hat{\mu}_{\eps}(\ve) \xrightarrow{\fflat} \star \d \j(\v)$ and such that the limsup inequality \eqref{eq:gamma_limsup1} holds. 
The recovery sequence $\ve$ is constructed as in \cite{ADGP}. For the sake of clarity, we highlight the main points.
First of all we suppose that $\v$ is smooth, otherwise (as in \cite{ADGP})
we can approximate it with a smooth vector field in the $W^{1,2}$ norm (see \cite{SU} and \cite{AGM-VMO}). 
Then, we recall \eqref{eq:conc_energy_intr} that gives that, for a fixed $\rho>0$, 
\[
\lim_{j\to +\infty}\frac 12\int_{A^{i}_{2^{-j-1}\rho, 2^{-j}\rho}} \vert\D\v\vert^2 + \vert \d \ggamma[\v]\vert^2 \d S = \pi \log 2.
\]
Thus,
  Lemma~\ref{lem:hedgehogsV}
  gives that we can find a matrix~$R = R(j)\in\mathrm{SO}(2)$ such that for any $i$
  \begin{equation} \label{eq:bordo_H1/2j}
  \lim_{j\to +\infty}\norm{\nablas \v - \nablas \h_{R(j)}}_{L^2(A_{2^{-j-1}\sigma, 2^{-j}\sigma}^i)}
  =
    0.
  \end{equation}
  Now, we construct a tangent vector field on $A_{2^{-j-1}\sigma, 2^{-j}\sigma}^i$ that interpolates between $\h_{R(j)}$ and $\v$ on 
 $\partial B_{2^{-j-1}\rho}(x_i)$ and $\partial B_{2^{-j}\rho}(x_i)$, respectively. Let $\psi:[\frac{1}{2},1]\to \mathbb{R}$ be a smooth 
 cut off function such that $\psi(1/2)=0$ and $\psi(1)=1$. Let $\theta_{\overline{\v}}$ and $\theta_{\overline{\h}_{R(j)}}$
 be the liftings of the vector fields $\varphi^{*}\v$ and $\varphi^{*}\h_{R(j)}$. 
 We set
  $u^i_{j}:= \psi(2^j\rho\vert x \vert)\theta_{\overline{\v}} + (1-\psi(2^{j}\rho\vert x\vert))\theta_{\overline{\h}_{R(j)}},$
 and then we set $\overline{\u}^{i}_{j}:= \e^{\iota u^i_{j}}$, $\iota$ being the immaginary unit.
 Finally, we map $\overline{\u}^{i}_{j}$ back on $M$ using $\varphi_{*}$, namely we set
 $\u_{j}^i := \varphi_{*}\overline{\u}^{i}_{j}$.
 Using \eqref{eq:bordo_H1/2j} it is not difficult to see that for any $i=1,\ldots,$ there holds that 
 \begin{equation}
 \label{eq:conv_interpol}
 \lim_{j\to +\infty}\int_{A_{2^{-j-1}\sigma, 2^{-j}\sigma}^i}\vert \nablas \u_{j}^i\vert^2 \d S = \pi\log 2.
 \end{equation}
 
  We consider the following sequence of vector fields
\begin{equation}
\label{eq:rec_sec1}
 \v_{j}: = 
 \begin{cases}
  \v &\hbox{in } M\setminus \bigcup_{i=1}^{\KK}B_{2^{-j-1}\rho}(x_i)\\
  \tilde{\u}^{i}_{j} &\hbox{ in } A^{i}_{2^{-j-1}\rho, 2^{-j}\rho}.
 \end{cases}
\end{equation}
Now, for any $i\in \left\{1,\ldots,\KK \right\}$ we consider the
discrete vector field $\v_{\eps}^i:=\varphi_{*}\overline{\v}_{\eps}^i$
with $\overline{\v}_{\eps}^i$ be the minimizer of \eqref{eq:core_2} in 
$B_{2^{-j-1}\rho}$ with $\overline{\h_{R(j)}}$ as boundary condition 
on $\partial_\eps B_{2^{-j-1}\rho}$. 
Finally, we let the recovery sequence be the sequence of discrete vector fields
that coincides with $\v_{j}$ on the nodes of $M\setminus \bigcup_{i=1}^{\KK}B_{2^{-j-1}\rho}(x_i)$
and with $\v_{\eps}^i$ on the nodes of $B_{2^{-j-1}\rho}$. 
More precisely, we set 
\begin{equation*}
 \v_{\eps,j} := 
 \begin{cases}
  \v_{j}  &\hbox{in } \Big(M\setminus \bigcup_{i=1}^{\KK}B_{2^{-j-1}\rho}(x_i)\Big)\cap \TT^{0}_{\eps},\\
  \ve^{i} &\hbox{in } B_{2^{-j-1}\rho}(x_i)\cap \TT^{0}_{\eps}\,\,\,\hbox{ for } \,\,\,i=1,\ldots,\KK.
 \end{cases}
\end{equation*}
Now, let $\w_{\eps,j}:=\hat{\v}_{\eps, j}\circ\widehat{P}_\eps^{-1}$ as in \eqref{eq:def_w}.
By a diagonal argument, we have that there exists sequence $j(\eps)\xrightarrow{\eps\to 0}0$ such that 
\[
\w_{\eps,j(\eps)}\to \v \,\hbox{ strongly in } L^2(M;\mathbb{R}^3) \,\hbox{ and weakly in }
W^{1,2}_{\mathrm{loc}}(M\setminus\bigcup_{i=1}^{\KK}x_i;\mathbb{R}^3).
\]
Moreover, $\hat{\mu}_\eps(\v_{\eps,j(\eps)})\xrightarrow{\fflat} \star\d\j(\v)$. 
Finally, along the same sequence $j(\eps)$ $\v_{\eps, j(\eps)}$ is indeed a recovery sequence. In fact the following holds
\begin{align*}
XY_\eps(\v_{\eps,j}) - \pi\KK\vert \log\eps\vert &= 
\frac{1}{2}\vert \w_{\eps,j}\vert^2_{\Wude(M)}- \pi\KK\vert \log\eps\vert \\
&=\frac 12\vert \v\vert^2_{\Wude(M\setminus \bigcup_{i=1}^{\KK}B_{2^{-j}\rho}(x_i))}
-\pi\KK\vert \log 2^{-j}\rho\vert \\
&+ \frac 12\sum_{i=1}^{\KK}\vert \w_{\eps,j}\vert^2_{\Wude(A^{i}_{2^{-j-1}\rho, 2^{-j}\rho})}
-\pi\KK\log 2 \\
&+ \frac 12\sum_{i=1}^{\KK}\vert \w_{\eps,j}\vert^2_{\Wude( B_{2^{-j-1}\rho}(x_i))}
-\pi\KK\abs{\log\frac{\eps}{2^{-j-1}\rho}} \! .
\end{align*}
Thus, using \eqref{eq:conv_interpol},  \eqref{eq:rec_sec1}, and Proposition \ref{prop:core_energy} we conclude. 
\end{proof}

\section*{Acknowledgements}
G.C.'s  research has received funding from the European Research Council under the European Union's Seventh Framework Programme (FP7/2007-2013) / ERC grant agreement n° 291053.
A.S. gratefully acknowledges the financial support of the FP7-IDEAS-ERC-StG \#256872
(EntroPhase). Part of this work has been done while A.S. was visiting the Erwin Schr\"odinger Institute in Vienna for the ESI Thematic Program 
'Nonlinear Flows'. 
G.C. and A.S. acknowledge the partial support of the GNAMPA (Gruppo Nazionale per l'Analisi Matematica, la Probabilit\`a 
e le loro Applicazioni) group of INdAM (Istituto Na\-zio\-na\-le di Alta Matematica). 
The authors would like to thank John Ball and Epifanio Giovanni Virga,
for inspiring conversations and interesting remarks.

\bibliographystyle{plain}
\bibliography{renormalised_energy}
\end{document}